\documentclass[11pt,a4paper]{article}
\title{Heavy-traffic analysis through uniform acceleration \\of queues with diminishing populations}
\date{\today}
\author{Gianmarco Bet, Remco van der Hofstad, and Johan S.H. van Leeuwaarden} 

\usepackage[usenames,dvipsnames]{xcolor}
\usepackage[shadow]{todonotes} 

\usepackage{amsfonts}
\usepackage{amsmath}\numberwithin{equation}{section}
\usepackage{amsthm}

\usepackage{geometry} 
\usepackage{enumitem} 

\usepackage{dsfont}

\usepackage{amssymb} 

\usepackage{pbox} 

\usepackage{pgfplots} 
\usepgfplotslibrary{groupplots} 
\usepackage{standalone}

\usepackage{mathtools} 

\newtheorem{lemma}{Lemma}
\newtheorem{theorem}{Theorem}
\newtheorem*{theorem*}{Theorem}
\newtheorem{corollary}{Corollary}

\newtheorem{definition}{Definition}

\newcommand{\Exp}{\mathrm{Exp}}
\newcommand{\sss}{\scriptscriptstyle}
\newcommand{\nin}{\notin}

\newcommand{\Var}{\textrm{Var}}
\newcommand{\Lou}{\cite{louchard1994large}}
\newcommand{\Ward}{\cite{honnappa2012delta}}
\newcommand{\nnl}{\notag\\}
\newcommand{\sr}{\stackrel}
\newcommand{\E}{\mathbb E}
\newcommand{\fT}{f_{\sss T}}
\newcommand{\FT}{F_{\sss T}}
\newcommand{\OP}{O_{\mathbb P}}

\newcommand{\gr}{}
\newcommand{\bl}{}
\newcommand{\red}{}%

\newcommand{\eqn}[1]{\begin{equation}#1\end{equation}}
\newcommand{\e}{\mathrm{e}}
\newcommand{\prob}{{\mathbb P}}
\newcommand{\expec}{{\mathbb E}}
\def\1{{\mathchoice {1\mskip-4mu\mathrm l}      
{1\mskip-4mu\mathrm l}
{1\mskip-4.5mu\mathrm l} {1\mskip-5mu\mathrm l}}}
\newcommand{\indic}[1]{\1_{\{#1\}}}

\begin{document}
\maketitle

\begin{abstract}
We consider a single server queue  that serves a finite population of $n$ customers that will enter the queue (require service) only once{\gr, also known as the $\Delta_{(i)}/G/1$ queue}.  This paper presents a method for analyzing heavy-traffic behavior by using uniform acceleration, which simultaneously lets $n$ and the service rate grow large, while the initial resource utilization approaches one. 
A key feature of the model is that, as time progresses, more customers have joined the queue, and fewer customers can potentially join. This {\it diminishing population} gives rise to a class of reflected stochastic processes that vanish over time, and hence do not have a stationary distribution.  We establish that, {\gr when the arrival times are exponentially distributed,} by suitably rescaling space and time, the queue length process converges to a Brownian motion with parabolic drift, a stochastic-process limit that captures the effect of a diminishing population by a negative quadratic drift. {\gr When the arrival times are generally distributed, our techniques provide information on the typical queue length and the first busy period.} 

\end{abstract}


\section{Introduction} \label{sec:Introduction}
The queueing literature is to a large extent built on the assumption that arrivals are governed by some renewal process, a mathematically convenient assumption as it allows the full use of probabilistic tools based on regenerative processes and ergodic theory. {\bl  This paper, however, considers a \emph{transitory} queueing model, {\gr known in the literature as the $\Delta_{(i)}/G/1$ queue, }which operates only a finite amount of time and cannot be viewed as a standard regenerative process.} {\gr The $\Delta_{(i)}/G/1$ queue} assumes a finite population of customers entering the queue only once. As time progresses, more customers have joined the queue, and fewer customers can potentially join. This modelling assumption of  a {\it diminishing population of customers} gives rise to a class of reflected stochastic processes {\bl that lack a stationary distribution, and instead display relevant behavior only during a \emph{finite time window}. Therefore, only the time-dependent behavior is of interest.}

When we assume the arrival times of the customers to be sampled independently from an identical distribution, the arrival times
are the order statistics of the sample, and the interarrival times are differences of order statistics. Further assuming a single server, and generally distributed independent service times, this model was coined the $\Delta_{(i)}/G/1$ queue in \cite{honnappa2012delta}, in which fluid and
diffusion limits for the queue length process were established. In \cite{honnappa2014transitory} a wider class of transitory queues was introduced, with the $\Delta_{(i)}/G/1$ queue still as the prime example, and stochastic-process limits were established for large population sizes. 

We will introduce a heavy-traffic regime for {\gr the $\Delta_{(i)}/G/1$ queue}, leading to stochastic-process limits and heavy-traffic approximations.  Considering queueing processes in their critical regimes typically leads to a reduction in complexity, since the complicated processes can often be shown to  converge to much simpler limiting stochastic processes. Stochastic-process limits have been studied for single-server queues that have a time-varying arrival rate. Newell \cite{newell1968queues,newell1982applications} pioneered this direction by deriving diffusion approximations, see also \cite{keller1982time,louchard1994large}. Rigorous results in terms of stochastic-process limits were obtained in \cite{mandelbaum1995strong} (building on \cite{massey1982non,massey1985asymptotic}). Here, stochastic-process limits were established as refinements to deterministic ODE limits for the time-dependent $M/M/1$ queue, known as the $M_t/M_t/1$ queue. {\bl  See also \cite{yang1997asymptotic} for a systematic treatment of the $M_t/G/1$ queue}. The technique used in \cite{mandelbaum1995strong} to develop
Functional Law of Large Numbers (FLLN) and Functional Central Limit Theorem (FCLT)
results uses strong approximation and what is known as the {\it uniform accelaration} (UA) technique. UA relies on the assumption that the relevant time scale for changes in the queue length process is of the order {\bl $O(1/\varepsilon)$} for some $\varepsilon>0$. Accelerating the process in a uniform manner by scaling the arrival and service rates by $\varepsilon$ then reveals the dominant model behavior as $\varepsilon\to 0$. While in \cite{honnappa2012delta, honnappa2014transitory} the arrival and service rates are scaled in a similar manner, {\bl the time scale considered is of the order $O(1)$. In particular, the time of the process is not scaled}. 
The UA technique has been extensively applied to non-stationary queueing systems with non-homogeneous Poisson input, but it remained unclear whether it is also useful for transitory {\bl  or time vanishing} queueing models as considered in \cite{honnappa2014transitory} {\bl and in the present paper}. {\bl  We will show how the key idea behind the UA technique can be applied to these models}. We shall now explain our approach in terms of the easiest setting, in which the identical distribution that generates the arrival times is exponential.

Assume a finite population of $n$ customers, with $n$ {\bl  very} large, and where each customer has an independent exponential clock with mean $1/\lambda$. Customers join the queue when their clocks ring. The initial arrival process ({\bl  close to time zero}) is then  roughly Poisson  with rate $n\lambda$. However, as time progresses, the arrival intensity decreases, due to those customers that have left the system. {\bl  Thus, the arrival process is a Poisson process that is thinned according to some time-dependent rule.} Denote the mean service time by $1/\mu$. In order to create heavy-traffic conditions we let the population size $n$ grow to infinity, while at the same time making sure that the (initial) traffic intensity {\bl  $\rho_n=n\lambda/\mu$} is close to one. The system can initially be underloaded (when $n\lambda<\mu$), overloaded (when $n\lambda>\mu$), or critically-loaded (when $n\lambda=\mu$). In case $n\lambda>\mu$, the queue {\bl  initially shows a roughly linear increase} and therefore the correct scaling of the queue length to obtain meaningful limits is $n$ (FLLN) for a first order approximation and $n^{1/2}$ (FCLT) for second order approximations. {\bl  In particular, no time scaling is needed to obtain these approximations; these are the most relevant approximations obtained in \cite{honnappa2012delta}.}

We focus on the critically-loaded regime, and we combine this with uniform acceleration through the population size $n$. {\bl  In the spirit of UA, we let the arrival and service rates scale with $n$ while also rescaling time so as to observe the queue length process at a time scale of order $ O(1/n^{\gamma})$ for some $\gamma>0$.}
Denote the density of the arrival distribution as $f_{\sss T}(t)$ (with $f_{\sss T}(t)=\lambda \e^{-\lambda t}$, $t\geq 0$ as an important special case), and denote the
i.i.d.~service requirements of consecutive customers by $S_1, S_2,\ldots$ with generic random variable $S$. Assuming a service rate of $n$, the service times of consecutive customers are then $D_1=S_1/n, D_2= S_2/n,\ldots$  with generic random variable $D$. {\gr For the sake of clarity, we now first give a simplified version of the more general heavy-traffic condition \eqref{crit2}.} The heavy-traffic regime we consider is given by the condition
\begin{equation}\label{crit}%
\rho_n:= n f_{\sss T}(0)\mathbb{E}[D]=f_{\sss T}(0)\mathbb{E}[S]=1, \quad {\rm for }\ n \ {\rm large}.
\end{equation}%
{\bl For} the exponential case $f_{\sss T}(0)=\lambda$, {\bl  and the condition reads $\rho_n=\lambda\E[S]=1$ can be interpreted as, for times close to zero,} the expected number of newly arriving customer during one service time is roughly one. For general service times, the condition can be understood by {\bl  interpreting $\fT(t)$ as the \emph{instantaneous arrival rate} in $t$. Since we consider time scales of the order $O(1/n^{\gamma})$, 
only the mass in zero $f_{\sss T}(0)$ matters for describing the new arrivals. }

In what follows we actually consider a slightly {\bl  more general} definition of the random variables $(D_i)_{i=1}^{\infty}$ and this leads to {\bl the more precise critical scaling }
\begin{equation}\label{crit2}%
\rho_n=1 + \beta n^{-1/3}.
\end{equation}%
The additional term $\beta n^{-1/3}$ arises from detailed calculations, but can be interpreted as the factor that describes the onset of the heavy-traffic period: when $\beta>0$ (resp. $<0$) the queue is initially slightly overloaded (resp. underloaded).

It is clear that {\gr the $\Delta_{(i)}/G/1$ queue }is strongly influenced by the service-time distribution.
In particular, the heavy-traffic behavior is crucially different depending on whether the second moment of the service time distribution is finite or not. In this paper we assume throughout that  $\mathbb{E}[S^2]<\infty$, in which case the queueing process is in the domain of attraction of Brownian motion. Indeed, {\bl  we will take our process and} scale space and time so that the stochastic-process limit turns out to be a (reflected) Brownian motion {\bl with parabolic drift}. The latter process is defined as $(W_{t})_{t\geq 0}=(at+bt^2/2+cB_{t})_{t\geq 0}$ with $(B_{t})_{t\geq 0}$ a standard Brownian motion, and $a,b,c$ constants. The constant $b$ is negative, so that eventually the free process $(W_{t})_{t\geq 0}$ drifts to minus infinity according to $bt^2/2$, causing the reflected process to be essentially stuck at zero. This effect is due to the {\it diminishing population} effect. One could interpret the quadratic term in the limit as the (cumulative) effect of the customers already served not being able to join the queue again.
The stochastic-process limit provides insight into the macroscopic behavior (for $n$ large) of the transitory queueing process, and the different phenomena occurring at different space-time scales. It also gives insight into the orders of the average queue lengths, the probability of large queue lengths occurring, and the time scales of busy periods.

\subsection{Comparison with known results}\label{sec:comparison_known_results}
To create the right circumstances for non-degenerate limiting behavior in heavy traffic, we work not only under the critical-load assumption \eqref{crit}, but also under the additional assumption that the maximum of the density $f_{\sss T}(t)$, $t\geq 0$ is assumed in $t=0$. {\bl  Indeed, $f_{\sss T}(t)$ should be interpreted as an \emph{instantaneous arrival rate}, and then $\max_t f_{\sss T}(t)\E[S]= f_{\sss T}(t_{\mathrm{max}})\E[S]=1$ means that at the peak hour $t_{\mathrm{max}}$, the queue is critical.} {\gr While this assumption is trivially satisfied in the representative case of exponential arrivals, it turns out to be of crucial importance when dealing with general arrival times.}
%
%
Let us explain this subtle point in more detail below, and at the same time draw a comparison with existing lines of related research.

{\bl  In \cite{honnappa2012delta,louchard1994large}, the $\Delta_{(i)}/G/1$ queue is considered, for $t_{\max}>0$ and a wide range} of possible system behaviors related to the maximum  $f_{\sss T}(t_{\rm max})$. {\bl It seems that all the different operating regimes studied in the literature so far can be classified in terms of the parameter $f_{\sss T}(t_{\rm max})\E[S]$ and whether this quantity is smaller (underloaded), equal (heavy traffic), or larger (overloaded) than $1$. In fact, even when $f_{\sss T}(t_{\rm max})\mathbb{E} [S]\approx 1$, additional regimes can be considered by changing the rate of convergence (in our case this has to be $1+\beta/n^{-1/3}$).}
%

{\bl In \Ward~the following result was obtained for the overloaded $\Delta_{(i)}/G/1$ queue:
\begin{theorem*}[Diffusion limit for overloaded $\Delta_{(i)}/G/1$ queue \Ward]\label{th:diff_limit_overloaded_Delta_queue}
Let $\lim_{n\rightarrow\infty}\bar{Q}^{\Delta}_n(t)/n=\bar{Q}^{\Delta}(t)$ be the fluid limit of the queue length process $Q^{\Delta}_n(t)$. Let 
\begin{equation}
\hat Q_n^{\Delta} (t) = (Q_n^{\Delta}(t) - n\bar{Q}_n^{\Delta}(t))/\sqrt{n}
\end{equation}
be the diffusion-scaled queue length process. Then, as $n\rightarrow\infty$,
\begin{equation}\label{eq:ward_main_result}%
\hat Q_n^{\Delta} (\cdot) \sr{\mathrm{d}}{\rightarrow} \hat Q^{\Delta}(\cdot),
\end{equation}%
with $\hat Q^{\Delta}(\cdot)$ some stochastic process.
\end{theorem*}%
The process $ \hat Q_n^{\Delta}(\cdot)$ is a diffusion that switches between three regimes: a free Brownian motion, a reflected Brownian motion and the zero process. 
From the definition of $\hat Q_n^{\Delta}$ in \Ward~it follows easily that, if $\fT(t)\not\equiv\fT(t_{\max})$, the constant $\max_{t\geq0} \bar Q^{\Delta}(t)$ and the random variable $\max_{t\geq0}\hat Q_n^{\Delta}(t)$ are different from zero if and only if 
\begin{equation}%
\fT(t_{\mathrm{max}})\E[S] > 1.
\end{equation}%
Therefore, the queue length grows linearly around time $t_{\max}$ (i.e. $\max_{t\geq0} \bar Q^{\Delta}(t)>0$) if and only if $\fT(t_\mathrm{max})\E[S] >1$ (overloaded) and the condition $\fT(t_\mathrm{max})\E[S] = 1$ is the threshold between the overloaded and underloaded regimes. Hence, $\fT(t_\mathrm{max})\E[S] = 1$ defines the critical regime in which the queue grows like a non-trivial power of $n$, so  not linearly.
}

{\bl  Louchard \Lou~considers an approximation of the $\Delta_{(i)}/G/1$ queue, defined as follows. Let $S_n(t)$ be a renewal process associated with the service time random variables $(S_i)_{i=1}^n$. Further, let $A_n(t) = \sum_{i=1}^n \mathds 1_{\{T_i\leq t\}}$ be the usual arrival process. Then the queue length process in \Lou~is defined as 
\begin{equation}%
Q'_n(t) = \phi(J_n(t)),
\end{equation}%
where $J_n(t) := A_n(t) - S_n(t)$ and $\phi(\cdot)$ is the reflection mapping. This is the so-called Borovkov modified system, an approximation in which the server never idles and when a customer arrives and finds an empty queue, his service time is defined to be the remaining time before the next jump time of $S_n(\cdot)$. Louchard \Lou~then invokes a result by Iglehart and Whitt \cite{iglehart1970multiple} to argue that this approximation converges to the $\Delta_{(i)}/G/1$ queue length process in \emph{heavy traffic}.}
For this system, Louchard \Lou~studies a wide range of system behaviors, by dividing time into intervals that are associated with specific assumptions on $f_{\sss T}(t)$ and $\mathbb{E}[S]$, and then establishing convergence results within each interval. In that way Louchard \Lou  ~identifies using mostly heuristic arguments  the possible behaviors of the model. Our setting would then correspond to Cases $2$ and $3$ in \Lou~with $\alpha = 1/3$. 
Hence, while Louchard \Lou  ~identifies many possible stochastic-process limits, we exclusively focus on the critical system behavior, and it is that one specific heavy-traffic regime for which we formally derive the stochastic-process limits.
Moreover, we also identify further conditions on the density $f_{\sss T}$ that influence limiting behavior. We show, for instance, that different stochastic-process limits arise depending on how many derivatives of the density $f_{\sss T}(t)$ in $t_{\rm max}$ are zero. {\gr Although in the first part of the paper we focus on the  exponential case, for which all derivatives are different from zero},  our methodology can be used to obtain process limits for all such scenarios, as discussed in Section \ref{sec:general_arrivals_overview}. 

Hence, to force non-degenerate limiting behavior in the critical $\Delta_{(i)}/GI/1$ queue, uniform acceleration is required, in combination with additional assumptions on the arrival times density in its maximum. This approach  also has connections with the work of Mandelbaum and Massey \cite{mandelbaum1995strong} for the $M_t/M_t/1$ queue, who derive a fluid approximation through a FLLN and use this approximation to classify various operating regimes. In this setting, our result corresponds to the `Onset of Critical Loading' regime  \cite[Theorem $3.4$]{mandelbaum1995strong} and the results of \cite{honnappa2012delta} correspond to the FLLN and the FCLT \cite[Theorems $2.1$ and $2.2$]{mandelbaum1995strong}, as can be seen in Table \ref{tab:comparison_models}. 

{\renewcommand{\arraystretch}{2.5}
\begin{table}[!htbp]
\centering
\begin{tabular}{ c | c c | c c}
model & $\Delta_{(i)}/G/1$ & $\Delta_{(i)}/G/1$ \cite{honnappa2012delta, louchard1994large} & $M_t/M_t/1$ \cite{mandelbaum1995strong} & $M_t/M_t/1$ \cite{mandelbaum1995strong} \\
\hline
regime & \pbox{20cm}{crit. loaded\\in $t_{\mathrm max}$} & overloaded &  \pbox{20cm}{crit. loaded\\in $t_{\mathrm max}$}  & overloaded \\
time scaling & $n^{-1/3}$ & --- &$\varepsilon^{1/3}$ & --- \\
spatial scaling & $n^{-1/3}$ & \pbox{20cm}{$n^{-1}$ (FLLN) \\ $n^{-1/2}$(FCLT)}  & $\varepsilon^{1/3}$ & \pbox{20cm}{$\varepsilon$ (FLLN) \\ $\varepsilon^{1/2}$(FCLT)} \\
fluid limit & --- & $\phi(\int_0^t(\fT(s)-\mu)\mathrm{d}s)$ & --- & $\phi(\int_0^t(\lambda(s)-\mu(s))\mathrm{d}s)$\\
diffusion limit & $\phi(W)(t)$ & \pbox{20cm}{regime\\ switching}~$\left\{\pbox{20cm}{$B(t)$ \\ $\phi(B)(t)$ \\ 0}\right.$& $\phi(W)(t)$ & \pbox{20cm}{regime\\ switching}~$\left\{\pbox{20cm}{$B(t)$ \\ $\phi(B)(t)$ \\ 0}\right.$\\
\end{tabular}
\caption{A comparison between the results for our model, the $\Delta_{(i)}/G/1$ queue and the $M_t/M_t/1$ queue. For simplicity, we have omitted the variances of the Brownian motions.}\label{tab:comparison_models}
\end{table}
}
{\bl  In Table \ref{tab:comparison_models}, $W(t) = B(t) + \frac{\fT'(t_{\rm max})}{2}t^2$, $\fT(t)$ denotes the probability density function of the arrival random variables $T$ and $\phi(\cdot)$ is the reflection mapping. Moreover, for the $M_t/M_t/1$ queue it is assumed that the arrival and service processes are inhomogeneous Poisson processes with time-dependent rate $\lambda(s)$ and $\mu(s)$, respectively. The results for the critically loaded $\Delta_{(i)}/G/1$ queue appear in \Lou~(heuristically) as well as in this paper.}
It can also be seen from Table \ref{tab:comparison_models} that the $M_t/M_t/1$ queue and the  $\Delta_{(i)}/G/1$ queue are intimately related. To see this, consider an inhomogeneous Poisson process $N(\int_0^t\lambda(s)\mathrm{d}s)$ (with $N(\cdot)$ a Poisson process of rate one). Then, conditioned on  $\{N(\int_0^t\lambda(s)\mathrm{d}s) = k\}$, the $k$ points themselves are a family of i.i.d.~random variables with density given by ${\lambda(s)}/{\int_0^t\lambda(z)\mathrm{d}z}$.
In particular, if $\int_0^{\infty}\lambda(z)\mathrm{d}z < \infty$, conditioned on the total number of Poisson points in $[0,\infty)$ (say, $n$), the points are a family of $n$ i.i.d. random variables. In \cite{honnappa2014transitory}, the authors explore this connection in relation to the $\Delta_{(i)}/G/1$ queue in great detail.

\section{Main results}
{\gr Let us now summarize our results. We begin by focusing on the case of exponentially distributed arrival times. 
Given a process $X(\cdot)$ with trajectories in $\mathcal D([0,\infty),\mathbb R)$, the space of c\`adl\`ag functions (endowed with the Skorohod $J_1$ topology), we denote by $\phi (X)(\cdot)$ its reflected version, obtained through the map
\begin{equation}\label{OneDimRefl}%
\phi(f)(x):=f(x)-\inf_{y\leq x}(f \wedge 0)(y).
\end{equation}%
The critical behavior of the $\Delta_{(i)}/G/1$ queue is then determined by the following theorem:}
\begin{theorem}[The critically loaded $\Delta_{(i)}/G/1$ queue with exponential arrivals]\label{MainTheorem_delta_G_1}
Let $Q^{\Delta}_n(\cdot)$ be the queue length process of $\Delta_{(i)}/G/1$ queue with rate $\lambda$ exponential arrivals and service times $(S_i)_{i\geq1}$ such that $\E[S^2]<\infty$. Assume that the heavy-traffic condition \eqref{crit2} holds. Then
\begin{equation}\label{eq:main_theorem_delta_G_1_conclusion}%
n^{-1/3}Q_n^{\Delta}( \cdot n^{-1/3})\stackrel{\mathrm{d}}{\rightarrow} \phi(W)( \cdot ),
\end{equation}%
where $W(\cdot)$ is the diffusion process
\begin{equation}\label{eq:main_theorem_delta_G_1_diffusion_definition}%
W(t)= \beta\lambda t   -\frac{\lambda^2}{2}t^2 + \sigma B(t),
\end{equation}%
with $\sigma^2 = \lambda^3 \mathbb E[S^2]$ and $B(\cdot)$ a standard Brownian motion.
\end{theorem}
{\gr As an immediate consequence of Theorem \ref{MainTheorem_delta_G_1} we get an asymptotic result for the $\text{BP}_n$, the length of the first busy period in the $\Delta_{(i)}/G/1$ queue (so $\text{BP}_n$ describes the number of customers that are served without the queue becoming empty). In order to obtain a sizeable first busy period, we assume that the queue length at time zero is deterministic and grows with $n$ as
\begin{equation}\label{eq:busy_period_starting_queue_length}%
 \lim_{n\rightarrow\infty}\frac{Q_n^{\Delta}(0)}{n^{1/3}} = q>0.
\end{equation}%
As the next theorem shows, the size of the first busy period depends crucially both on $\beta$ and $q$:
\begin{theorem}[Busy period of the critical $\Delta_{(i)}/G/1$ queue with exponential arrivals]\label{cor:busy_period_critical_delta_queue}%
Let $Q_n^{\Delta}(\cdot)$ be the queue length process of the $\Delta_{(i)}/G/1$ queue. Assume that the heavy-traffic condition \eqref{crit2} holds. Assume further that \eqref{eq:busy_period_starting_queue_length} holds.
Let $\emph{BP}_n$ denote the first busy period of $Q_n^{\Delta}(\cdot)$. Then 
\begin{equation}\label{eq:busy_period_critical_delta_queue_statement}%
n^{1/3}\emph{BP}_n \sr{\mathrm d}{\rightarrow} T_{\sss W_q}^{\beta\lambda}(0),
\end{equation}%
where $T^{\beta\lambda}_{\sss W_q}(0)$ is the time until the process $W_q(\cdot)$ crosses level $0$, with
\begin{equation}%
W_q(t)=q+\beta \lambda t  -\frac{\lambda^2}{2}t^2 + \sigma B(t).
\end{equation}%
As above, $\sigma^2 = \lambda^3 \mathbb E[S^2]$ and $B(\cdot)$ is a standard Brownian motion.
\end{theorem}%

\begin{proof}%
For a function $f(\cdot):\mathbb R^+\mapsto\mathbb R$ such that $f(0) >0$, let $T_f(0):= \inf\{t>0: f(t)\leq0\}$ denote the time until $f(\cdot)$ crosses level $0$.
By Theorem \ref{MainTheorem_delta_G_1} and \eqref{eq:busy_period_starting_queue_length},
\begin{equation}%
n^{-1/3}Q_n^{\Delta}(\cdot n^{-1/3})\sr{\mathrm d}{\rightarrow} W_q(t).
\end{equation}%
The functional $f\mapsto T_f(0)$ is a.s.~continuous in $W_q(\cdot)$ by \cite[Chapter VI, Proposition 2.11]{jacod2003limit}. Moreover,
\begin{equation}%
n^{1/3}\text{BP}_n = n^{1/3} \inf\{t>0: Q_n^{\Delta}(t)\leq0\} = \inf\{t>0: n^{-1/3}Q_n^{\Delta}(tn^{-1/3})\leq 0\}.
\end{equation}%
The conclusion then immediately follows from these observations and the Continuous Mapping Theorem.
\end{proof}%
}
\begin{figure}[!hbtp]

	\centering
	\includegraphics{
	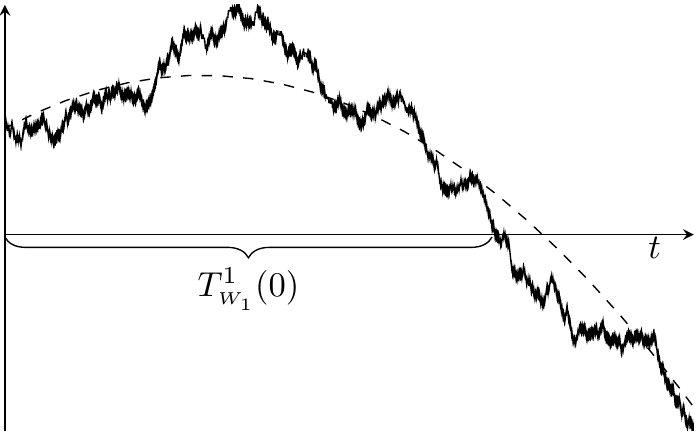}
	\caption{A sample path of $W_1(t)$ with $\beta =1$ (solid) and the drift $1+t-\frac{1}{2}t^2$ (dashed).}
	\label{fig:brownian_motion_parabolic_drift_headstart}
	
\end{figure}%
\subsection{Numerical examples}
{\bl  Theorem \ref{cor:busy_period_critical_delta_queue} can be used to obtain numerical approximations of quantities related to the first busy period of the critically loaded $\Delta_{(i)}/G/1$ queue. In \cite{martin1998final} an explicit expression for the density of $T_{\sss W}^{\beta}(0)$ is given, and numerical simulations are carried out to display how the shape of the density crucially depends on the two parameters $q$ and $\beta$.  Let Ai$(x)$ and Bi$(x)$ denote the classical Airy functions (see \cite{abramowitz1964handbook}). 
\begin{theorem}[First passage time density \cite{martin1998final}]\label{th:first_passage_time_density}%
The first crossing time of zero of $W_q(t) = q+\beta t -1/2t^2 + \sigma B(t)$ has probability density
\begin{equation}\label{eq:first_passage_time_density}%
f_q(t;\beta,\sigma) = \e^{-((t-\beta)^3+\beta^3)/6\sigma^2-\beta x}\int_{-\infty}^{+\infty}\e^{tu}\frac{B(u)A(u-x)-A(u)B(u-x)}{\pi(A^2(u)+B^2(u))}\mathrm{d}u,
\end{equation}%
where $c = (2\sigma^2)^{1/3}$, $x=q/\sigma^2>0$, $A(u)=\emph{Ai}(cu)$ and $B(u) = \emph{Bi}(cu)$.
\end{theorem}%
{\bl  Figure \ref{fig:distribution_function_convergence} shows the convergence of the empirical density function of the first excursion length of a $\Delta_{(i)}/G/1$ queue to the analytic expression provided by \cite{martin1998final} and also illustrates the influence of the parameters $\beta$ and $q$. 
}
\begin{figure}[!htb]
	\centering
	\includegraphics{
	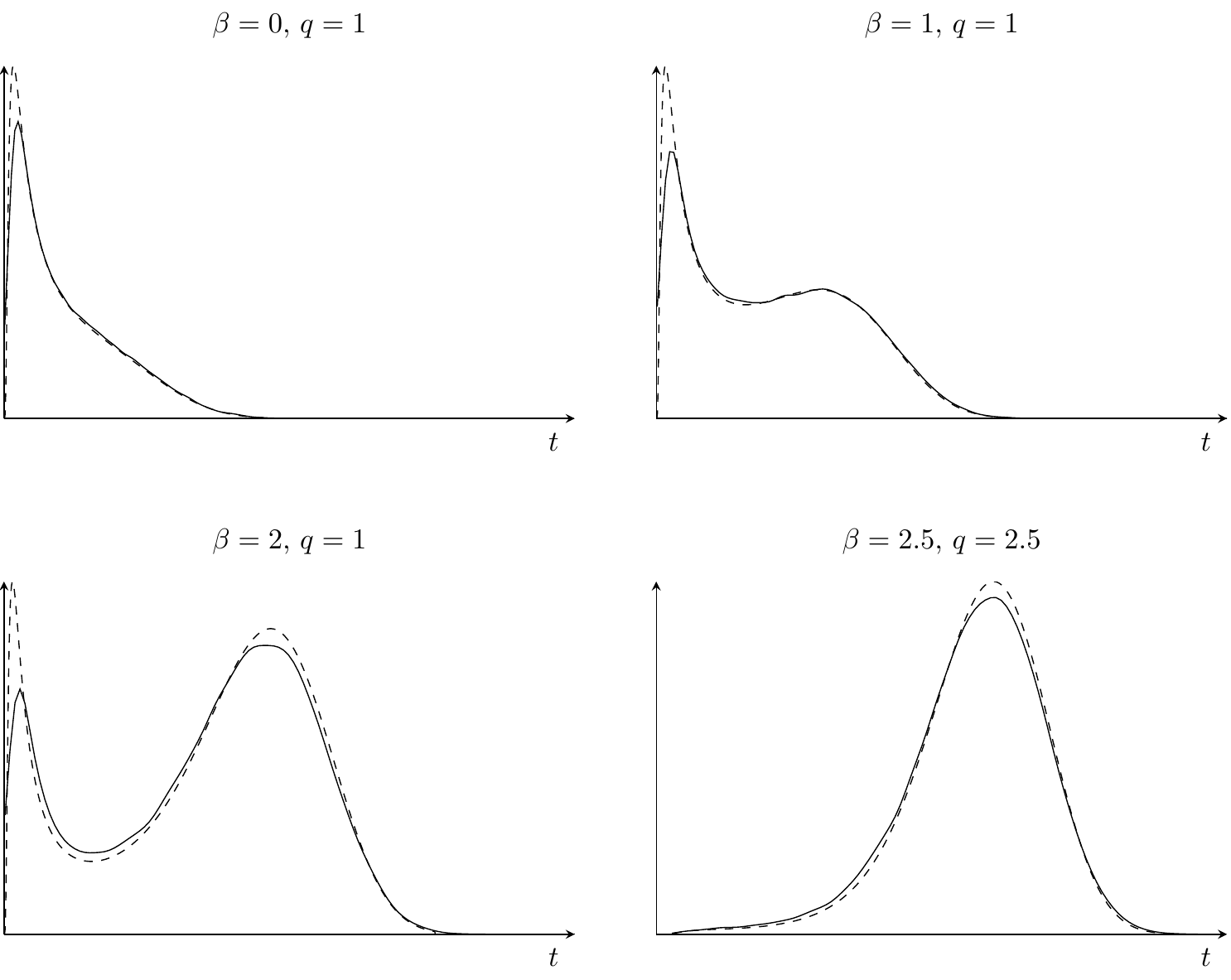}
	\caption{Density plot (dashed line) and Gaussian kernel density estimate (thick line) obtained by running $100000$ simulations of a $\Delta_{(i)}/G/1$ queue with $n=10000$ customers.}			
	\label{fig:distribution_function_convergence}
\end{figure}

{\bl Figure \ref{fig:distribution_function_convergence} suggests that to obtain a considerable first busy period, both parameters $\beta$ and $q$ must be chosen appropriately in order to avoid a concentration of the probability mass close to zero. Indeed, in \cite{martin1998final} it is conjectured that there exists a $\bar q$ such that for all $q>\bar q$ and every choice of $\beta$ the distribution function is unimodal, while for $q<\bar q$ there exists a $\bar \beta = \bar \beta (q)$ such that for $\beta <\bar \beta$ the distribution is unimodal, and bimodal otherwise.
}

In \cite{van2010critical} asymptotic expressions are determined for the tail probabilities $\mathbb P(T_{\sss W_q}^{\beta}(0)>x)$ for $x$ large. The most relevant result is the following:
\begin{theorem}[Tail busy period length for large $x$ \cite{van2010critical}]%
For bounded $q/\sigma^2>0$,
\begin{equation}\label{eq:tail_busy_period}
\lim_{n\rightarrow\infty}\mathbb P(\emph{BP}_n\geq x n^{2/3}) = \frac{3\sigma\sqrt{x}}{\sqrt{2\pi}}\frac{\exp(-\mathcal F_3(x)/6\sigma^2)}{\mathcal F_3'(x)}(1+O(1/x)), \qquad x\rightarrow\infty,
\end{equation}
with $\mathcal F_3(x) = (x - \beta)^3 - \frac{1}{4}x^3 - 3qx + \beta^3 + 6\beta q$.
\end{theorem}%
{\gr The expression in \eqref{eq:tail_busy_period} can be used, for example, to study the dependence of the probability of a very large first busy period on the parameters $q$ and $\beta$.}

We conclude by showing in Table \ref{tab:convergence_expectation_busy_period} numerical values for the mean busy period for exponential clock times with mean $1$ and different values of $q$ and $\beta$. Observe that the approximation $\E[\text{BP}_n]\approx n^{-1/3} \E[T_{\sss W}^{\beta}(0)]$ is accurate, also for moderate values of $n$.
%
\begin{table}[!htbp]
\centering
\begin{tabular}{ c  c c c c}
& \multicolumn{2}{ c }{$q=1$, $\beta=1$} & \multicolumn{2}{ c }{$q=2$, $\beta=1$}\\
\multicolumn{1}{c|}{$n$}  & $n^{1/3}\E[\mathrm{BP}_n]$ & rel. error & $n^{1/3}\E[\mathrm{BP}_n]$ & rel. error\\
\hline
\multicolumn{1}{c|}{$10$} &		3.0201	& 	0.5072 & 4.0407	&	0.4079	\\
\multicolumn{1}{c|}{$100$} &	2.2170	& 	0.1062 & 3.2611	&	0.1362	\\
\multicolumn{1}{c|}{$1000$} &	2.0341	& 	0.0151 & 2.9813	&	0.0387	\\
\multicolumn{1}{c|}{$10000$} &	2.0306	& 	0.0133 & 2.9351	&	0.0226	\\
\multicolumn{1}{c|}{$100000$} &	2.0295	& 	0.0128 & 2.9145	&	0.0155		\\
\multicolumn{1}{c|}{$\infty$} &	2.0038	& 	---    & 2.8701 &	---		\\
\end{tabular}
\caption{Mean busy period for the pre-limit queue with different population sizes and the exact expression for $n=\infty$ computed using \eqref{eq:first_passage_time_density}. Each value for the pre-limit queue is the average of $10000$ simulations. }\label{tab:convergence_expectation_busy_period}
\end{table}
}
\subsection{Outline}%
{\gr
The remainder of the paper is devoted to proving Theorems \ref{MainTheorem_delta_G_1} and \ref{cor:busy_period_critical_delta_queue}. The proof of Theorem \ref{MainTheorem_delta_G_1} consists of several steps. In Section \ref{sec:overview_proof} we give a detailed overview of the proof. In Section \ref{sec:Preparation} we settle some preliminaries and the required notation. Briefly, the proof then proceeds as follows.  First we approximate the $\Delta_{(i)}/G/1$ queue via a certain process $Q_n(\cdot)$ and this is shown to converge when appropriately rescaled (Section \ref{sec:main_theorem_exponential}). The approximation is then shown to be asymptotically equivalent to the $\Delta_{(i)}/G/1$ queue (Section \ref{sec:proof_main_theorem_delta}). As it turns out, the process $Q_n(\cdot)$ converges under much more general assumptions on the density in zero of the arrival clocks (Section \ref{sec:GeneralArrivals} and Section \ref{sec:ell_order_contact_proof}). While it does not seem to be possible to use this result to prove process convergence of the $\Delta_{(i)}/G/1$ queue process for more general arrivals, it can be exploited to provide insight into the behavior of its first busy period, and on the typical queue length (Section \ref{sec:general_arrivals_overview}).
In Section \ref{sec:ExtendedDiscussion} we discuss various connections with the literature in queueing theory, random graphs and statistics. There, we also discuss several future research directions.
}
\section{Overview of the proof}\label{sec:overview_proof}
{\gr The proof of Theorem \ref{MainTheorem_delta_G_1} proceeds in several steps. First, we construct a queueing model that \emph{approximates} the $\Delta_{(i)}/G/1$ queue. Second, we show that it \emph{coincides} with the $\Delta_{(i)}/G/1$ queue when embedded after service completions. Here we crucially use the properties of the exponential distribution. Third, we prove a version of Theorem \ref{MainTheorem_delta_G_1} for the (embedded) approximating model. Lastly, we show via a time change argument that the $\Delta_{(i)}/G/1$ queue and its embedded counterpart converge to the same limit.

Let us now present some more details about the approximating model we consider.}{\bl ~Like in the $\Delta_{(i)}/G/1$ queue, we consider} $n$ customers in a population {\bl  that all possess independent clocks that ring after i.i.d.~clock times with {\gr exponential} random variable $T$}. Whenever a clock rings, that customer joins the queue. Customers are served in order of arrival.
The service requirements of consecutive customers are given by the i.i.d.~random variables $(S_i)_{i= 1}^n$, independent from $n$. We assume that $ \mathbb E[S^2]<\infty$. We further assume that the service capacity per time unit scales as $\frac{1+\beta n^{-1/3}}{n}$, so that the service times are given by 
\begin{align}\label{eq:rescaled_service_definition}%
D_i=\frac{S_i}{n}\Big(1+\beta n^{-1/3}\Big),\qquad i=1,\ldots,n.
\end{align}%
%
After service completion, the customer is removed from the system.
As explained below \eqref{crit} we shall work under the heavy-traffic condition
\begin{equation}\label{eq:CriticalityHypBeta_exponential}%
\rho_n = n\lambda \mathbb E[D]=1 + \beta n^{-1/3}.
\end{equation}%
 {\red If, after a service completion, the system is empty, the customer with the smallest arrival time is drawn from the population and is immediately put into service.}
 {\red As will become clear, considering the queue length process embedded at service completions makes the process more amenable to mathematical analysis (e.g.~allowing access to discrete-time martingale techniques).} Let $Q_n(k)$ denote the number of customers in the queue just after the service completion of the $k$-th customer. The queue length process $(Q_n(k))_{k\geq0}$, embedded at service completions, is then given by $Q_n(0)=0$ and
%
\begin{align}\label{TrueQueueLengthIntro}%
Q_n(k)=(Q_n(k-1)+A_n(k)-1)^+,\qquad k=1,2,\ldots
\end{align}%
with $x^+=\max\{0,x\}$ and $A_n(k)$ the number of arrivals during the service time of the $k$-th customer.
Here $A_n(k)$ is given by 
\begin{equation}\label{eq:ArrivalsDefinitionGeneralIntro}%
A_n(k)={\red \sum_{i\nin\nu_{k}}}\mathds{1}_{\{\sum^{k-1}_{j=1}D_j\leq {\red T_{i}}\leq \sum^{k}_{j=1}  D _j\}}
\end{equation}%
where {\red  $\nu_k$ is the set of customers no longer in the population at the beginning of the service of the $k$-th customer}. The system defined in \eqref{TrueQueueLengthIntro} and \eqref{eq:ArrivalsDefinitionGeneralIntro} neglects idle times, which is a simplification of the $\Delta_{(i)}/G/1$ model. This in turn will greatly simplify the analysis since it allows for the representation of the process as \eqref{TrueQueueLengthIntro} and \eqref{eq:ArrivalsDefinitionGeneralIntro}. 

It is possible to give an equivalent definition of the process $Q_n(k)$ in \eqref{TrueQueueLengthIntro} through the reflection map. Define the process $(N_n(k))_{k\geq0}$ by $N_n(0)=0$ and
\begin{align}%
N_n(k)&=N_n(k-1)+A_n(k)-1.\label{QueueLength}
\end{align}%
Then it is easy to see that 
\[(Q_n(k))_{k\geq0}=(\phi(N_n)(k))_{k\geq0} \qquad \textrm{w.p.}~1.
 \]
To avoid unnecessary notation, for all discrete-time processes $(X(k))_{k\geq0}$, we write $X(t)$, with $t\in\mathbb R^+$  instead of $X(\lfloor t \rfloor)$. Note that the process defined in \eqref{QueueLength} may take negative values. We first prove a limit theorem for $N_n(\cdot)$ and then apply the reflection map to obtain a limit for the queue length process $Q_n(\cdot)$ defined in \eqref{TrueQueueLengthIntro}.
The critical behavior of this model is determined in the following theorem:
\begin{theorem}[{\gr Convergence of the approximating process}]\label{MainTheorem_exponential}
Let $Q_n(\cdot)$ be the process defined in \eqref{TrueQueueLengthIntro} and \eqref{eq:ArrivalsDefinitionGeneralIntro}, associated with rate $\lambda$ exponential arrival random variables $(T_i)_{i\geq1}$ and service times $(S_i)_{i\geq1}$ such that $\E[S^2]<\infty$. Then
\begin{equation}%
n^{-1/3}N_n(\cdot n^{2/3}) \sr{\mathrm{d}}{\rightarrow} W(\cdot),
\end{equation}%
where $W(\cdot)$ is the diffusion process
\begin{equation}
W(t)= \beta t  -\frac{1}{2}t^2 + \sigma B(t),
\end{equation}%
with $\sigma^2 = \lambda^2 \mathbb E[S^2]$ and $B(\cdot)$ a standard Brownian motion. Moreover,
\begin{equation}
n^{-1/3}Q_n( \cdot n^{2/3})\stackrel{\mathrm{d}}{\rightarrow} \phi(W)( \cdot ).
\end{equation}%
\end{theorem}

Recall that  $\phi(W)(\cdot)$ denotes the reflected version of $W(\cdot)$.
%
{\bl The time-scaling exponent $2/3$ that appears in  Theorem \ref{MainTheorem} is intimately related to the exponent that appears in Table \ref{tab:comparison_models}. By \eqref{eq:ArrivalsDefinitionGeneralIntro}, the clocks $T_i$ of the customers joining the queue are of the order $T_i=O_{\mathbb P}(1/n)$. Thus, after scaling time by $n^{2/3}$ we effectively observe the queue at times of the order $O(n^{-1/3})$.}

We now provide a heuristic argument that explains the scaling exponents in Theorem \ref{MainTheorem}.
With  $\Sigma_i := \sum_{l=1}^iD_l/n$,
\begin{align}\label{eq:first_order_approx_queue_length}%
Q_n(tn^{\alpha}) &\approx \sum_{i=1}^{tn^{\alpha}}\big ( {\sum_{j\nin\nu_{i}}} \mathds 1_{\{\sum^{i-1}_{l=1}D_l\leq {T_{j}}\leq \sum_{l=1}^i D_l\}} -1\big)\nnl
&\approx \sum_{i=1}^{tn^{\alpha}} \big(n(F_{\sss T}(\Sigma_i) - F_{\sss T}(\Sigma_{i-1}))-1\big)\nnl
&\approx \sum_{i=1}^{tn^{\alpha}}  (f_{\sss T}(\Sigma_{i-1})\mathbb E[S]-1)\approx \sum_{i=1}^{tn^{\alpha}} \sum^{i-1}_{l=1} \frac{S_l}{n} f_{\sss T}'(0)\E[S],
\end{align}%
where the last approximation comes from \eqref{crit}. This computation gives us the leading order term of the queue length process up to a multiplicative constant
\begin{align}%
Q_n(tn^{\alpha}) \approx \sum_{i=1}^{tn^{\alpha}} \frac{i}{n}\approx \frac{t^{2}}{2} n^{2\alpha - 1}.
\end{align}%
The queue is the sum of (the order of) $n^{\alpha}$ contributions, thus (ignoring dependencies) the correct spatial scaling in order to obtain Gaussian fluctuations is $n^{\alpha/2}$. Then, in order to obtain both a deterministic drift and a Brownian contribution in the limit, the order of magnitude of the first order approximation \eqref{eq:first_order_approx_queue_length} should equate that of the diffusion approximation, $n^{\alpha/2}$. Hence $2\alpha-1 = \alpha/2$, so that $\alpha$ should be $2/3$.
\subsection{Overview of the proof of Theorem \ref{MainTheorem_delta_G_1}}\label{sec:proof_main_theorem_delta_overview}
{\gr We now show how Theorem \ref{MainTheorem_delta_G_1} can be deduced from Theorem \ref{MainTheorem_exponential} through a time change argument.
With $k\mapsto\tilde Q_n^{\Delta}(k)$ denoting the $\Delta_{(i)}/G/1$ queue embedded at service completions, we first argue that $\tilde Q_n^{\Delta}(\cdot)$ is closely related to $Q_n$:
\begin{lemma}[Distribution of the embedded $\Delta_{(i)}/G/1$ queue]\label{lem:embedded_queue_coincides_with_delta}
For all $k\geq 1$,
\begin{equation}%
Q_n(k) \sr{\mathrm d}{=}\tilde Q_n^{\Delta}(k).
\end{equation}%
\end{lemma}
\begin{proof}

One can see this by coupling the two queues as follows. The sequence of service times $(S_i)_{i=1}^n$ are taken to be the same for the two queues while the arrival clocks coincide until the end of the first busy period. After that, assign new clocks to the customers still in the population. The first customer after the idle period in the $\Delta_{(i)}/G/1$ queue is also the customer placed into service in the approximating model. At the beginning of the busy period, assign new clocks to the customers still in the population. The coupling then proceeds in this manner until the population in both queues is depleted. By the properties of exponentials, these new processes (with the clocks drawn multiple times) coincide in distribution with the original ones (with the clocks drawn at the start of the system), since
\begin{equation}%
\mathbb P (T_i \geq B + I + x \vert T_i \geq B + I) = \mathbb P (T_i \geq x ) = \mathbb P (T_i \geq B  + x \vert T_i \geq B ), 
\end{equation}%
where $B = B(t)$ is the busy time process and $I=I(t)$ the idle time process at the instant $t$ in which a new busy period starts.
The number of arrivals during one service time is then the same in the two coupled queues because the arrival times are equal. In particular, the queues sampled at the end of a service time have the same distribution.
\end{proof}

The next step is to prove that the supremum distance between $Q_n(\cdot)$ and $Q_n ^{\Delta}(\cdot)$ (when suitably rescaled in space and time) converges to zero. Let $\|f\|_T:= \sup_{t\leq T}\vert f(t)\vert$ denote the supremum norm. The claim is contained in the following lemma:
}
\begin{lemma}[Asymptotic equivalence of the approximating model]\label{lem:asymptotic_equivalence_approximating_model}
For each $T>0$, as $n\rightarrow\infty$,
\begin{equation}\label{eq:claim_delta_queue_converges_to_our_model}%
n^{-1/3}\Big\| Q_n\big(\frac{\cdot}{\E[S]} n^{2/3}\big) - {Q}_n^{\Delta}(\cdot n^{-1/3})\Big\|_T \sr{\mathbb P}{\rightarrow} 0.
\end{equation}%
\end{lemma}
In particular, by Slutsky's theorem, Lemma \ref{lem:asymptotic_equivalence_approximating_model} and Theorem \ref{MainTheorem_exponential} imply
\begin{equation}%
n^{-1/3}{Q}_n^{\Delta}(\cdot n^{-1/3})\sr{\mathrm{d}}{\rightarrow} W(\lambda\cdot),
\end{equation}%
where $W(t) = \beta t - 1/2 t^2 + B(t)$ and $B(\cdot)$ Brownian motion.
Without loss of generality, we will assume from now on that $\E[S]=1/\lambda=1$. To prove \eqref{eq:claim_delta_queue_converges_to_our_model} we split
\begin{equation}%
\| Q_n(\cdot) - {Q}_n^{\Delta}(\cdot)\|_T \leq \| Q_n(\cdot) - Q_n(\varphi_n(\cdot))\|_T + \| Q_n(\varphi_n(\cdot)) - {Q}_n^{\Delta}(\cdot)\|_T,
\end{equation}%
for an appropriate yet still unspecified time change $\varphi_n(\cdot)$.

Thus, we are left to prove the two following claims:
\begin{enumerate}%
\item[(i)] $n^{-1/3}\| Q_n(\cdot n^{2/3}) - Q_n(\varphi_n(\cdot) n^{2/3})\|_T \sr{\mathbb P}{\rightarrow} 0$;
\item[(ii)] $n^{-1/3}\| Q_n(\varphi_n(\cdot) n^{2/3}) - {Q}_n^{\Delta}(\cdot n^{-1/3})\|_T \sr{\mathbb P}{\rightarrow} 0$.
\end{enumerate}%
The idea behind introducing $\varphi_n(\cdot)$ is to rescale time  so that each time-step (corresponding to one service) is replaced by the actual length of the service time. In this way, the interval $[0,1]$ is replaced by $[0,S_1/n]$, interval $[1,2]$ is replaced by $[S_1/n, S_1/n+S_2/n]$ (if a customer has arrived during the first service), and so on. The time change $\varphi_n(\cdot)$ must also take idle times into account. {\gr The precise expression of $\varphi_n(\cdot)$ is given in Section \ref{sec:proof_main_theorem_delta}.} Claim (i) then states that the time change $\varphi_n(\cdot)$ is asymptotically trivial.  Indeed, the following holds:
\begin{lemma}[Verification of Claim (i)]\label{lem:claim_1}
As $n\rightarrow\infty$,
\begin{equation}\label{eq:claim_time_change_converges_identity}%
\| t - \varphi_n(t)\|_T\sr{\mathbb P}{\rightarrow}0.
\end{equation}%
Consequently,
\begin{equation}\label{eq:claim_distance_between_queue_and_time_scaled_queue_converges_to_zero}%
n^{-1/3} \| Q_n(\cdot n^{2/3}) - Q_n(\varphi_n(\cdot) n^{2/3})\|_T\sr{\mathbb P}{\rightarrow}0.
\end{equation}%
\end{lemma}
A crucial step in proving \eqref{eq:claim_time_change_converges_identity} is in proving that the idle time process of the $\Delta_{(i)}/G/1$ queue is negligible in the scaling regime we consider. We postpone the proof of this and of Lemma \ref{lem:claim_1} to Section \ref{sec:proof_main_theorem_delta}.

After this time change the two queues $Q_n(\varphi_n(\cdot))$ and ${Q}_n^{\Delta}(\cdot)$ are synchronized in time. It still remains to be proven that their supremum distance converges to zero. However, in their coupling $Q_n(\varphi_n(\cdot))$ is constructed by sampling ${Q}_n^{\Delta}(\cdot)$ at service completions, so that the two coincide at the time of each service completion. In other words, the maximum distance between $Q_n(\varphi_n(\cdot))$ and ${Q}_n^{\Delta}(\cdot)$ is the maximum number of arrivals during a single service time up until time $T$, that is 
\begin{equation}%
\| Q_n(\varphi_n(\cdot) n^{2/3}) - {Q}^{\Delta}_n(\cdot n^{-1/3})\|_T = \max_{k\leq Tn^{2/3}}A_n(k),
\end{equation}%
where $A_n(k)$ is the number of arrivals during the $k$-th service time, as defined in \eqref{eq:ArrivalsDefinitionGeneralIntro}. This is proven to be negligible, thus concluding the proof of Lemma \ref{lem:asymptotic_equivalence_approximating_model}, in the following lemma:
\begin{lemma}[Verification of Claim (ii)]\label{lem:claim_2}%
As $n\rightarrow\infty$,
\begin{equation}\label{eq:number_arrivals_o_small_n_one_third}%
n^{-1/3}\max_{k\leq Tn^{2/3}}A_n(k)\sr{\mathbb P}{\rightarrow} 0.
\end{equation}%
Consequently,
\begin{equation}\label{eq:claim_two_in_lemma}%
n^{-1/3}\| Q_n(\varphi_n(\cdot) n^{2/3}) - {Q}_n^{\Delta}(\cdot n^{-1/3})\|_T \sr{\mathbb P}{\rightarrow} 0.
\end{equation}%
\end{lemma}%
We postpone the proof to Section \ref{sec:proof_main_theorem_delta}.

\subsection{General arrivals}\label{sec:general_arrivals_overview}

{\gr Theorem \ref{MainTheorem_exponential} can be generalized to allow for generally distributed arrival times. Let us now introduce our assumptions for this case.}

The arrival times of the $n$ customers are drawn independently from a common distribution. Let the random variable $T$ have distribution function $F_{\sss T}$. 
Assume that $T$ admits a positive and continuous density function $f_{\sss T}(\cdot)$, with $f_{\sss T}(0)\in(0,\infty)$. 
In addition, as a technical assumption, we assume that the sublinear terms of the distribution function $F_{\sss T}$ decay as quickly as
\begin{align}\label{eq:ArrivalDistriNearXBar}%
F_{\sss T}(x) - F_{\sss T}(\bar x) = f_{\sss T}(\bar x) (x - \bar x) +o(\vert \bar x - x\vert^{4/3}) \qquad\forall\bar x\in(0,\infty).
\end{align}%
This is, for example, the case when $F_{\sss T}\in \mathcal C^2([0,\infty))$. As an additional technical assumption, this error should be uniform over $\bar x$, for small values of $\bar x$, as in 
\begin{equation}\label{eq:ArrivalDistrUnifError}%
\sup_{\bar x\leq Cy^{1/3}}\left\vert F_{\sss T}(\bar x + y) - F_{\sss T}(\bar x) - f_{\sss T}(\bar x)y\right\vert = o( y),
\end{equation}%
where we set $x-\bar x = y$ for convenience. We also assume that $f'_{\sss T}(\cdot)$  exists and is continuous in a neighborhood of zero. In particular \eqref{eq:ArrivalDistrUnifError} is always satisfied since
\begin{equation}%
\sup_{\bar x\leq Cy^{1/3}}\left\vert F_{\sss T}(\bar x + y) - F_{\sss T}(\bar x) - f_{\sss T}(\bar x)y\right\vert = \sup_{\substack{\bar x\leq Cy^{1/3}\\ \zeta\in(\bar x,\bar x+y)}}\left\vert \frac{\fT'(\zeta)}{2}y^2 \right\vert \leq \frac{M}{2} y^2,
\end{equation}%
where $M$ is the supremum of $\fT'(\cdot)$ in a neighborhood of zero.
In particular, it also holds that
\begin{align}\label{eq:DensityDistrNearZero}%
f_{\sss T}(x)=f_{\sss T}(0)+f'_{\sss T}(0)\cdot x + o (x),
\end{align}%
Since $\lim_{x\to\infty} f_{\sss T}(x) = 0$ and $f_{\sss T}$ is continuous on $[0,\infty)$, it admits a maximum in $[0,\infty)$. While the maximum can be attained in principle in multiple points, a crucial additional assumption is that it is also attained in 0, that is
\begin{equation}\label{eq:sup_density_is_in_zero}%
f_{\sss T}(0) = \sup_{x\geq0}f_{\sss T}(x).
\end{equation}%
The service requirements of consecutive customers are again given by \eqref{eq:rescaled_service_definition} and are such that $\E[S^2]<\infty$.
The heavy-traffic regime for this model is given by a condition analogous to \eqref{eq:CriticalityHypBeta_exponential}, that is
\begin{equation}\label{eq:CriticalityHypBeta}%
\rho_n = nf_{\sss T}(0)\cdot  \mathbb E[D]=1 + \beta n^{-1/3}.
\end{equation}%
\begin{theorem}[{\gr Convergence of the approximating process for general arrivals}]\label{MainTheorem}
Let $Q_n(\cdot)$ be the process defined in \eqref{TrueQueueLengthIntro} and \eqref{eq:ArrivalsDefinitionGeneralIntro}, associated with the  arrival random variables $(T_i)_{i\geq1}$ and service times $(S_i)_{i\geq1}$ such that $\E[S^2]<\infty$. Then, under the assumptions as above,
\begin{equation}%
n^{-1/3}N_n(\cdot n^{2/3}) \sr{\mathrm{d}}{\rightarrow} W(\cdot),
\end{equation}%
where $W(\cdot)$ is the diffusion process
\begin{equation}\label{eq:main_theorem_diffusion_definition}%
W(t)= \beta t  + \frac{f'_{\sss T}(0)}{2f_{\sss T}(0)^2}t^2 + \sigma B(t),
\end{equation}%
with $\sigma^2 = f^2_{\sss T}(0)\mathbb E[S^2]$ and $B(\cdot)$ a standard Brownian motion. Moreover,
\begin{equation}\label{eq:main_theorem_conclusion}%
n^{-1/3}Q_n( \cdot n^{2/3})\stackrel{\mathrm{d}}{\rightarrow} \phi(W)( \cdot ).
\end{equation}%
\end{theorem}
{\gr 
We postpone the involved proof to Section \ref{sec:GeneralArrivals}. Note that when $T$ is exponentially distributed, $f'_{\sss T}(0)/(2f_{\sss T}(0)^2)=-1/2$, so that Theorem \ref{MainTheorem} can be seen as a generalization of Theorem \ref{MainTheorem_exponential}.
The lack of the memoryless property for the arrival clocks $T_i$ makes that the arguments in Section \ref{sec:proof_main_theorem_delta_overview}  cannot be carried over to this case. Therefore, it does not seem possible to deduce Theorem \ref{MainTheorem_delta_G_1} (for general arrival clocks) from Theorem \ref{MainTheorem}. For example, the coupling between the approximating model and the $\Delta_{(i)}/G/1$ queue in the proof of Lemma \ref{lem:embedded_queue_coincides_with_delta} will break down after the end of the first busy period. However, this still allows us to prove results for the first busy period of the $\Delta_{(i)}/G/1$ queue with general arrivals. In particular we have the following generalization of Theorem \ref{cor:busy_period_critical_delta_queue}:
\begin{theorem}[Busy period of the critical $\Delta_{(i)}/G/1$ queue with general arrivals]\label{cor:busy_period_critical_delta_queue_general_arrivals}%
Let $Q_n^{\Delta}(\cdot)$ be the queue length process of the $\Delta_{(i)}/G/1$ queue. Assume that \eqref{eq:busy_period_starting_queue_length} holds. Assume further \eqref{eq:ArrivalDistriNearXBar}-\eqref{eq:CriticalityHypBeta} and that $\fT'(0)<0$. 
Let $\emph{BP}_n$ denote the first busy period of $Q_n^{\Delta}(\cdot)$. Then 
\begin{equation}%
n^{1/3}\emph{BP}_n \sr{\mathrm d}{\rightarrow} T_{\sss W_q}^{\beta}(0),
\end{equation}%
where $T^{\beta}_{\sss W_{q,f_T}}(0)$ is the time until the process $W_{q,f_T}(\cdot)$ crosses level $0$, with
\begin{equation}\label{eq:limiting_brownian_delta_general_arrivals}%
W_{q,f_T}(t)=q+\beta \fT(0)t  + \frac{\fT'(0)}{2}t^2 + \sigma B(t).
\end{equation}%
Here, $\sigma^2 = \fT^3(0) \mathbb E[S^2]$ and $B(\cdot)$ is a standard Brownian motion.
\end{theorem}%
\begin{proof}
By coupling the approximating model and the $\Delta_{(i)}/G/1$ queue as in the proof of Lemma \ref{lem:embedded_queue_coincides_with_delta}, we get that
\begin{equation}%
Q_n(k) \sr{\mathrm d}{=} \tilde Q_n^{\Delta}(k),\qquad \text{for}~k\leq T_{Q_n}(0).
\end{equation}%
that is, the two embedded processes coincide in distribution until the end of the first busy period. By proceeding as in Lemma \ref{lem:asymptotic_equivalence_approximating_model} we see that
\begin{equation}
n^{-1/3}\| Q_n(\cdot\fT(0) n^{2/3}) - Q_n^{\Delta}(\cdot n^{-1/3}) \|_{T_{Q_n}(0)}\sr{\mathbb P}{\rightarrow}0.
\end{equation}%
In particular,
\begin{equation}
n^{-1/3}Q_n^{\Delta}( \cdot  \wedge T_{Q^{\Delta}_n}(0)n^{-1/3})\sr{\mathrm d}{\rightarrow}W_q((\cdot \wedge T_{W_q}(0))\fT(0)).
\end{equation}
Then,
\begin{align}%
n^{1/3} \text{BP}_n &= n^{1/3} \inf\{t>0: Q_n^{\Delta}(t) \leq 0\} \nnl
&= \inf\{t>0: n^{-1/3}Q_n^{\Delta}(tn^{-1/3})\leq0\}.
\end{align}%
Finally, the same continuity argument as in Theorem  \ref{cor:busy_period_critical_delta_queue} gives the conclusion.
\end{proof}
}
{\gr In Table \ref{tab:convergence_expectation_busy_period_general_arrivals} we show numerically that the (rescaled) average busy period of the $\Delta_{(i)}/G/1$ queue with general (hyperexponential) arrivals converges to the exact value obtained by Lemma \ref{th:first_passage_time_density}. The arrival random variable is exponentially distributed with rate $\lambda_1=2$ with probability $p_1=0.2$ and with rate $\lambda_2=3/4$ with probability $p_2=0.8$. Note that formula \eqref{eq:first_passage_time_density} holds for a parabolic drift of the form $-\frac{1}{2}t^2$. However, this can be extended to more general coefficients of the parabolic term by some simple scaling properties. In particular, the first hitting time of zero of $W_{q,f_T}(\cdot)$ is distributed as 
\begin{equation}\label{eq:scaling_property_hitting_time_zero}
T_{W_{q,f_T}}^{\beta}(0) \sr{\mathrm d}{=}k^{-2/3} T_{W_{q k^{1/3}}}^{\beta \fT(0)k^{-1/3}}(0) ,
\end{equation}
where $k = \fT'(0)$ and $T_{W_q}^{\beta}(0)$ is defined as below \eqref{eq:busy_period_critical_delta_queue_statement}. Relation \eqref{eq:scaling_property_hitting_time_zero} follows from a more general scaling relation, see e.g. \cite[Section 4.1]{van2010critical}.
}
\begin{table}[!htbp]
\centering
\begin{tabular}{ c  c c c c}
& \multicolumn{2}{ c }{$q=1$, $\beta=1$} & \multicolumn{2}{ c }{$q=2$, $\beta=1$}\\
\multicolumn{1}{c|}{$n$}  & $n^{1/3}\E[\mathrm{BP}_n]$ & rel. error & $n^{1/3}\E[\mathrm{BP}_n]$ & rel. error\\
\hline
\multicolumn{1}{c|}{$10$} &		2.8630	& 	0.6581 & 3.8646	&	0.5620	\\
\multicolumn{1}{c|}{$100$}&		1.9862	& 	0.1503 & 2.9665	&	0.1991	\\
\multicolumn{1}{c|}{$1000$} &	1.8103	& 	0.0484 & 2.6486	&	0.0706	\\
\multicolumn{1}{c|}{$10000$} &	1.7725	& 	0.0265 & 2.5596	&	0.0346	\\
\multicolumn{1}{c|}{$100000$} &	1.7440	& 	0.0100 & 2.5050	&	0.0125	\\
\multicolumn{1}{c|}{$\infty$} &	1.7267	& 	--- & 2.4740	&	---		\\
\end{tabular}
\caption{Mean busy period for the pre-limit queue with general (hyperexponential) arrivals and different population sizes and the exact expression for $n=\infty$ computed using \eqref{eq:first_passage_time_density}. The hyperexponential distribution has parameters $(p_1,p_2,\lambda_1,\lambda_2)=(0.2,0.8,2,3/4)$. Each value for the pre-limit queue is the average of $10000$ simulations. }\label{tab:convergence_expectation_busy_period_general_arrivals}
\end{table}

{\gr Several difficulties arise when trying to prove Theorem \ref{MainTheorem_delta_G_1} for general arrivals using similar techniques as the ones for exponential arrivals. These difficulties can be traced back to the fact that the $\Delta_{(i)}/G/1$ queue and its approximating process cannot be coupled beyond  the first busy period, due to the lack of memoryless property of the arrival clocks. However, there are several strong indications that Theorem \ref{MainTheorem_delta_G_1} should also hold for generally distributed arrival clocks. The proof of Theorem \ref{cor:busy_period_critical_delta_queue_general_arrivals} implies that the first excursion of the $\Delta_{(i)}/G/1$ queue length process to the first excursion of an appropriate Brownian motion with parabolic drift. Moreover, the limiting process \eqref{eq:limiting_brownian_delta_general_arrivals} only depends on the distribution of $T$ through $\fT(0)$ ($\lambda$ for exponential clocks), suggesting that the result is insensible to the arrival clocks distribution, as long as $\fT(0)>0$. This can be further justified as follows. Let $T_{(1)}\leq T_{(2)}\leq\ldots \leq T_{(n)}$ the order statistics of the arrival clocks $(T_i)_{i=1}^n$. These describe the ordered arrival times of the customers. By a standard argument, $T_{(i)}$ is distributed as 
\begin{equation}\label{eq:order_statistics_distributional_equality}%
T_{(i)} \stackrel{\mathrm d}{=} \FT^{-1} ( 1 - \exp(-E_{(i)})),
\end{equation}%
where $E_{(i)}$ is the $i$-th order statistic associated with $n$ i.i.d.~exponential random variables $E_1,E_2,\ldots, E_n$ with unitary mean. By Taylor expanding expression \eqref{eq:order_statistics_distributional_equality} one sees that the order statistics of $(T_i)_{i=1}^n$ are closely approximated by those of $(E_i)_{i=1}^n$. Therefore, the result for exponential arrivals should be closely related to the one for general arrivals. These considerations lead us to formulate the following conjecture:

\flushleft
{\leftskip4em
\textbf{Conjecture}: Theorem \ref{MainTheorem_delta_G_1} holds for generally distributed arrival times satisfying assumptions \eqref{eq:ArrivalDistriNearXBar}-\eqref{eq:CriticalityHypBeta}, with $W(t) = \beta\fT(0)t-\frac{\fT'(0)}{2}t^2+\sigma B(t)$, where $\sigma^2=\fT(0)^3\E[S^2]$ and $B(\cdot)$ a standard Brownian motion.

\rightskip\leftskip}

}

\subsubsection{The $\ell$-th order contact case}\label{sec:HigherOrderContact}
The technique developed to prove Theorem \ref{MainTheorem} can be exploited to prove limit results for the more general case in which the function $t\mapsto f_{\sss T}(t)-1/{\E[S]}$ has $\ell$-th order contact in zero, defined as follows:
\begin{definition}[$\ell$-th order contact point]\label{def:ellthOrderContactPoint}%
Given a smooth, real-valued, function $f(t)$, we say it has $\ell$-th order contact in $\bar t$ if  $f(\bar t) = 0$, $f^{(l)}(\bar t) =  0$ for $l=1,\ldots, \ell-1$ and $f^{(\ell)}(\bar t) \neq 0$. 
\end{definition}%
If $f(\bar t) = o(1)$ and all the other assumptions on Definition \ref{def:ellthOrderContactPoint} are kept, we still say that $f(\cdot)$ has an $\ell$-th order contact in $\bar t$. Indeed, our criticality assumption is $f_{\sss T}(0) - 1/\E[S]= o(1)$, where the error term is specified later.  The assumption that the argmax of $f_{\sss T}(\cdot)$ is zero allows us to consider both odd and even order contacts. In this case, Theorem \ref{MainTheorem} can be generalized as follows 
\begin{theorem}[Asymptotics for the critical $\ell$-th order contact queue]\label{th:MainTheoremKthOrderContact}%
Assume that the function $f_{\sss T}(t) - 1/\E[S]$ has $\ell$-th order contact in $0$, where $\ell\geq 1$. Define 
\begin{equation}\label{eq:AlphaDefinition}%
\alpha = \frac{\ell}{\ell + 1/2}.
\end{equation}%
Then,
\begin{align}%
n^{-\alpha/2} Q_n(\cdot n^{\alpha}) \stackrel{\text{d}}{\rightarrow} \phi(W)(\cdot),
\end{align}%
where $W(\cdot)$ is given by
\begin{align}%
W(t) =  \beta t - c t^{\ell+1} + \sigma B(t),\qquad c,\sigma\in\mathbb R^+, 
\end{align}%
and $B(\cdot)$ is a standard Brownian motion.
\end{theorem}%
Note that $\ell = 1$, corresponding to the case considered in Theorem \ref{MainTheorem}, returns the same scaling as in Theorem \ref{MainTheorem}. Moreover, the case $\ell = 2$ has already been known in the literature for quite some time, at least at a heuristic level. Newell derived the correct exponents ($\alpha = 4/5$) through an argument using the Fokker-Planck equation associated with the queue length process (see \cite[part III]{newell1968queues}).
Notice also that $\lim_{\ell\to\infty}\frac{\alpha}{2}=\lim_{\ell\to\infty}\frac{\ell}{2\ell + 1} = \frac{1}{2}$, suggesting that the right scaling for the uniform arrivals case ($\infty$-order contact) is the diffusive one.

An analogous heuristic to the one below \ref{MainTheorem_exponential} motivates the expressions for $\alpha$  and the limit process. We first consider a first-order approximation of the queue. Let  $\Sigma_i := \sum_{l=1}^iD_l/n$, then
\begin{align}%
Q_n(tn^{\alpha}) &\approx \sum_{i=1}^{tn^{\alpha}}\big ( {\red \sum_{j\nin\nu_{i}}} \mathds 1_{\{\sum^{i-1}_{l=1}D_l\leq {\red T_{j}}\leq \sum_{l=1}^i D_l\}} -1\big)\nnl
&\approx \sum_{i=1}^{tn^{\alpha}} \big(n(F_{\sss T}(\Sigma_i) - F_{\sss T}(\Sigma_{i-1}))-1\big)\nnl
&\approx \sum_{i=1}^{tn^{\alpha}}  (f_{\sss T}(\Sigma_{i-1})\cdot\mathbb E[S]-1)\approx \sum_{i=1}^{tn^{\alpha}} \Big (\sum^{i-1}_{l=1} S_l/n\Big)^{\ell} f_{\sss T}^{(\ell)}(0)\E[S],
\end{align}%
where as before the last approximation comes from the criticality assumption. Thus the leading order term (up to a multiplicative constant) of the queue length process is
\begin{align}\label{eq:kOrderContactQFirstOrderApprox}%
Q_n(tn^{\alpha}) \approx \sum_{i=1}^{tn^{\alpha}} \frac{i^{\ell}}{n^{\ell}}\approx t^{\ell+1} n^{(\ell+1)\alpha - \ell}.
\end{align}%
For both a deterministic drift and a Brownian contribution to appear in the limit, the order of magnitude of the first order approximation \eqref{eq:kOrderContactQFirstOrderApprox} should equate $n^{\alpha/2}$. This gives
\begin{align}%
(\ell+1)\alpha - \ell = \alpha /2 ~\Rightarrow~ \alpha = \frac{\ell}{\ell + 1/2}.
\end{align}%
The formal proof of Theorem \ref{th:MainTheoremKthOrderContact} mimics what has been done for the case $\ell = 1$. One might expect that assumptions on the existence of higher moments of $S$ are needed, but this is not the case. In Section \ref{sec:ell_order_contact_proof}, we perform the key steps in the analysis in order to show how to proceed in this general case.

{\gr When the arrival clocks distribution function has an $\ell$-th order contact in 0, a result analogue to Theorem \ref{cor:busy_period_critical_delta_queue_general_arrivals} holds for the $\Delta_{(i)}/G/1$ queue. However, the parabolic drift is replaced by a higher order polynomial drift, and no closed form formulas are available for computing statistics of the first busy period. We leave this issue for future research. In Table \ref{tab:convergence_expectation_busy_period_halfnormal_arrivals} we compute the mean first busy period of a $\Delta_{(i)}/G/1$ queue with half-normally distributed arrival clocks. Recall that, if $Z$ is normally distributed with mean $0$ and variance $\sigma^2$, then $\vert Z \vert$ is half-normally distributed with scale parameter $\sigma$. In particular, $\fT'(0) = 0$ and $\fT''(0)<0$.
\begin{table}[!htbp]
\centering
\begin{tabular}{ c  c  c}
& \multicolumn{1}{ c }{$q=1$, $\beta=1$} & \multicolumn{1}{ c }{$q=2$, $\beta=1$}\\
\multicolumn{1}{c|}{$n$}  & $n^{1/5}\E[\mathrm{BP}_n]$ &  $n^{1/5}\E[\mathrm{BP}_n]$\\
\hline
\multicolumn{1}{c|}{$10$} &		3.8340	& 	 5.3984			\\
\multicolumn{1}{c|}{$100$}&		3.0997	& 	 4.1232			\\
\multicolumn{1}{c|}{$1000$} &	2.8378	& 	 3.8772			\\
\multicolumn{1}{c|}{$10000$} &	2.7801	& 	 3.7721			\\
\multicolumn{1}{c|}{$100000$} &	2.7942	& 	 3.7548			\\
\end{tabular}
\caption{Mean busy period for the pre-limit queue with general (half-normal) arrivals and different population sizes. The half-normal distribution has scale parameter $\sigma=\sqrt{\pi}/\sqrt{2}$, and $\fT'(0)=0$, $\fT''(0) < 0$. Each value for the pre-limit queue is the average of $10000$ simulations.}\label{tab:convergence_expectation_busy_period_halfnormal_arrivals}
\end{table}
}
\subsection{Preliminaries}
Here we present some auxiliary results used in the proof of Theorem \ref{MainTheorem}. All random variables defined from now on are defined on some complete probability space $(\Omega,\mathcal F, \mathbb P)$. When needed, elements of $\Omega$ will always be denoted by $\omega$. Given two real-valued random variables $X,Y$ we say that $X$ \emph{stochastically dominates} $Y$, and we denote it by $Y \preceq X$, if 
\begin{align*}%
\mathbb P (X\leq x)\leq \mathbb P (Y\leq x),\qquad\forall x\in\mathbb R,
\end{align*}%
so that for every nondecreasing function $f:\mathbb R\rightarrow\mathbb R$
\begin{align}%
\mathbb E [f(Y)] \leq \mathbb E[f(X)].\label{StochDomLemma}
\end{align}%
If $X$ and $Y$ are defined on the same probability space $\Omega$, and $X(\omega) \leq Y(\omega)$ for almost every $\omega\in\Omega$, then we write $X\sr{\mathrm{a.s.}}{\leq} Y$. We write $f(n) = O(g(n))$ for functions $f, g\geq0$ and $n\rightarrow \infty$ if there exists a constant $C>0$ such that $\lim_{n\rightarrow\infty}f(n)/g(n)\leq C$, and $f(n) = o(g(n))$ if $\lim_{n\rightarrow\infty}f(n)/g(n) = 0$. Furthermore, we write $O_{\mathbb P}(a_n)$ for a sequence of random variables $X_n$ for which $\vert X_n\vert/a_n$ is tight as $n\rightarrow\infty$. Moreover, we write $o_{\mathbb P}(a_n)$ for a sequence of random variables $X_n$ for which $\vert X_n\vert/a_n\stackrel{\mathbb P}{\rightarrow}0$ as $n\rightarrow\infty$. We say that a sequence of events $(E_n)_{n\geq1}$ holds with high probability, or w.h.p. for short, if $\mathbb P(E_n) \rightarrow1$ as $n\rightarrow\infty$.

All the processes we deal with are elements of the space $\mathcal D([0,\infty),\mathbb R)$ (or $\mathcal D$, for short) of c\`adl\`ag functions,  defined as the space of all functions $f:\mathbb [0,\infty)\rightarrow\mathbb R$ such that for all $\bar x\in[0,\infty)$
\begin{itemize}
\item[-] $\lim_{x\rightarrow \bar x^+}f(x)=f(\bar x)$;
\item[-] $\lim_{x\rightarrow \bar x^-}f(x)<\infty$.
\end{itemize}
Hence, $N_n(\cdot)$ is a $\mathcal D$-valued random variable. Following \cite{billingsley2009convergence}, we say that $X_n$ \emph{converges in distribution} (or \emph{weakly converges}) to $X$ (and denote it by $X_n\stackrel{\mathrm{d}}{\rightarrow} X$) if $\mathbb E[f(X_n)]\rightarrow \mathbb E[f(X)]$ as $n$ tends to infinity for every $f(\cdot)$ that is real-valued, bounded and continuous. In particular, if $X$ is $\mathcal D$-valued, $f(\cdot)$ can be any (almost everywhere) continuous function from $\mathcal D$ to $\mathbb R$. Thus, to formally establish convergence in distribution in the space of $\mathcal D$-valued random variables, a metric, or at least a topology, on $\mathcal D$ is needed (in order to define continuity of the functions $f(\cdot)$). Several topologies on the space $\mathcal D$ have been defined (all by Skorokhod, in his famous paper \cite{skorokhod1956limit}). For our purposes we will consider the $J_1$ topology, which can be described as being generated by some metric $d_{\infty}$ on $\mathcal D([0,\infty),\mathbb R)$ defined as an extension of some metric $d_{t}$ on $\mathcal D([0,t],\mathbb R)$. The latter is defined as follows. Let $\Vert\cdot\Vert$ indicate the supremum norm, id$(\cdot)$ the identity function on $[0,t]$ and $\Lambda_t$ the space of nondecreasing homeomorphisms on $[0,t]$. Define, for any $x_1,x_2\in \mathcal D$,
\begin{align}
d_{t}(x_1,x_2):=\inf_{\lambda\in\Lambda_t}\{\max\{\Vert\lambda(\cdot)-\textrm{id}(\cdot)\Vert, \Vert x_1(\cdot)-x_2(\lambda(\cdot))\Vert\}\}
\end{align} 
and
\begin{align}
d_{\infty}(x_1,x_2):=\int_0^{\infty}\e^{-t}[d_{t}(x_1,x_2)\wedge 1]\mathrm{d}t.
\end{align}
It is possible to show (see, e.g., \cite{whitt1980some}) that this is the correct way of extending the metric, and thus the topology, from $\mathcal D([0,t],\mathbb R)$ to $\mathcal D([0,\infty),\mathbb R)$ in the sense that convergence with respect to $d_{\infty}$ is equivalent to convergence with respect to $d_{t}$ on any compact subset of the form $[0,t]$.

The general idea to prove Theorem \ref{MainTheorem} is to first show the weak convergence of the (rescaled) process $(N_n(k))_{k\geq0}$ defined in \eqref{QueueLength} and then to deduce the convergence of the reflected process $(\phi(N_n)(k))_{k\geq0}$ exploiting the following continuous-mapping theorem. In fact, this is a special case of a general technique known as the \emph{continuous-mapping approach} (in \cite{whitt1980some} and \cite{StochasticProcess} a detailed description of this technique is given). Through the continuous-mapping approach one reduces the problem of establishing convergence of random objects to one of continuity of suitable functions:
\begin{theorem}[Continuous Mapping Theorem]
If $X_n\stackrel{\mathrm d}{\rightarrow} X$ and $f$ is continuous almost surely with respect to the distribution of $X$, then $f(X_n)\stackrel{\mathrm d}{\rightarrow} f(X)$.
\end{theorem}
Suppose we have shown that $n^{-1/3}N_n(\cdot n^{2/3})\stackrel{\mathrm d}{\rightarrow} W(\cdot)$. To prove Theorem \ref{MainTheorem} we are left to prove that the function $\phi:\mathcal{D}\rightarrow\mathcal{D}$ given by
\begin{equation}%
\phi: f(\cdot)\mapsto f(\cdot)-\inf_{y\leq\cdot}f^-(y)
\end{equation}%
is continuous almost surely with respect to the distribution of $W(\cdot)$. For this note that $\mathbb P (W(\cdot)\in \mathcal C)=1$, where $\mathcal C=\mathcal C([0,\infty),\mathbb R)\subset \mathcal D$ denotes the space of continuous functions from $[0,\infty)$ to $\mathbb R$. Then by \cite[Theorem $4.1$]{whitt1980some} and \cite[Theorem $6.1$]{whitt1980some}, $\phi$ is continuous almost surely with respect to the distribution of $W(\cdot)$.

To prove the convergence of the rescaled  process
 $(N_n(k))_{k\geq1}$ we make use of a well-known martingale functional central limit theorem (see \cite[Section 7]{MarkovProcesses}) in the special case in which the limit process is a standard Brownian motion. For a thorough overview of this important case, see \cite{whitt2007proofs}. For completeness we include the theorem that we shall apply, which is  the one-dimensional version of the one contained in \cite{whitt2007proofs}. Recall that if $M(t)$ is a square-integrable martingale with respect to a filtration $\{\mathcal F_{t}\}_{t\geq0}$, the \emph{predictable quadratic variation process} associated with $M(t)$ is the unique nondecreasing, nonnegative, predictable, integrable process $V(t)$ such that $M^2(t)-V(t)$ is a martingale with respect to $\{\mathcal F_{t}\}_{t\geq0}$.

\begin{theorem}[MFCLT]\label{FCLT}
Let $ \{\mathcal F_n\}_{n\in\mathbb N}$ be an increasing filtration and $\{\bar M_n\}_{n\in\mathbb N}$ be a sequence of continuous-time, real-valued, square-integrable martingales, each with respect to $\mathcal F _n$, such that $\bar M_n(0)=0$. Assume that $(\bar V_n(t))_{t\geq0}$, the predictable quadratic variation process associated with $(\bar M_n(t))_{t\geq0}$, and $\bar M_n(t)$ satisfy the following conditions:
\begin{itemize}[label=(alph*)]%
\item[\emph{(a)}]$\bar V_n(t)\stackrel{\mathbb P}{\longrightarrow}\sigma^2t,\qquad\forall t\in[0,\infty]$\emph{;}
\item[\emph{(b)}]$\lim_{n\rightarrow\infty}\mathbb E [\sup_{t\leq \bar t}\vert\bar V_n(t)-\bar V_n(t^-)\vert]=0,\qquad\forall\bar t\in\mathbb R^+$\emph{;}
\item[\emph{(c)}]$\lim_{n\rightarrow\infty}\mathbb E [\sup_{t\leq \bar t}\vert\bar M_n(t)-\bar M_n(t^-)\vert^2]=0,\qquad\forall\bar t\in\mathbb R^+$.
\end{itemize}%
Then, as $n\rightarrow\infty$,  $\bar M_n$ converges in distribution in $\mathcal D([0,\infty))$ to a centered Brownian motion with variance $\sigma^2 t$.
\end{theorem}

In order to apply the MFCLT, we apply a Doob-type decomposition to the process $N_n(k)$ to write it as the sum of a martingale---which will converge to the Brownian motion---and an appropriate drift term:
\begin{align}\label{QDecomposition}%
N_n(k)&=\sum_{i=1}^k\left( A_n(i)- \mathbb E[A_n(i)|\mathcal F_{i-1}]\right) +\sum_{i=1}^k( \mathbb E[A_n(i)|\mathcal F_{i-1}]-1)\notag\\
&=:M_n(k)+C_n(k),
\end{align}%
with $\{\mathcal F_{i}\}_{i\geq1}$ the filtration generated by $(A_n(k))_{k\geq1}$, i.e. $\mathcal F_i = \sigma (\{A_n(k)\}_{k=1}^i)$. Another Doob decomposition of interest is
\begin{equation}\label{MSquareDecomposition}%
M^2_n(k)=Z_n(k)+V_n(k)
\end{equation}
with $Z_n(k)$ a martingale and $V_n(k)$  the discrete-time predictable quadratic variation of the process $M_n(k)$. Note that for every fixed $n$ and $k$, $\vert M_n(k)\vert$ is bounded and thus has a finite second moment. Therefore  $V_n(k)$ exists and is given by
\begin{align}%
V_n(k)&=\sum_{i=1}^k\mathbb E[(A_n(i)-\mathbb E [A_n(i)\vert \mathcal F_{i-1}])^2\vert \mathcal F_{i-1}]\nnl
&=\sum_{i=1}^k( \mathbb E[A_n(i)^2|\mathcal F_{i-1}]- \mathbb E[A_n(i)|\mathcal F_{i-1}]^2).
\end{align}%
To see this, we rewrite
\begin{align}%
M^2_n(k)&=\sum_{i=1}^k(A_n(i)-\mathbb E[A_n(i)\vert \mathcal F_{i-1}])^2+\sum_{\substack{i,j\leq k \\ i\neq j} }(A_n(i)-\mathbb E[A_n(i)\vert \mathcal F_{i-1}])(A_n(j)-\mathbb E[A_n(j)\vert \mathcal F_{j-1}])\nnl
&=:\sum_{i=1}^k(A_n(i)-\mathbb E[A_n(i)\vert \mathcal F_{i-1}])^2+L_n(k).
\end{align}%
By developing $L_n(k)$ one can easily see that it is a martingale. The decomposition \eqref{MSquareDecomposition} then follows from:
\begin{align}%
Z_n(k)&:=\sum_{i=1}^k(A_n(i)-\mathbb E[A(i)\vert \mathcal F_{i-1}])^2-\sum_{i=1}^k\mathbb E [(A_n(i)-\mathbb E[A(i)\vert \mathcal F_{i-1}])^2\vert \mathcal F_{i-1}]+L_n(k),\nnl
V_n(k)&:=\sum_{i=1}^k\mathbb E [(A_n(i)-\mathbb E[A(i)\vert \mathcal F_{i-1}])^2\vert \mathcal F_{i-1}].
\end{align}%
Note that $Z_n(k)$ is the sum of two martingales and thus is itself a martingale.

\section{Proof of Theorem \ref{MainTheorem_exponential}}\label{sec:main_theorem_exponential}
Note that, when the arrival clocks are exponentially distributed, the conditional law of $A_n(k)$ in \eqref{eq:ArrivalsDefinitionGeneralIntro} (conditioned on $\{A_n(k-1),\sum^{k-1}_{j=1}D_k\}$) is distributed as
\begin{equation}\label{eq:ArrivalsDefinition}
\sum_{i=1}^{P_n(k)}\mathds{1}_{\{T_{i,k}\leq  D _k\}},
\end{equation}
where  $T_{i,k}\stackrel{d}{=}T_i$, which means that the clocks are re-drawn after each service, and $P_n(k):={\red \vert [n]\setminus\nu_k\vert} = n{\red -Q_n(k-1)-k}$ is the number of customers still in the population. {\red Since the dependence of $T_{i,k}$ on $k$ does not play a role in our analysis}, we will write $T_i$ instead of $T_{i,k}$. By directly defining \eqref{eq:ArrivalsDefinition} as the random variable that describes the new arrivals during the $k$-th service time, this case becomes significantly lighter on the technicalities.

Note also that the exponential distribution satisfies assumptions \eqref{eq:ArrivalDistriNearXBar} and \eqref{eq:DensityDistrNearZero}. However, Theorem \ref{MainTheorem_exponential} actually holds for all random variables $(T_i)_{i\geq1}$ satisfying 
\begin{align}%
\FT(x)=f_{\sss T}(0)\cdot x + o(x^{4/3}),\label{ArrivalDistriNearZero}
\end{align}%
which is strictly weaker than \eqref{eq:ArrivalDistriNearXBar}. 
\subsection{Supporting lemmas}\label{sec:Preparation}

We begin with an important lemma, which should be thought of as stating that, for the small-o term in the Taylor expansions we exploit, it holds that 
\begin{align*}
\mathbb E[o_{\mathbb P}(x)] = o(x).
\end{align*}%
\begin{lemma}\label{oTaylorExpansionLemma}
Let $S$, $T$ be positive random variables. Define $D:=S/n$. If \eqref{ArrivalDistriNearZero} holds for $T$ and  $\mathbb E[S^2]<\infty$, then
\begin{align}%
\mathbb E\big[\big\vert\mathbb P(T\leq D\vert D)-f_{\sss T}(0)\cdot D\big\vert\big] &= o(n^{-4/3})\label{DistributionTaylor},\\
\mathbb E\big[\big\vert D(\mathbb P(T\leq D\vert D)-f_{\sss T}(0)\cdot D)\big\vert\big] &= o(n^{-2}),\label{DistributionTaylorHalfSquare}\\
\mathbb E\big[\big\vert\mathbb P(T\leq D\vert D)-f_{\sss T}(0)\cdot D\big\vert^2\big] &= o(n^{-2}).\label{DistributionTaylorSquare}
\end{align}%
\end{lemma}

\begin{proof}
Since
\begin{align}
\mathbb P(T\leq D\vert D)=F_{\sss T}(D)=f_{\sss T}(0)\cdot D + o( S  ^{4/3}n^{-4/3}),
\end{align}
pointwise convergence trivially holds:
\begin{equation}%
n^{4/3}\vert  F_{\sss T}(D) - f_{\sss T}(0)\cdot D\vert \stackrel{\textrm{a.s.}}{\rightarrow}  0,\quad\textrm{as } n\to\infty.
\end{equation}%
Moreover, there exists a constant $C>0$ such that
\begin{equation}%
\vert F_{\sss T}(x)-f_{\sss T}(0)\cdot x\big\vert\leq C  x^{4/3},
\end{equation}%
which implies 
\begin{equation}%
n^{4/3}\vert F_{\sss T}(D)-f_{\sss T}(0)\cdot D\big\vert\stackrel{\text{a.s.}}{\leq} n^{4/3}\cdot C  (S/n)^{4/3}.
\end{equation}%
Since $\mathbb E[S^{4/3}]<  \infty$, the random variable $n^{4/3}\vert F_{\sss T}(D) - c_T\cdot D\vert$ is bounded by an integrable random variable not depending on $n$. We can then apply the Dominated Convergence Theorem to obtain \eqref{DistributionTaylor}.
Using an analogous argument we can also prove \eqref{DistributionTaylorHalfSquare} and \eqref{DistributionTaylorSquare}. Pointwise convergence is again trivial. Dominance is obtained by noting that there exists constant $C_1, C_2>0$ such that
\begin{equation}\label{eq:oTaylorExpansionDomination}%
x\vert F_{\sss T}(x) - f_{\sss T}(0)\cdot x\vert  \leq C_1 x^2\quad\text{and}\quad\vert F_{\sss T}(x) - f_{\sss T}(0)\cdot x\vert ^2 \leq C_2 x^2.
\end{equation}%
Indeed, for $x\ll 1$, $\vert F_{\sss T}(x) - f_{\sss T}(0)\cdot x\vert ^2 \leq  C_2 x^{8/3} \leq C_2 x^2$ for some $N_2>0$, and for $x\gg1$ it is enough to notice that $F_{\sss T}(x)$ is bounded. The first bound in \eqref{eq:oTaylorExpansionDomination} is obtained in the same way. Since $\mathbb E [S^2] < \infty$ by assumption, \eqref{DistributionTaylorHalfSquare} and \eqref{DistributionTaylorSquare} again follow  by dominated convergence.
\end{proof}
A random variable that plays a major role in the proof is the one associated with the arrivals when the queue would not deplete,
\begin{align}\label{APrimeDef}
A_n':=\sum_{i=1}^{n}\mathds{1}_{\{T_i\leq D\}}.
\end{align}
Note that 
\begin{align}
A_n(k)\preceq A_n',\qquad\forall k\geq1.
\end{align}

Before turning to the proof of Theorem \ref{MainTheorem} in Section \ref{sec:main_theorem_exponential}, we establish some lemmas that will prove to be useful both in the proof and in providing some insight into the behavior of the process $N_n(\cdot)$.

\begin{lemma} \label{oQueueLemma}
Let $N_n(\cdot)$ be the process defined in \eqref{QueueLength}. Then, for $k=O(n^{2/3})$, $n^{-2/3}N_n(k)\preceq G_n(k)$, where $G_n(k)$ is a random variable such that $G_n(k)\stackrel{\mathbb P}{\rightarrow}0$.
\end{lemma}
\begin{proof}
First note that 
\begin{align}\label{QueueDominationForoSmallLemma}%
N_{n}(k-1)\preceq \sum_{j=1}^{k-1}\Big(\sum_{l=1}^n \mathbf 1 _{\{T_l\leq D_j\}}-1\Big)=\sum_{j=1}^{k-1}\big(A'_n-1\big).
\end{align}%
%
By the Weak LLN for uncorrelated random variables (see e.g. \cite{klenke2007probability}) it is enough to prove that $\sup_{n\in\mathbb N} \textrm{Var}(A_n')<\infty$. Write
\begin{align}%
\Var(A'_n)^2 & = \mathbb E[(A'_n)^2] - \mathbb E[A'_n]^2\notag\\
&= \mathbb E [A'_n] + \mathbb E[\sum_{\substack{i,j\leq n \\ i\neq j}}\mathds1_{\{T_i\leq D\}}\mathds1_{\{T_j\leq D\}}]- \mathbb E [A'_n]^2.
\end{align}%
The terms $\mathbb E [A'_n]$ and  $\mathbb E [A'_n]^2$ are uniformly bounded  since
\begin{equation}\label{eq:APrimeExpectedValue}%
\mathbb E [A'_n] = 1 + \beta n^{-1/3} + o(n^{-1/3}).
\end{equation}%
Moreover,
\begin{align}\label{eq:FTimesFExpectedValue}%
 \mathbb E[\sum_{\substack{i,j\leq n \\ i\neq j}}\mathds1_{\{T_i\leq D\}}\mathds1_{\{T_j\leq D\}}]=\sum_{\substack{i,j\leq n \\ i\neq j}} \mathbb E[F_{\sss T}(D)^2] \leq C + o(1),
\end{align}%
for some $C>1$, where we have used Lemma \ref{oTaylorExpansionLemma} to bound the terms of lower order than $D^2$ appearing when Taylor expanding $F_{\sss T}(D)^2$. Both error terms in \eqref{eq:APrimeExpectedValue} and \eqref{eq:FTimesFExpectedValue} can be bounded from above by a constant independent of $n$. Therefore  $\sup_{n\in\mathbb N} \textrm{Var}(A_n')<\infty$ and the Weak LLN applies.
\end{proof}

Intuitively Lemma \ref{oQueueLemma} states that the 
process $(N_n(k))_{k\geq1}$ is of  order smaller than $n^{2/3}$. Note however that the convergence we established is not uniform in $j\leq k$, with $k = O (n^{2/3})$. We now work towards this result.

We will make use of the following well-known lemma for the order statistics of exponential random variables. Recall that $ X_{(1)},X_{(2)},\ldots,X_{(n)}$ denote the order statistics of random variables $X_1,\ldots, X_n$.
\begin{lemma}\label{OrderStatistics}
Let $E_1,\ldots,E_n$ be independent exponentially distributed random variables with mean one. Then,
\begin{align*}%
(E_{(j)})_{j=1}^n\stackrel{d}{=}\big(\sum_{s=1}^j\frac{E_s}{n-s+1}\big)_{j=1}^n.
\end{align*}%
In particular we can couple $(E_{(j)})_{j=1}^n$ with $(E_{j})_{j=1}^n$ in such a way that for all $j\leq n$,  
\begin{equation}%
E_{(j)} \stackrel{\mathrm{a.s.}}{\geq} \frac{E_1+\cdots+E_j}{n}.
\end{equation}%
\end{lemma}
See for example \cite[Section 2.5]{david2003order} for a proof of Lemma \ref{OrderStatistics}.

We next investigate the random variable $A'_n$. The following lemma states that on average, in the limit, the contribution to the queue length of arrivals of order $n^{1/3}$ or greater is negligible:
\begin{lemma}\label{LemmaSecondMomentAPrime}	
Define $A_n'$ as in \eqref{APrimeDef}. Then $A_{n}'^{\,2}$ is stochastically dominated by a family of uniformly integrable \emph{(}with respect to $n$\emph{)} random variables. In particular,
\begin{align}%
\mathbb E \big[(A_n')^2\mathds{1}_{\{ (A_n')^2 > \varepsilon n^{2/3}\}}\big]\rightarrow 0,\qquad \mathrm{as }~ n\rightarrow\infty.
\end{align}%
\end{lemma}
\begin{proof}
Observe that, if $E_i\sim\Exp(1)$, $u_i=1-\exp(-E_i)$ is a uniform random variable on $[0,1]$ so that
\begin{align}%
A_n'=\sum_{i=1}^n\mathds 1_{\{T_i\leq D\}}=\sum_{i=1}^n\mathds 1_{\{T_{(i)}\leq D\}}\stackrel{d}{=}\sum_{i=1}^n\mathds 1_{\{F^{-1}_{\sss T}(1-\exp({-E_{(i)}}))\leq D\}}.
\end{align}%
Since the function $x\mapsto F_{\sss T}^{-1}(1-\exp({-x}))$ is monotone, by using the coupling in Lemma \ref{OrderStatistics},
\begin{align}%
\sum_{i=1}^n\mathds 1_{\{F^{-1}_{\sss T}(1-\textrm{exp}(-E_{(i)}))\leq D\}}&\preceq\sum_{i=1}^n\mathds 1_{\{F^{-1}_{\sss T}(1-\textrm{exp}(-\sum_{j=1}^iE_j/n))\leq D\}}\notag\\
&\stackrel{\mathrm {a.s.}}{=}\sum_{i=1}^n\mathds 1_{\{\sum_{j=1}^iE_j\leq -n\textrm{log}(1-F_{\sss T}(D))\}}.\label{ArrivalsDomination}
\end{align}%
By \eqref{ArrivalDistriNearZero}, $F_{\sss T}(x)/x$ is bounded from above by a positive constant $K\in\mathbb R^+$, so that
\begin{align}%
\sum_{i=1}^n\mathds 1_{\{\sum_{j=1}^iE_j\leq -n\textrm{log}(1-F_{T}(D))\}}&\preceq \sum_{i=1}^n\mathds 1_{\{\sum_{j=1}^iE_j\leq -n\textrm{log}(1-K\cdot D)\}}.\label{ArrivalsDomination2}
\end{align}%
Fix $\varepsilon$ and find the respective $C$ such that $-\textrm{log}(1-x)\leq C x$ for all $0\leq x \leq 1-\varepsilon$. We do this in order to remove the dependencies from $n$. We then obtain that
\begin{align}\label{ArrivalsDominations3bis}%
A_n'&\preceq\Big(\sum_{i=1}^{\infty}\mathds 1_{\{\sum_{j=1}^iE_j\leq CK\cdot nD\}}\Big)\mathds 1_{\{K \cdot D\leq 1-\varepsilon\}}+A_n'\mathds 1_{\{K \cdot D> 1-\varepsilon\}}\notag\\
&\preceq N(CK\cdot nD) + A_n'\mathds 1_{\{K \cdot D> 1-\varepsilon\}},
\end{align}%
where
\begin{align}%
N(t,\omega):=\sum_{i=1}^{\infty}\mathds 1_{\{\sum_{j=1}^iE_j\leq t\}}(\omega)
\end{align}%
is a Poisson process with rate one. We now prove that each of the two terms in \eqref{ArrivalsDominations3bis} is a family of uniformly integrable random variables, and thus also their sum is. Since by assumption $S_k$ has a finite second moment, also $N(CK \cdot n D)$ has. Since the latter does not depend on $n$, it is uniformly integrable with respect to $n$. Moving to the second term, note that, since $A_n'\leq n$ a.s.,
\begin{align}%
\mathbb E\big[A_{n}'^{\,2}\mathds{1}_{\{K \cdot D> 1-\varepsilon\}}\big]&\leq n^2\mathbb P \Big( S_k>\frac{(1-\varepsilon) n}{K(1+\beta n^{-1/3})} \Big)\notag\\
&\leq n^2K^2(1+\beta n^{-1/3})^2\frac{\mathbb E\big[S^2_k\mathds 1 _{\{S_k>\frac{(1-\varepsilon) n}{K(1+\beta n^{-1/3})}\}}\big]}{(1-\varepsilon)^2 n^2}.\label{ArrivalsDomination4}
\end{align}%
Since $\mathbb E[ S^2]< \infty$,
\begin{align}
 \mathbb E[S^2_k\mathds 1 _{\{S_k>\frac{(1-\varepsilon) n}{K(1+\beta n^{-1/3})}\}}]\rightarrow 0,\qquad\text{as}~n\rightarrow\infty.
\end{align}
The second moments of the second term in \eqref{ArrivalsDominations3bis} converge to zero as $n$ goes to infinity, and thus $\{A_n^{'\,2}\mathbf 1_{\{K \cdot D> 1-\varepsilon\}}\}_{n\geq1}$ is a uniformly integrable family.
Therefore, $(N(CK\cdot n D)+A_n'\mathbf 1_{\{K \cdot D> 1-\varepsilon\}})^2$ is uniformly integrable. 
We have then shown that $(A_{n}'^{\,2})_{n\geq1}$ is stochastically dominated by a random variable with uniformly integrable second moments. The second claim then follows by the stochastic domination result in \eqref{StochDomLemma}.
\end{proof}

\subsection{Proof of Theorem \ref{MainTheorem_exponential}}
Recall that $N_n(k)$ can be decomposed as $N_n(k) = M_n(k) + C_n(k)$, where $M_n(k)$ is a martingale and $C_n(k)$ is the drift term. Moreover, $M_n^2(k)$ was also written as $M_n^2(k) = Z_n(k) + V_n(k)$ with $Z_n(k)$ the Doob martingale and $V_n(k)$ its drift.
The proof then consists of verifying the following conditions:

\begin{itemize}
\item[(i)] $\sup_{t\leq \bar t} \vert n^{-1/3}C_n(tn^{2/3})- \beta t + \frac{1}{2} t^2\vert\stackrel{\mathbb P}{\longrightarrow}0,\qquad \forall \bar t\in \mathbb R^+$;
\item[(ii)] $ n^{-2/3}V_n(tn^{2/3})\stackrel{\mathbb  P}{\longrightarrow}  \sigma^2  t,\qquad \forall t\in \mathbb R^+$;
\item[(iii)] $\lim_{n\rightarrow\infty}n^{-2/3} \mathbb E[\sup_{t\leq\bar t} \vert V_n(tn^{2/3})-V_n(tn^{2/3}-)\vert]=0,\qquad \forall \bar t\in \mathbb R^+$;
\item[(iv)] $\lim_{n\rightarrow\infty} n^{-2/3} \mathbb E[\sup_{t\leq\bar t} \vert M_n(tn^{2/3})-M_n(tn^{2/3}-)\vert^2]=0,\qquad \forall \bar t\in \mathbb R^+$.
\end{itemize}
Recall that $\sigma^2 := f^2_{\sss T}(0) \mathbb E[S^2_1]$.

Condition (i) implies the convergence of the drift term, while conditions (ii)-(iv) imply the convergence of the (rescaled) process $M_n(k)$ to a centered Brownian motion, by Theorem \ref{FCLT}. By standard convergence results one can then conclude that the rescaled version of the sum $C_n(k)+M_n(k)$ converges in distribution to the sum of the respective limits.

\subsubsection{Proof of (i)}

We first prove (i) and to that end, we expand the term $ \mathbb E[A_n(i)\vert \mathcal F_{i-1}]$. Recall that we have defined $\nu_i$ as the set of the customers that are no longer in the population at the beginning of the service of the $i$-th customer.  Then,  
\begin{align}%
\mathbb E[A_n(i)|\mathcal F_{i-1}]&={\red \sum_{s\notin \nu_{i}}} \mathbb E[\mathds{1}_{\{ T_s\leq  D_i\}}\vert \mathcal F_{i-1}]\notag\\
&=\sum_{s\notin \nu_{i}} \mathbb E[ \mathbb E[\mathds{1}_{\{ T_s\leq  D_i\}}\vert \mathcal F_{i-1},D_i]\vert \mathcal F_{i-1}]\notag\\
&=\sum_{s\notin \nu_{i}}\left( \mathbb E[f_{\sss T}(0) D _i\vert \mathcal F_{i-1}]+o(n^{-4/3})\right),
\end{align}%
where, in the last equality, we have used Lemma \ref{oTaylorExpansionLemma}. Since $D_i$ is independent from  $\mathcal F _{i-1}$, we obtain
\begin{align}%
\mathbb E[A_n(i)|\mathcal F_{i-1}]&={\red \sum_{s\notin \nu_{i}}}\left( \mathbb E[f_{\sss T}(0)D _i]+o(n^{-4/3})\right).
\end{align}%
In what follows, we denote by $\vert A \vert$ the cardinality of a set $A$. The summation can then be simplified to
\begin{align}%
\mathbb E[A_n(i)|\mathcal F_{i-1}]&=(n-\vert\nu_{i}\vert)\Big( f_{\sss T}(0)\mathbb E[ D _i]+o(n^{-4/3})\Big)\notag\\
&=  f_{\sss T}(0)\cdot\mathbb E[ S_{i}](1+\beta n^{-1/3})- f_{\sss T}(0)\mathbb E[ S_{i}](1+\beta n^{-1/3})\frac{|\nu_{i}|}{n}\notag\\
&\quad+(n-\vert\nu_{i}\vert)o(n^{-4/3}),\label{ArrivalsExp}
\end{align}%
where we note that $\vert\nu_i\vert$ is an integer-valued random variable. Then,
\begin{equation}\label{Arrivals}%
 \mathbb E[A_n(i)-1\vert\mathcal F_{i-1}]= f_{\sss T}(0)(1+\beta n^{-1/3})\mathbb E[S _i ]-1-f_{\sss T}(0)(1+\beta n^{-1/3})\mathbb E[S_i ]\frac{\vert\nu_{i}\vert}{n}+ o(n^{-1/3}),
\end{equation}%
By \eqref{eq:CriticalityHypBeta}, $f_{\sss T}(0)(1+\beta n^{-1/3})\mathbb E[S_i ] = 1 + \beta n^{-1/3} + o(n^{-1/3})$, so that \eqref{Arrivals} yields
\begin{align}\label{ArrivalsSimplified}%
\mathbb E[A_n(i)-1\vert\mathcal F_{i-1}]&=\beta n^{-1/3}-\frac{\vert\nu_{i}\vert}{n}\big(1+O(n^{-1/3})\big)  +o(n^{-1/3}).
\end{align}%

Note that the service times are independent from the history of the system, thus the conditioning has no effect. Since $\
\vert\nu_{i}\vert = {\red  i + } N_n(i-1)$, the drift term in the decomposition of $N_n(k)$ can now be written as
\begin{equation}\label{DriftTerm}%
C_n(k)=k\beta n^{-1/3}-\Big( \frac{k^2 {\red  + } k}{2}+\sum_{i=1}^kN_n(i-1)\Big)\Big(\frac{ 1}{n}+O(n^{-4/3})\Big)+k o\bigl(n^{-1/3}\bigr).
\end{equation}%

The term $-\sum_{i=1}^k N_n(i-1)$ in \eqref{DriftTerm} accounts for the fact that the customers already in the queue cannot rejoin it. This term converges to zero as $n$ tends to infinity (after appropriate scaling) by the following result. In fact, Lemma \ref{oQueueLemmaUnif} proves that the process $n^{-2/3}N_n(j)$ tends to zero in probability, uniformly on $j\leq an^{2/3}$, a significantly stronger statement than Lemma \ref{oQueueLemma}:
\begin{lemma}\label{oQueueLemmaUnif}
Let $(N_n(k))_{k\geq0}$ be the process defined in \eqref{QueueLength} and $a\in\mathbb R^+$. Then
\begin{align}%
n^{-2/3}\sup_{j\leq an^{2/3}}\vert N_n(j)\vert \stackrel{\mathbb P}{\longrightarrow}0.
\end{align}%
\end{lemma}

\begin{proof}
Recall that
$
N_n(j)=M_n(j)+C_n(j).
$
Then
\begin{align}\label{UniformQConvLemmaEstimate}%
\mathbb P \big(n^{-2/3}\sup_{j\leq an^{2/3}}\vert N_n(j)\vert\geq 2\varepsilon\big) &\leq \mathbb P \big(n^{-2/3}\sup_{j\leq an^{2/3}}\vert M_n(j)\vert\geq \varepsilon\big)\notag\\
&\quad+\mathbb P \big(n^{-2/3}\sup_{j\leq an^{2/3}}\vert C_n(j)\vert\geq \varepsilon\big).
\end{align}%
We can separately bound the first and second terms. Applying Doob's inequality to the martingale $M_n(\cdot)$ gives
\begin{equation}
\mathbb P \big(n^{-2/3}\sup_{j\leq an^{2/3}}\vert M_n(j)\vert\geq \varepsilon\big)\leq \frac{\mathbb E\big[M_n^2(an^{2/3})\big]}{(\varepsilon n^{2/3})^2}.
\end{equation}
By \eqref{MSquareDecomposition}, $\mathbb E[M_n^2(k)] = \mathbb E [V_n(k)]$. Expanding this term gives
\begin{align}%
\mathbb E[V_n(k)]&=\mathbb E \Big[\sum_{i=1}^k\big(\mathbb E[A_n(i)^2\vert \mathcal F_{i-1}]-\mathbb E[A_n(i)\vert\mathcal F_{i-1}]^2\big)\Big]\notag\\
&\leq \mathbb E\Big[\sum_{i=1}^k\mathbb E[A_n(i)^2\vert\mathcal F_{i-1}]\Big]\leq k\cdot\mathbb E [A_n'^{\,2}],
\end{align}%
where $A_n'$ is defined as in \eqref{APrimeDef}. By rescaling we get
\begin{align}\label{BEstimate}%
n^{-4/3}\mathbb E[V_n(an^{2/3})]\leq an^{-2/3}\mathbb E[A_n'^{\,2}].
\end{align}%
Since $A_n'$ has a finite second moment as was shown in \eqref{eq:FTimesFExpectedValue}, the right-most term in \eqref{BEstimate} tends to zero as $n$ tends to infinity. 

For the second term in \eqref{UniformQConvLemmaEstimate} we make use of the decomposition of the drift term in \eqref{DriftTerm}. From there we obtain
\begin{align}\label{DriftDoubleBound}%
-(n^{-1}+O(n^{-4/3}))\sum_{i=1}^k\Big({\red i +}N_n(i-1)\Big)+k o\bigl(n^{-1/3}\bigr)\leq C_n(k)\leq k\beta n^{-1/3}+k o\bigl(n^{-1/3}\bigr),
\end{align}%
and thus,   
\begin{align}\label{DriftMaxBound}%
\sup_{j\leq an^{2/3}}\vert C_n(j)\vert &\stackrel{\textrm{a.s.}}{\leq} (an^{2/3} \beta n^{-1/3} +an^{2/3}o(n^{-1/3}))\notag\\ &\quad \vee \Big((n^{-1}+O(n^{-4/3}))\sum_{i=1}^{an^{2/3}}\Big({\red i+}N_n(i-1)\Big)+an^{2/3} o\bigl(n^{-1/3}\bigr)\Big),
\end{align}%
since the two sides of the bound \eqref{DriftDoubleBound} are monotone functions of $k$. By rescaling the first term by $n^{-2/3}$ we obtain
$a\beta n^{-1/3}+a\cdot o(n^{-1/3})$, which tends to zero almost surely as $n$ goes to infinity.
The second term in \eqref{DriftMaxBound} needs more attention. Notice that the function $i\mapsto i+N_n(i)$ is nonnegative and nondecreasing. Thus, isolating the sum, we can bound it through its final term:
\begin{align}%
\sum_{i=1}^{an^{2/3}}\Big({\red i+}N_n(i-1)\Big)\stackrel{\textrm{a.s.}}{\leq} an^{2/3}(an^{2/3} + N_n(an^{2/3}-1))
\end{align}%
Rescaling by $n^{-2/3}$ we get for the second term in \eqref{DriftMaxBound} 
\begin{align}\label{DriftMaxBoundSecondTerm}%
n^{-2/3} \Big((n^{-1}+O(n^{-4/3}))\sum_{i=1}^{an^{2/3}}\Big({\red i+}N_n(i-1)\Big)+an^{2/3} o\bigl(n^{-1/3}\bigr)\Big)\notag\\
\stackrel{\textrm{a.s.}}{\leq} a\big(n^{-1}+O\big(n^{-4/3}\big)\big)(an^{2/3} + N_n(an^{2/3}-1))+a\cdot o(n^{-1/3})\qquad 
\end{align}%
The $O(n^{-4/3})$ term is of lower order than $n^{-1}$ and can be ignored. By Lemma \ref{oQueueLemma}, the right-hand side of \eqref{DriftMaxBoundSecondTerm} tends to zero in probability as $n$ tends to infinity and this concludes the proof that the second term in \eqref{UniformQConvLemmaEstimate} converges in probability to zero.
\end{proof}
Substituting $k=tn^{2/3}$ into \eqref{DriftTerm} and multiplying by $n^{-1/3}$ yields
\begin{align}\label{eq:DriftTermLastComputation}%
n^{-1/3}C_n(tn^{2/3})&=\beta t -\Big( \frac{t^2 {\red  + }tn^{-2/3}}{2}+\sum_{i=1}^kN_n(i-1)\Big)\Big(1+O(n^{-1/3})\Big)+o(1).
\end{align}%
Both the small-o and the big-O terms in \eqref{eq:DriftTermLastComputation} are independent of $t$ (or, equivalently, of $k$). Indeed, the small-o term originates from Lemma \ref{oTaylorExpansionLemma} (and is therefore independent of $k$) and the big-O term was introduced in \eqref{ArrivalsSimplified} and depends only on $n$ and $\beta$. Therefore, the convergence of $n^{-1/3}C_n(tn^{2/3})$ is uniform in $t\leq \bar t$ for fixed $\bar t$ as required, and this concludes the proof of Lemma \ref{oQueueLemmaUnif} and thus of (i).

\subsubsection{Proof of (ii)}
In order to prove (ii) we first compute 
\begin{align}%
\mathbb E[A_n(i)^2|\mathcal F_{i-1}] &=  \mathbb E\Big[{\red \sum_{j\notin\nu_{i}}}\mathds{1}_{\{T_j\leq  D_i\}}^2 + {\red \sum_{\substack{l\neq m\\l,m\notin\nu_{i} }}}\mathds{1}_{\{T_l\leq  D_i\}}\mathds{1}_{\{T_m\leq  D_i\}}\Big\vert \mathcal F_{i-1}\Big]\notag\\
&= \mathbb E[A_n(i)\vert \mathcal F_{i-1}]+ \mathbb E[\sum_{\substack{l\neq m\\l,m\notin\nu_{i} }}\mathds{1}_{\{T_l\leq  D_i\}}\mathds{1}_{\{T_m\leq  D_i\}}\vert \mathcal F_{i-1}],
\end{align}%
which yields
\begin{align}\label{ArrivalsSecondMomentBegin}%
\mathbb E[A_n(i)^2|\mathcal F_{i-1}] - \mathbb E[A_n(i)\vert \mathcal F_{i-1}]&= {\red \sum_{\substack{l\neq m\\l,m\notin\nu_{i} }}} \mathbb E[ \mathbb E[\mathds{1}_{\{T_l\leq  D_i\}}\mathds{1}_{\{T_m\leq  D_i\}}\vert \mathcal F_{i-1}, D_i]\vert \mathcal F_{i-1}]\notag\\
&=\sum_{\substack{l\neq m\\l,m\notin\nu_{i} }} \mathbb E[\mathbb P(T_l\leq  D_i\vert  D_i)\mathbb P(T_m\leq  D_i\vert  D_i)\vert \mathcal F_{i-1}]\notag\\
&= \sum_{\substack{l\neq m\\l,m\notin\nu_{i} }} \mathbb E[F_{\sss T}(D_i)^2\vert \mathcal F_{i-1}].
\end{align}%
Using \eqref{ArrivalDistriNearZero} we can rewrite the summation term as
\begin{align}\label{eq:CondExpectationDistrFunctProduct}%
\mathbb E[F_{\sss T}(D_i)^2\vert \mathcal F_{i-1}] = f_{\sss T}(0)^2 \mathbb E[D_i^2] + 2f_{\sss T}(0)\mathbb E[D_i(F_{\sss T}(D_i) - f_{\sss T}(0) D_i)] +\mathbb E[(F_{\sss T} (D_i) -f_{\sss T}(0) D_i)^2].
\end{align}%
The second and third terms are  $o(n^{-2})$ by Lemma \ref{oTaylorExpansionLemma}. Therefore, we rewrite \eqref{ArrivalsSecondMomentBegin} as 
%
%
\begin{align}\label{ArrivalsSecondMoment}%
{\red \sum_{\substack{l\neq m\\l,m\notin\nu_{i} }}}\Big(\frac{f_{\sss T}(0)^2}{n^2}\big(1+\beta n^{-1/3}\big)^2 \mathbb E[S^2]+o(n^{-2})\Big) =  \frac{{\red \vert\Xi_{i} \vert} f_{\sss T}(0)^2}{n^2} \big(1+\beta n^{-1/3}\big)^2\mathbb E[S^2] +o (1).
\end{align}%
Note that the cardinality of the set $\Xi_{i}:=\{(l,m):~l\neq m,~l,m\notin\nu_{i}\}$ is 
\begin{align}%
\vert \Xi_{i}\vert =(n-N_{n}(i-1){\red -i})^2-(n-N_{n}(i-1){\red -i}),
\end{align}%
thus of the order $n^2$. Then, for $k=O(n^{2/3})$,
\begin{align}\label{eq:QuadraticVariationExplicit}%
C_n(k)&=\sum_{i=1}^k\left( \mathbb E[A_n(i)^2|\mathcal F_{i-1}]- \mathbb E[A_n(i)\vert \mathcal F_{i-1}]^2\right) \notag\\
&=\sum_{i=1}^k \Big( \frac{{\red \vert \Xi_{i} \vert} f_{\sss T}(0)^2}{n^2}\big(1+\beta n^{-1/3}\big)^2 \mathbb E[S^2]+\mathbb E[A_n(i)\vert \mathcal F_{i-1}]- \mathbb E[A_n(i)\vert \mathcal F_{i-1}]^2 +o(1)\Big).
\end{align}%
Using \eqref{ArrivalsSimplified}, together with the observation that $\vert\nu_{i}\vert/n = O_{\mathbb P}(n^{-1/3})$ uniformly for $i=O(n^{2/3})$,
\begin{align}%
C_n(k)&=\sum_{i=1}^k \Big(\frac{{\red \vert \Xi_{i} \vert} f_{\sss T}(0)^2}{n^2}\big(1+\beta n^{-1/3}\big)^2\mathbb E[S^2] +  O_{\mathbb P}(n^{-1/3})+ o(1) \Big)\nnl
&=\sum_{i=1}^k\frac{(n-N_{n}(i-1)-i)^2-(n-N_{n}(i-1)-i)}{n^2}f_{\sss T}(0)^2\big(1+\beta n^{-1/3}\big)^2\mathbb E[S^2]\nnl
&\quad + O_{\mathbb P}(kn^{-1/3}) + o(k).
\end{align}%
We then split the term inside the summation to isolate the contribution of the process $N_n$, and write
\begin{align}%
V_n(k)&=\sum_{i=1}^k\frac{(n{\red -i})^2-(n{\red -i})}{n^2}f_{\sss T}(0)^2\big(1+\beta n^{-1/3}\big)^2 \mathbb E[S^2]\notag\\
&\quad+\sum_{i=1}^k\frac{N_{n}(i-1)(N_n(i-1)-2(n-i)+1)}{n^2}f_{\sss T}(0)^2\big(1+\beta n^{-1/3}\big)^2\mathbb E[S^2]\notag\\
&\quad+O{_{\mathbb P}}(kn^{-1/3}) + o(k).\label{QuadraticDrift}
\end{align}%
By Lemma \ref{oQueueLemmaUnif} the second term (accounting for the process history) tends to zero in probability when rescaled appropriately.  An elementary computation shows that
\begin{align}%
\frac{f_{\sss T}(0)^2\mathbb E[S^2]}{n^2}\sum_{l=n-k}^{n-1}(l^2-l)&=\frac{\sigma^2}{n^2}(\frac{2}{3}k+k^2-\frac{1}{3}k^3-2kn-k^2n+kn^2)\notag\\
&=\sigma^2k + O(k^2n^{-1}).
\end{align}%
The remaining terms were omitted because they are of order smaller than $O(k^2n^{-1})$ when $k = sn^{2/3}$. When rescaling space and time appropriately in \eqref{QuadraticDrift} we finally obtain that 
\begin{align}%
n^{-2/3}V_n(tn^{2/3})\stackrel{\mathbb P}{\rightarrow} \sigma ^2 t,
\end{align}%
as required. This completes the proof of (ii).
\subsubsection{Proof of (iii)}
It is well known that $(V_n(k))_{k\in\mathbb N}$ in \eqref{MSquareDecomposition} is almost surely increasing. For (iii), we need to estimate the largest  possible jump
\begin{equation}\label{QuadraticVariationJump}%
n^{-2/3}\vert V_n(k+1)-V_n(k)\vert = n^{-2/3}\vert  \mathbb E[A_n(k+1)^2\vert \mathcal F_{k}]- \mathbb E[A_n(k+1)\vert  \mathcal F_{k}]^2\vert,
\end{equation}%
with $k=O(n^{2/3})$. In order to apply the Dominated Convergence Theorem, we first prove that it converges pointwise to zero, and then that it is (almost surely) bounded. The jump \eqref{QuadraticVariationJump} has already been implicitly computed as the term in the summation in \eqref{QuadraticDrift} and it takes the form
\begin{align}\label{ArrivalsEstimate}%
& n^{-2/3} \vert  \mathbb E[A_n(k+1)^2\vert \mathcal F_{k}] - \mathbb E[A_n(k+1)\vert \mathcal F_{k}]^2\vert\\
&= n^{-\frac{2}{3}}\Big\vert f_{\sss T}(0)^2 \mathbb E[S^2]\bigg(\frac{(n - k -1)^2 - (n - k - 1)}{n^2}  + \frac{N_{n}(k)(N_n(k) - 2(n - k -1) + 1)}{n^2}\bigg) \nnl
&\qquad+ O_{\mathbb P}(n^{-\frac{1}{3}})\Big\vert\notag
%
\end{align}%
%
The $O_{\mathbb P}(n^{-1/3})$ term is a byproduct of $\mathbb E[A_n(k)\vert\mathcal F_{k-1}] - \mathbb E[A_n(k)\vert\mathcal F_{k-1}]^2$ (as computed in  \eqref{eq:QuadraticVariationExplicit} and following calculations). We now compute it precisely by using the exact expression for $\mathbb E[ A_n(k)\vert\mathcal F_{k-1}]$ found in \eqref{ArrivalsSimplified}, as follows
\begin{align}\label{eq:arrivals_estimate_second}%
\big\vert\mathbb E[A_n(k)\vert\mathcal F_{k-1}] - \mathbb E[A_n(k)\vert\mathcal F_{k-1}]^2 \big\vert &\stackrel{\textrm{a.s.}}{\leq}  \beta n^{-1/3} + 2 \vert \nu_{k}\vert n^{-1} + 2\beta \vert \nu_{k}\vert n^{-4/3} + o(n^{-1/3}).
\end{align}%
The right-hand side of \eqref{eq:arrivals_estimate_second} is bounded by $3$ for all sufficiently large values of $n$, uniformly in $k\leq tn^{2/3}$ (recall that $\vert\nu_{k}\vert\leq n$). Then, by inserting this into \eqref{ArrivalsEstimate},
\begin{align}%
n^{-2/3}\vert  &\mathbb E[A_n(k+1)^2\vert \mathcal F_{k}] - \mathbb E[A_n(k+1)\vert  \mathcal F_{k}]^2\vert \notag\\
&\stackrel{\textrm{a.s.}}{\leq} n^{-2/3}\bigg\vert  \sigma ^2\bigg(\frac{(n-k-1)^2-(n-k-1)}{n^2}  +\frac{N_{n}(k)(N_n(k)-2(n-k-1)+1)}{n^2}\bigg) + 3 \bigg\vert.
\end{align}%
Since $\vert N_n(k)\vert\leq n$, we can find a constant $C$ such that, uniformly in $k\leq tn^{2/3}$, 
\begin{align}%
n^{-2/3}\vert  &\mathbb E[A_n(k+1)^2\vert \mathcal F_{k}] - \mathbb E[A_n(k+1)\vert\mathcal F_{k}]^2 \stackrel{\textrm{a.s.}}{\leq}  n^{-2/3}(3 + \sigma ^2 C).
\end{align}%
This gives both the assumptions in the Dominated Convergence Theorem (almost sure convergence and dominance with an integrable random variable) and therefore concludes the proof of (iii).
\qed
\subsubsection{Proof of (iv)}

We tackle (iv) through a coupling argument. First observe that
\begin{align}\label{MartingaleTriangleIneq}%
n^{-2/3}\mathbb  E[\sup_{t\leq\bar t} \vert M_n(tn^{2/3})&-M_n(tn^{2/3}-)\vert ^2] 
= n^{-2/3}\mathbb  E[\sup_{k\leq\bar tn^{2/3}}\vert A_n(k)-\mathbb  E[A_n(k)\vert \mathcal F_{k-1}]  \vert ^2]\notag\\
&\leq n^{-2/3}\mathbb E[\sup_{k\leq\bar tn^{2/3}}\vert A_n(k) \vert ^2]+n^{-2/3}\mathbb  E[\sup_{k\leq\bar tn^{2/3}}\vert \mathbb E[A_n(k)\vert \mathcal F_{k-1}]  \vert ^2].
\end{align}%
The second term in \eqref{MartingaleTriangleIneq} is easily estimated. Using the calculations on $ \mathbb E[A_n(k)\vert\mathcal F_{k-1}]$ done in \eqref{ArrivalsSimplified}, and that the second term there is negative, yields
\begin{equation*}%
 0\stackrel{\mathrm{a.s.}}{\leq}\mathbb E[A_n(k)\vert \mathcal F_{k-1}]\stackrel{\mathrm{a.s.}}{\leq} 1+O(n^{-1/3}).
\end{equation*}%
For the first term, we will use a coupling argument. 
For $\varepsilon>0$, we split
\begin{align}\label{ArrivalsDomination6}%
\mathbb E\Big[\sup_{k\leq\bar tn^{2/3}}\vert A_n(k) \vert ^2\Big]&=\mathbb E\Big[\sup_{k\leq\bar tn^{2/3}} A_n(k)^2\mathds{1}_{\{\sup_{k\leq\bar tn^{2/3}} A_n(k)^2\leq \varepsilon n^{2/3}\}}\Big]\notag\\
&\quad+\mathbb E\Big[\sup_{k\leq\bar tn^{2/3}} A_n(k)  ^2\mathds{1}_{\{\sup_{k\leq\bar tn^{2/3}} A(k)^2 > \varepsilon n^{2/3}\}}\Big].
\end{align}%
After multiplying \eqref{ArrivalsDomination6} by $n^{-2/3}$, the first term can  trivially be bounded by $\varepsilon$, while for the second term we estimate
\begin{align}\label{ArrivalsDomination5}%
\mathbb E\Big[\sup_{k\leq\bar tn^{2/3}} A_n(k)  ^2\mathds{1}_{\{\sup_{k\leq\bar tn^{2/3}} A_n(k)^2 > \varepsilon n^{2/3}\}}\Big] &\leq \sum_{k=1}^{\bar t n^{2/3}} \mathbb E[A_n(k)^2\mathds{1}_{\{ A_n(k)^2 > \varepsilon n^{2/3}\}}]\notag\\
&\leq\sum_{k=1}^{\bar t n^{2/3}}\mathbb E \big[{A_n'}^2\mathds{1}_{\{ {A'_n}^2 > \varepsilon n^{2/3}\}}\big]\notag\\
&= \bar t n^{2/3}\mathbb E \big[{A_n'}^2\mathds{1}_{\{ {A_n'}^2 > \varepsilon n^{2/3}\}}\big],
\end{align}%
where we have used the stochastic domination in \eqref{APrimeDef}. By Lemma \ref{LemmaSecondMomentAPrime}, $\mathbb E \big[{A_n'}^2\mathds{1}_{\{ {A_n'}^2 > \varepsilon n^{2/3}\}}\big]\rightarrow0$ and thus
\begin{align*}
n^{-2/3}\mathbb E\Big[\sup_{k\leq\bar tn^{2/3}} A_n(k)  ^2\mathds{1}_{\{\sup_{k\leq\bar tn^{2/3}} A_n(k)^2 > \varepsilon n^{2/3}\}}\Big]\rightarrow0.
\end{align*}
This concludes the proof of (iv).
\qed

\section{Proof of Theorem \ref{MainTheorem}}\label{sec:GeneralArrivals}
We now generalize Theorem \ref{MainTheorem_exponential} in the direction of allowing for non-exponential arrival times. Consider a family of arrival times $(T_i)_{i= 1}^n$ with distribution function $\FT(\cdot)$ and density function $\fT(\cdot)$. Denote their order statistics by $T_{(1)}\leq T_{(2)}\leq\ldots\leq T_{(n)}$. As before, let $(D_j)_{j\geq1}$ be the rescaled service times with $D_j := S_j(1+\beta n^{-1/3})/n$. Then, at time $k$, the number of arrivals during every service is defined as 
\begin{align}\label{eq:NewArrivalsDefinition}
A_n(k):={\red \sum_{i\nin\nu_{k}}} \mathds 1_{\{\sum_{j=1}^{k-1}D_j< {\red T_{i}}\leq\sum_{j=1}^kD_j\}}.
\end{align}%
%
We assume that $F_{\sss T}(\cdot)$ can be Taylor expanded in a neighborhood of every point, as in  \eqref{eq:ArrivalDistriNearXBar}, and that the density $f_{\sss T}(\cdot)$ can be Taylor expanded in a neighborhood of zero, as in  \eqref{eq:DensityDistrNearZero}.

\subsection{Supporting lemmas}

For readability, throughout this section we will use the notation $\Sigma_j:=\sum_{i=1}^j D_i$.
In this section we present several useful lemmas that are appropriate adaptations of lemmas in Section \ref{sec:Preparation}. We provide the proof only when it is substantially different from their counterparts in Section \ref{sec:Preparation}.
{\red It is often useful to bound $A_n(k)$ from above to avoid the difficulties of dealing with the complicated set $\nu_{k}$. Mimicking \eqref{APrimeDef}, we define $A'_n(k)$ as
\begin{equation}%
A'_n(k):= \sum_{i=1}^n \mathds 1_{\{\Sigma_{k-1}\leq T_i\leq \Sigma_k\}}.
\end{equation}%
Note in particular that 
\begin{equation}%
A_n(k) \stackrel{\mathrm{a.s.}}{\leq} A'_n(k).
\end{equation}%
}
\begin{lemma}\label{lem:oTaylorExpansionDistributionGeneralArrivals}
In the present setting and for $k = O(n^{2/3})$,  the following holds:
\begin{align}%
\mathbb E \big[\vert F_{\sss T} (\Sigma_k) - F_{\sss T}(\Sigma_{k-1}) - f_{\sss T} (\Sigma_{k-1})\cdot D_k\vert  \big\vert \Sigma_{k-1}\big] &= o_{\mathbb P}(n^{-4/3}),\label{lem:DistributionTaylorGeneralArrivals}\\
\mathbb E \big[\vert D_k( F_{\sss T} (\Sigma_k) - F_{\sss T}(\Sigma_{k-1}) - f_{\sss T} (\Sigma_{k-1})\cdot D_k )\vert\big\vert \Sigma_{k-1}\big] &= o_{\mathbb P}(n^{-2}),\label{lem:DistributionTaylorHalfSquareGeneralArrivals}\\
\mathbb E \big[{\vert F_{\sss T} (\Sigma_k) - F_{\sss T}(\Sigma_{k-1}) - f_{\sss T} (\Sigma_{k-1})\cdot D_k\vert}^2 \big\vert \Sigma_{k-1}\big] &= o_{\mathbb P}(n^{-2}),\label{lem:DistributionTaylorSquareGeneralArrivals}
\end{align}%
Moreover, all the statements of convergence hold uniformly for $k = O(n^{2/3})$.
\end{lemma}
\begin{proof}
We give the proof for \eqref{lem:DistributionTaylorHalfSquareGeneralArrivals}, the rest can be shown in an analogous way. Note that, by our assumptions on $F_{\sss T}$,
\begin{align}%
&\mathbb E[n^2\vert D_k\cdot( F_{\sss T} (\Sigma_k) - F_{\sss T}(\Sigma_{k-1}) - f_{\sss T} (\Sigma_{k-1}) D_k )\vert\big\vert\Sigma_{k-1}] \nnl
&\quad\leq \sup_{y\leq Cn^{-1/3}}n^2 \mathbb E[\vert F_{\sss T}(y+D_k)-F_{\sss T}(y) - f_{\sss T}(y) D_k\vert\cdot D_k]\nnl
&\quad\leq n^2\mathbb E[\sup_{y\leq Cn^{-1/3}}\vert F_{\sss T}(y+D_k)-F_{\sss T}(y) - f_{\sss T}(y) D_k\vert\cdot D_k]
\end{align}%
since with high probability $\Sigma_{k-1}\leq Cn^{-1/3}$ for some $C>0$. The right term  tends to zero by the Dominated Convergence Theorem and assumption \eqref{eq:ArrivalDistrUnifError}, and this immediately implies \eqref{lem:DistributionTaylorHalfSquareGeneralArrivals}.
\end{proof}
%
%
As has already been seen for the exponential arrivals case, the 
 process $N_n(k)_{k\geq1}$ is roughly of the order $n^{1/3}$ around time $tn^{2/3}$. The following lemma is the counterpart for general arrivals of Lemma \ref{oQueueLemma} and we prove it in a similar fashion:
\begin{lemma}\label{lem:QueueLengthoSmallGeneralArrivals}%
Let $N_n(k)$ be the 
process defined in \eqref{QueueLength}. Then $n^{-2/3}N_n(k)\preceq G_n(k)$, where $G_n(k)$ is a random variable such that $G_n(k)\stackrel{\mathbb P}{\rightarrow}0$ for $k= O(n^{2/3})$.
\end{lemma}%
\begin{proof}
For simplicity we only consider $k = tn^{2/3}$. Fix an arbitrary $\varepsilon>0$. Then,
\begin{align}\label{eq:FirstStepQueueLengthoSmallGeneralArrivalsProof}%
&\mathbb P\Big(n^{-2/3}\sum_{j=1}^{tn^{2/3}}{\red A_n(j)} - n^{-2/3}t\geq \varepsilon\Big)\notag\\
&\quad\leq \mathbb P\Big(n^{-2/3}\sum_{j=1}^{tn^{2/3}}{\red A'_n(j)} - n^{-2/3}t\geq \varepsilon\Big)\notag\\
&\quad\leq\mathbb P\Big(n^{-2/3}\big\vert\sum_{j=1}^{tn^{2/3}}{\red A'_n(j)} - n^{-2/3}t\big\vert\geq \varepsilon\Big)\notag\\
&\quad = \mathbb P\Big(n^{-2/3}\big\vert\sum_{l=1}^n \mathds 1_{\{T_{l}\leq \Sigma_{tn^{2/3}}\}}-n^{-2/3}t\big\vert\geq \varepsilon\Big)\notag\\
&\quad \leq \frac{\mathbb E[(\sum_{l=1}^n \mathds 1_{\{T_l \leq \Sigma_{tn^{2/3}}\}})^2\mathds 1_{\{\vert\sum_{l=1}^n \mathds 1_{\{T_{l}\leq \Sigma_{tn^{2/3}}\}}-n^{-2/3}t\vert\geq \varepsilon n^{2/3}\}}]}{n^{4/3}\varepsilon ^2}.
\end{align}%
We are left to bound the expected value in \eqref{eq:FirstStepQueueLengthoSmallGeneralArrivalsProof}. To do so, we define the event $\mathcal E_n:=\{\vert\sum_{l=1}^n \mathds 1_{\{T_{l}\leq \Sigma_{tn^{2/3}}\}}-n^{-2/3}t\vert\geq \varepsilon n^{2/3}\}$ and write
\begin{align}\label{eq:SecondStepQueueLengthoSmallGeneralArrivalsProof}%
\mathbb E[(\sum_{l=1}^n \mathds 1_{\{T_l \leq \Sigma_{tn^{2/3}}\}})^2\mathds 1_{\mathcal E_n}] &\leq n + \mathbb E[ \sum_{h\neq k}\mathds 1_{\{T_h\leq \Sigma_{tn^{2/3}}\}}\mathds 1_{\{T_k\leq \Sigma_{tn^{2/3}}\}}\mathds 1_{\mathcal E_n}]\notag\\
&\leq n + n^2 \mathbb E[ F_{\sss T}(\Sigma_{tn^{2/3}})^2\mathds 1_{\mathcal E_n}]\leq n + Cn^2\mathbb E[(\Sigma_{tn^{2/3}})^2\mathds 1_{\mathcal E_n}]\notag\\
&\leq n+Cn^{4/3}\mathbb E[S^2\mathds 1_{\mathcal E_n}],
\end{align}%
for a large constant $C$. The last inequality in \eqref{eq:SecondStepQueueLengthoSmallGeneralArrivalsProof} was obtained by an application of the Cauchy-Schwarz inequality. Since $\mathbb P(\mathcal E_n)\rightarrow 0$ and $\mathbb E[S^2] < \infty$, by inserting \eqref{eq:SecondStepQueueLengthoSmallGeneralArrivalsProof} into \eqref{eq:FirstStepQueueLengthoSmallGeneralArrivalsProof} and the Dominated Convergence Theorem, we get the desired convergence.
\end{proof}

The following lemma reduces the task to estimating quantities involving order statistics of exponentials, which are possibly complicated objects, to the one of dealing with a Poisson process, which is much simpler. It will be used later for proving Lemma \ref{LemmaSecondMomentAPrimeGeneralArrivals}, which is the equivalent in this setting of Lemma \ref{LemmaSecondMomentAPrime}.
\begin{lemma}\label{OrderStatisticsExponentialDominatedByPoissonStatement}
Consider the order statistics $(E_{(i)})_{i=1}^{n-\upsilon}$ $(\upsilon\leq n)$ of $n$ exponential unit mean random variables. Define $\vert \Upsilon^{(n-\upsilon)}_{(0,c)} \vert $ as the cardinality of the set $\Upsilon^{(n-\upsilon)}_{(0,c)}:=\{j\in[n-\upsilon]:E_{(j)}\in (0,c)/n\}$. Then,
\begin{align}\label{eq:OrderStatisticsExponentialDominatedByPoissonStatement}%
\vert\Upsilon^{(n-\upsilon)}_{(0,c)}\vert\preceq N\Big(\frac{n-\upsilon}{n} c\Big),
\end{align}%
where $N(t)$ is a Poisson process with unit rate.
\end{lemma}
\begin{proof} The statement is a consequence of Lemma \ref{OrderStatistics}. Fix $j\in\{1,\ldots,n-\upsilon\}$. By definition of stochastic domination
\begin{align}\label{eq:OrderStatisticsExponentialDominatedByPoissonProof}%
\mathbb P (E_{(j)} \leq c/n)\leq \mathbb P \Big( \frac{\sum_{i=1} ^jE_i}{n-\upsilon}\leq \frac{c}{n}\Big)\leq \mathbb P\Big(\Pi_j\leq \frac{n-\upsilon}{n}c\Big), 
\end{align}%
where $\Pi_j$ is the $j$-th point of a Poisson process with rate one. The computation in \eqref{eq:OrderStatisticsExponentialDominatedByPoissonProof} intuitively means that there are more Poisson points in an interval of length $(n-\upsilon)\frac{c}{n}$ than order statistics in an interval of length $\frac{c}{n}$. This implies \eqref{eq:OrderStatisticsExponentialDominatedByPoissonStatement}.
%
%
\end{proof}
Since
\begin{align*}%
N\big(\frac{n-\upsilon}{n}\cdot c\big)\preceq N(c),\qquad\forall \upsilon\leq n,
\end{align*}%
it follows from \eqref{eq:OrderStatisticsExponentialDominatedByPoissonStatement} that
\begin{align}%
\vert\Upsilon^{(n-\upsilon)}_{(0,c)}\vert\preceq N(c),\qquad\forall \upsilon\leq n.
\end{align}%
\begin{corollary}\label{cor:OrderStatisticsExponentialDominatedByPoissonStatementCorollary}
Under the same assumptions as in \emph{Lemma \ref{OrderStatisticsExponentialDominatedByPoissonStatement},}
\begin{align}\label{eq:OrderStatisticsExponentialDominatedByPoissonStatementCorollary}%
\Upsilon^{(n)}_{(a,b)}\preceq N(b-a).
\end{align}%
\end{corollary}
\begin{proof}
By Lemma \ref{OrderStatisticsExponentialDominatedByPoissonStatement},
\begin{align}\label{eq:OrderStatisticsExponentialDominatedPoissonCorollaryTwo}%
\mathbb P (N(b-a)\leq x)\leq \mathbb P\big(\vert\Upsilon^{(n-\upsilon)}_{(0,b-a)}\vert\leq x\big). 
\end{align}%
Note that, by the memoryless property,
\begin{align}\label{eq:OrderStatisticsExponentialDominatedPoissonCorollaryOne}%
\mathbb P\big(\vert\Upsilon^{(n-\upsilon)}_{(0,b-a)}\vert\leq x\big)\stackrel{\mathrm{a.s.}}{=}\mathbb P \big(\vert\Upsilon^{(n)}_{(a,b)}\vert\leq x ~\big\vert~ \vert\Upsilon^{(n)}_{(0,a)}\vert = \upsilon\big)  .
\end{align}%
Since the left side of \eqref{eq:OrderStatisticsExponentialDominatedPoissonCorollaryTwo} does not depend on $\upsilon$, by combining  \eqref{eq:OrderStatisticsExponentialDominatedPoissonCorollaryTwo} and  \eqref{eq:OrderStatisticsExponentialDominatedPoissonCorollaryOne} and taking expectations on both sides in order to remove the conditioning, we get
\begin{align}%
\mathbb P (N(b-a)\leq x)\leq\mathbb P \big(\vert\Upsilon^{(n)}_{(a,b)}\vert\leq x \big),
\end{align}%
which is \eqref{eq:OrderStatisticsExponentialDominatedByPoissonStatementCorollary}.
\end{proof}

One of the cornerstones of the analysis in Section \ref{sec:main_theorem_exponential} was the uniform integrability of $(A_n'^2)_{n\geq1}$ in Lemma \ref{LemmaSecondMomentAPrime}. An analogous version holds in this general setting:
\begin{lemma}\label{LemmaSecondMomentAPrimeGeneralArrivals}
$A_n(k)$ is stochastically bounded by a random variable with uniformly integrable $($with respect to $n)$ second moment, uniformly in $k\leq tn^{2/3}$.
\end{lemma}
\begin{proof}

Note that $T_{(i)}\stackrel{d}{=} F_{\sss T}^{-1}(1-\exp(-E_{(i)}))$, where $(E_{(i)})_{i=1}^n$ are the order statistics of unit mean exponential random variables. Then,
\begin{align}\label{eq:FirstPassageEstimateArrivalGeneralArrivals}%
A_n(k)~{\red \stackrel{\mathrm{a.s.}}{\leq} A'_n(k)}&\stackrel{d}{=}\sum_{i=1}^n \mathds 1_{\{\Sigma_{k-1}\leq F_{\sss T}^{-1}(1-\exp(-E_{(i)}))\leq\Sigma_k\}}\notag\\
&\stackrel{d}{=}\sum_{i=1}^n \mathds 1_{\{ F_{\sss T}(\Sigma_{k-1})-1\leq -\exp(-E_{(i)})\leq F_{\sss T}(\Sigma_k)-1\}}\notag\\
&\stackrel{d}{=}\sum_{i=1}^n \mathds 1_{\{ -\log(1-F_{\sss T}(\Sigma_{k-1}))\leq E_{(i)}\leq -\log(1-F_{\sss T}(\Sigma_k))\}}.
\end{align}%
By Corollary \ref{cor:OrderStatisticsExponentialDominatedByPoissonStatementCorollary},
\begin{align}\label{eq:ApplicationOfCorollaryToArrivals}%
A_n(k)\preceq N\Big(n\log\Big(\frac{1-F_{\sss T}(\Sigma_{k-1})}{1-F_{\sss T}(\Sigma_k)}\Big)\Big).
\end{align}%
By splitting the event space $\Omega$ we can write \eqref{eq:ApplicationOfCorollaryToArrivals} as
\begin{align}\label{eq:ArrivalStochBoundedByPoissonPlusTerm}%
A_n(k)\preceq N\Big(n\log\Big(\frac{1-F_{\sss T}(\Sigma_{k-1})}{1-F_{\sss T}(\Sigma_k)}\Big)\Big)\mathds 1_{\{\Sigma_k \leq \bar x\}} + A_n(k) \mathds 1_{\{\Sigma_k> \bar x\}}, 
\end{align}%
where $\bar x$ is independent of $n$ and will be determined later on. The plan now is to show that the first term in \eqref{eq:ArrivalStochBoundedByPoissonPlusTerm} is bounded by a random variable independent of $n$ and with finite second moment, and that the second term has second moments converging to zero, as $n$ tends to infinity. These two facts together imply that the right-hand side of \eqref{eq:ArrivalStochBoundedByPoissonPlusTerm} has uniformly integrable second moments. 

The first term is easily bounded. We now choose $\bar x$ in such a way that $1-F_{\sss T}(\bar x) >0$. By Taylor expanding the function $x \mapsto \log\left(\frac{1-F_{\sss T}(\Sigma_{k-1})}{1-F_{\sss T}(\Sigma_{k-1}+x)}\right)$, we get, for some $\Sigma_{k-1}^*$ in between $\Sigma_{k-1}$ and $\Sigma_k$
\begin{align}\label{eq:ArrivalsHaveUnifIntegrableSecondMomentFirstSplitTerm}%
N\Big(n\log\Big(\frac{1-F_{\sss T}(\Sigma_{k-1})}{1-F_{\sss T}(\Sigma_k)}\Big)\Big)\mathds 1_{\{\Sigma_k \leq \bar x\}} &= N\Big(n\frac{f_{\sss T}(\Sigma^*_{k-1})}{1-F_{\sss T}(\Sigma^*_{k-1})}\frac{S_k}{n}\Big)\mathds 1_{\{\Sigma_k \leq \bar x\}} \nnl
&\preceq N\Big(\frac{\fT(0)}{1-F_{\sss T}(\bar x)}S_k \Big),
\end{align}%
where we have used that the density $f_{\sss T}(\cdot)$ has finite maximum value $\fT(0)$. The right-most term in \eqref{eq:ArrivalsHaveUnifIntegrableSecondMomentFirstSplitTerm} has finite second moment, since $\mathbb E[S^2]<\infty$.
For the second term we proceed by noting that
\begin{align}%
A_n(k)^2 \mathds 1 _{\{\Sigma_k\geq\bar x\}} \stackrel{\mathrm{a.s.}}{\leq}  A_n(k)^2 \mathds 1_{\{D_k\geq \bar x/2\}} + A_n(k)^2\mathds 1_{\{\Sigma_{k}\geq\bar x,D_k<\bar x/2\}},
\end{align}%
The mean of the first term can be bounded by
\begin{align}%
\mathbb E [ A_n(k)^2 \mathds 1_{\{D_k\geq \bar x/2\}}]\leq n^2 \mathbb P (S_k\geq n\bar x/2)\leq n^2 \frac{\mathbb E[S_k^2\mathds 1_{\{S_k\geq n\bar x/2\}}]}{(n\bar x/2)^2},
\end{align}%
and the right-hand side tends to zero as $n$ tends to infinity since $\mathbb E[S^2_k] < \infty$. For the second term, some more work is needed. First observe that $\mathds 1_{\{\Sigma_k \geq \bar x, D_k < \bar x/2\}} \leq \mathds 1_{\{\Sigma_{k-1} \geq \bar x/2\}}$. 
After dominating as usual $A_n(k)^2$ by  $n^2$, we compute
\begin{align}\label{eq:ArrivalsStochBoundedSecondTerm}%
\mathbb E[ A_n(k)^2\mathds 1_{\{\Sigma_{k}\geq\bar x,D_k<\bar x/2\}} ] &\leq n^{2}\mathbb E[\mathbb E[\mathds 1_{\{\Sigma_{k-1}\leq T_i\leq \Sigma_k \}} \mathds 1 _{\{\Sigma_{k-1}\geq \bar x/2\}}\vert\Sigma_k]]\notag\\
& =  n^{2}\mathbb E[\mathds 1_{\{\Sigma_{k-1}\geq \bar x/2\}} \mathbb E[\mathds 1_{\{\Sigma_{k-1}\leq T_i\leq \Sigma_k \}}\vert \Sigma_k]]  \notag\\
&= n^{2}\mathbb E[\mathds 1_{\{\Sigma_{k-1}\geq \bar x/2\}}(F_{\sss T}(\Sigma_k) - F_{\sss T} (\Sigma_{k-1}))].
\end{align}%
By applying the Mean Value Theorem to $F_{\sss T}(\cdot)$, we obtain
\begin{equation}
 \vert F_{\sss T}(\Sigma_k)-F_{\sss T}(\Sigma_{k-1}) \vert \stackrel{\textrm{a.s.}}{\leq} f_{\sss T}(0) D_k,
\end{equation}%
where we recall that $f_{\sss T}(0) = \max_{t\in\mathbb [0,\infty)}f_{\sss T}(t)$. Inserting this into \eqref{eq:ArrivalsStochBoundedSecondTerm}, 
\begin{align}%
\mathbb E[ A_n(k)^2\mathds 1_{\{\Sigma_{k}\geq\bar x,D_k<\bar x/2\}} ] &\leq nM\mathbb E [ S_k \mathds 1_{\{\Sigma_{k-1}\geq \bar x/2\}}]=nM\mathbb E [ S_k]\mathbb P( \Sigma_{k-1}\geq \bar x/2),
\end{align}%
the equality following from independence of $S_k$ and $\Sigma_{k-1}$.

It is easy to see that the right-hand side converges to zero by using Chebyshev's inequality. 
Indeed, taking $n$ so large that $n^{2/3} \mathbb E[S] \leq \frac{n \bar x}{4}$,
\begin{align}%
\mathbb P (\Sigma_{k-1} \geq \bar x/2) &\leq  \mathbb P \big(\vert\sum_{i=1}^{tn^{2/3}}S_i - n^{2/3} \mathbb E[S_i]\vert\geq   n\bar x /4 \big)\notag\\
&\leq16\frac{tn^{2/3} \textrm{Var}(S_i)}{ n^2\bar x } = o(n^{-1}).
\end{align}%
This concludes the proof that the second moment of the second term in \eqref{eq:ArrivalStochBoundedByPoissonPlusTerm} tends to zero as $n$ tends to infinity.
\end{proof}

We conclude with a useful application of Doob's inequality:
\begin{lemma}\label{lem:ApproxSWithExpectation}%
Assume $(S_i)_{i\geq0}$ is a sequence of \emph{i.i.d.} random variables with finite second moment. Then, for any $\alpha, \beta >0$ such that $\alpha < 2\beta$,
\begin{align}
\frac{\sup_{k\leq tn^{\alpha}}\vert\sum_{i=1}^kS_i-k\cdot\mathbb E[S]\vert}{n^{\beta}}\stackrel{\mathbb P}{\longrightarrow}0.
\end{align}
\end{lemma}%
\begin{proof}
Define $M_k := \sum_{j=1}^k (S_j - \mathbb E[S])$. Then $(M_k)_{k\geq1}$ is a martingale. Therefore, by Doob's inequality applied to the sub-martingale $ (\vert M_k\vert )_{k\geq1}$, we have
\begin{align}
\mathbb P \Big(\frac{\sup_{k\leq t n^{\alpha}}\vert M_k\vert}{n^{\beta}} >\varepsilon\Big) \leq \frac{\mathbb E[M_{tn^{\alpha}}^2]}{\varepsilon ^2 n^{2\beta}} = \frac{tn^{\alpha}\mathbb E[(S-\mathbb E[S])^2]}{\varepsilon^2n^{2\beta}}.
\end{align}
This converges to zero since $\alpha < 2\beta$. Note that $\varepsilon$ can depend on $n$, for example by defining $\varepsilon := n^{-\delta}$ and choosing $\delta$ such that $ \delta < \beta - \alpha/2$.
\end{proof}

\subsection{Proof of Theorem \ref{MainTheorem}}\label{ref:ProofMainTheorem}
The proof consists, again, in verifying four conditions in order to apply Theorem \ref{FCLT}. The first is the convergence of the drift
\begin{itemize}
\item[(i)] $\sup_{t\leq \bar t} \vert n^{-1/3}C_n(tn^{2/3})- \beta t - f'_{\sss T}(0)/f_{\sss T}(0)^2\cdot\frac{t^2}{2} \vert\stackrel{\mathbb P}{\longrightarrow}0,\qquad \forall \bar t\in \mathbb R^+$;
\end{itemize}
while (ii), (iii) and (iv) are technical conditions and analogous to the ones given in Section \ref{sec:main_theorem_exponential}. The filtration we consider henceforth is defined as 
\begin{align}%
\mathcal F_{i}:= \sigma(\{A_n(j),D_j\}_{j\leq i}).
\end{align}%

\subsubsection{Proof of (i)}
Computing the asymptotic drift essentially amounts to computing the discrete drift, which in turn depends heavily on 
\begin{align}%
\mathbb E [A_n(k)\vert \mathcal F_{k-1}] &={\red \sum_{i\nin\nu_{k}}}\mathbb E \left[\mathds 1 _{\{\Sigma_{k-1}\leq{\red  T_{i}}\leq \Sigma_k\}}\Big\vert \mathcal F_{k-1}\right]\nnl
&={\red \sum_{i\nin\nu_{k}}}\mathbb E \left[\mathds 1 _{\{ {\red T_{i}}\leq \Sigma_k\}}\Big\vert \Sigma_{k-1},\big\{T_{i}\geq\Sigma_{k-1}\big\}\right],
\end{align}%
where, as above, $\nu_i$ denotes the set of the customers that no longer remain in the population at the beginning of the service of the $i$-th customer.
Adding the conditioning on $\left\{T_{i}\geq\Sigma_{k-1}\right\}$ does not influence the conditional expectation, since $T_{i}$ is such that $i \nin\nu_{k-1}$. Indeed, note that $i\nin\nu_{k}$ implies $T_{i}\geq\Sigma_{k-1}$. Then, defining for simplicity $\mathcal{E}_{k-1}:=\{\Sigma_{k-1},\big\{T_i\geq\Sigma_{k-1}\big\}\}$,
\begin{align}\label{ArrivalsConditionedGeneralArrivals}
\mathbb E [A_n(k)\vert \mathcal F_{k-1}] &= {\red \sum_{i\nin\nu_{k}}} \mathbb E \Big[\mathbb E[\mathds 1 _{\{\Sigma_{k-1}\leq {\red T_{i}}\leq \Sigma_{k}\}} \vert D_k,\mathcal E_{k-1} ]\Big\vert \mathcal{E}_{k-1}\Big]\notag\\
&=(n-{\red \vert\nu_{k}\vert})\mathbb E \Big[  \frac{F_{\sss T}(\Sigma_{k})-F_{\sss T}(\Sigma_{k-1})}{1-F_{\sss T}(\Sigma_{k-1})}\Big\vert \mathcal{E}_{k-1}\Big]\notag\\
&=\frac{(n-\vert\nu_{k}\vert)}{1-F_{\sss T}(\Sigma_{k-1})}\mathbb E \big[  F_{\sss T}(\Sigma_{k})-F_{\sss T}(\Sigma_{k-1})\big\vert \mathcal{E}_{k-1}\big].
\end{align}
We now rearrange the terms in order to distinguish between the ones contributing to the limit and those vanishing, as follows
\begin{align}\label{eq:ArrivalsRewritingTwoTerms}%
\mathbb E [A_n(k)\vert \mathcal F_{k-1}] -1 &= \frac{(n-{\red \vert\nu_{k}\vert})}{1-F_{\sss T}(\Sigma_{k-1})}\mathbb E \big[  F_{\sss T}(\Sigma_{k})-F_{\sss T}(\Sigma_{k-1})\big\vert \mathcal{E}_{k-1}\big] -\frac{1-F_{\sss T}(\Sigma_{k-1})}{1-F_{\sss T}(\Sigma_{k-1})}\notag\\
&= \frac{n}{1-F_{\sss T}(\Sigma_{k-1})}\mathbb E \big[  F_{\sss T}(\Sigma_{k})-F_{\sss T}(\Sigma_{k-1}) - 1/n\big\vert \mathcal{E}_{k-1}\big]\notag\\
&\quad-\frac{1}{1-F_{\sss T}(\Sigma_{k-1})}\mathbb E \big[\vert\nu_{k}\vert(  F_{\sss T}(\Sigma_{k})-F_{\sss T}(\Sigma_{k-1}))-F_{\sss T}(\Sigma_{k-1})\big\vert \mathcal{E}_{k-1}\big]\notag\\
&=: A_n^{\sss (1)}(k) - A_n ^{\sss (2)}(k).
\end{align}%
$A_n^{\sss (1)}(k)$ groups all the terms appearing in the limit, while $A_n^{\sss (2)}(k)$ groups all the terms of lower order, vanishing in the limit. We treat them separately, starting with $A_n^{\sss (1)}(k)$.
The term $F_{\sss T}(\Sigma_k)-F_{\sss T}(\Sigma_{k-1})$ can be simplified by our assumptions. By \eqref{eq:ArrivalDistriNearXBar} and Lemma \ref{lem:oTaylorExpansionDistributionGeneralArrivals},
\begin{align}\label{eq:ArrivalsConditionedGeneralArrivals2}%
A_n^{\sss (1)}(k) &= \frac{n}{1-F_{\sss T}(\Sigma_{k-1})}\mathbb E \big[  f_{\sss T}(\Sigma_{k-1}) D_k - n^{-1} +o_{\mathbb P}(n^{-4/3})\big\vert \mathcal{E}_{k-1}\big],
\end{align}%
and by \eqref{eq:DensityDistrNearZero},
\begin{align}\label{eq:TaylorExpansionAppliedOnF}
\mathbb E \big[f_{\sss T}(\Sigma_{k-1}) D_k\big\vert\mathcal{E}_{k-1}\big]&=\left(f_{\sss T}(0)+f'_{\sss T}(0)\Sigma_{k-1}+o\left(\Sigma_{k-1}\right)\right)\mathbb E [D]\notag\\
&=f_{\sss T}(0) \mathbb E[D_k] + f'_{\sss T}(0)\Sigma_{k-1}\mathbb E [D_k] +o\left(\Sigma_{k-1}\right)\mathbb E [D_k],
\end{align}%
where, with a slight abuse of notation, we denoted  $\vert f_{\sss T} (\Sigma_{k-1})  - f_{\sss T}(0) - f_{\sss T}'(0) \Sigma_{k-1}\vert$ by $o(\Sigma_{k-1})$. Since, for $k = O(n^{2/3})$,
\begin{equation}\label{eq:StochSmalloArrivalsDensityGeneralArrivals}%
n^{1/3}\vert f_{\sss T} (\Sigma_{k-1})  - f_{\sss T}(0) - f_{\sss T}'(0) \Sigma_{k-1}\vert\stackrel{\mathrm{a.s.}}{\rightarrow} 0,
\end{equation}%
also convergence in probability holds, that is, $o(\Sigma_{k-1}) = o_{\mathbb P} (n^{-1/3})$. In particular $o(\Sigma_{k-1})\cdot \mathbb E[D_k] = o_{\mathbb P} (n^{-4/3})$.
%
Inserting \eqref{eq:TaylorExpansionAppliedOnF} into \eqref{eq:ArrivalsConditionedGeneralArrivals2} yields
\begin{align}%
A_n^{\sss (1)}(k)&=\frac{n}{1-F_{\sss T}(\Sigma_{k-1})} \Big(f_{\sss T}(0) \mathbb E[D_k]  -1/n + f'_{\sss T}(0)\mathbb E [D_k]\Sigma_{k-1} + o_{\mathbb P}(n^{-4/3})\Big)\notag\\
&=\frac{1}{1-F_{\sss T}(\Sigma_{k-1})} \Big(f_{\sss T}(0) \mathbb E[S_k]  -1 +\frac{f'_{\sss T}(0)}{n}\sum_{j=1}^{k-1}S_j\mathbb E [S_k] + o_{\mathbb P}(n^{-1/3})\Big).
\end{align}%
The criticality assumption $f_{\sss T}(0)\mathbb E[S_k]=1+\beta n^{-1/3}+o(n^{-1/3})$ leads to
\begin{align}\label{eq:ArrivalsConditionedFinalExpression}
A_n^{\sss (1)}(k)&=\frac{1}{1-F_{\sss T}(\Sigma_{k-1})} \Big(\beta  + f'_{\sss T}(0)\mathbb E [S_k]\frac{\sum_{j=1}^{k-1}S_j}{n^{2/3}} + o_{\mathbb P}(1)\Big)\cdot n^{-1/3}.
\end{align}
Recall that the drift term is defined as
\begin{align}\label{eq:DriftTermDuringProof}%
C_n(s)=\sum_{k=1}^s\big(\mathbb E[A_n(k)\vert\mathcal F_{k-1}]-1\big)=\sum_{k=1}^s\big(A_n^{\sss (1)}(k)-A_n^{\sss (2)}(k)\big)=: C_n^{\sss (1)}(s) + C_n^{\sss (2)}(s).
\end{align}%
%
We now sum \eqref{eq:ArrivalsConditionedFinalExpression} over $k$, obtaining
\begin{align}%
C_n^{\sss (1)}(s) = \sum_{k=1}^s\frac{\beta n^{-1/3}}{1-F_{\sss T}(\Sigma_{k-1})} &+ \frac{f'_{\sss T}(0)\mathbb E [S_1]}{n}\sum_{k=1}^s\frac{\sum_{j=1}^{k-1}S_j}{1- F_{\sss T} (\Sigma_{k-1})} + \sum_{k=1}^s\frac{1}{1- F_{\sss T} (\Sigma_{k-1})}o_{\mathbb P}(n^{-1/3}).
\end{align}%
By scaling time as $s=t n^{2/3}$ and space as $n^{-1/3}$, we obtain
\begin{align}\label{eq:LinearDriftGeneralArrivalsRescaled}%
n^{-1/3}C_n^{\sss (1)}(tn^{2/3}) = \sum_{k=1}^{tn^{2/3}}&\frac{\beta n^{-2/3}}{1-F_{\sss T}(\Sigma_{k-1})} + \frac{f'_{\sss T}(0)\mathbb E [S_1]}{n^{4/3}}\sum_{k=1}^{tn^{2/3}}\frac{\sum_{j=1}^{k-1}S_j}{1- F_{\sss T} (\Sigma_{k-1})} + \sum_{k=1}^{tn^{2/3}}\frac{n^{-1/3}o_{\mathbb P}(n^{-1/3})}{1- F_{\sss T} (\Sigma_{k-1})}.
\end{align}
Before continuing, we give two easy but useful results:
\begin{lemma}\label{lem:LLNForDoubleSums}%
Let $(S_i)_{i\geq1}$ be a sequence of \emph{i.i.d.} random variables such that $\mathbb E[S_1^2] <  \infty$. Then
\begin{equation}%
\Big\vert\frac{\sum_{n=1}^N\sum_{i=1}^n S_i}{N^2} - \frac{\mathbb E[S_1]}{2}\Big\vert\stackrel{\mathbb P}{\longrightarrow} 0,\qquad \mathrm{as~} N\rightarrow \infty.
\end{equation}%
Moreover,
\begin{equation}%
\Big\vert\frac{\sum_{n=1}^N\big(\sum_{i=1}^n S_i\big)^2}{N^3} - \frac{\mathbb E[S_1]^2}{3}\Big\vert\stackrel{\mathbb P}{\longrightarrow} 0,\qquad \mathrm{as~} N\rightarrow \infty.
\end{equation}%
\end{lemma}%
\begin{proof}
Both claims can be proven through Lemma \ref{lem:ApproxSWithExpectation}. We omit the details.
\end{proof}

Another useful ingredient is the following Taylor expansion:
\begin{align}\label{eq:TaylorExpansionOneOverOneMinusF}%
\frac{1}{1-F_{\sss T}(\Sigma_{k-1})} &= 1 + F_{\sss T}(\Sigma_{k-1})+\Big(\frac{1}{1-F_{\sss T}(\Sigma_{k-1})} - 1 - F_{\sss T}(\Sigma_{k-1}) \Big)\\
&= 1 + O_{\mathbb P}\big(\Sigma_{k-1}\big) \notag.
\end{align}%
In what follows, we compute the limits (in probability) for each term in \eqref{eq:LinearDriftGeneralArrivalsRescaled}.

\medskip
\noindent
\textbf{First term in \eqref{eq:LinearDriftGeneralArrivalsRescaled}.} By \eqref{eq:TaylorExpansionOneOverOneMinusF} and Lemma \ref{lem:LLNForDoubleSums}, we have for the first term
\begin{align*}%
&\sup_{t\leq\bar t}\Big\vert n^{-2/3}\sum_{k=1}^{tn^{2/3}}\frac{1}{1-F_{\sss T}(\Sigma_{k-1})} -  t\Big\vert  \notag\\
&=\sup_{t\leq\bar t}\Big\vert n^{-2/3}\sum_{k=1}^{tn^{2/3}}\Big(F_{\sss T}(\Sigma_{k-1})+\big(\frac{1}{1-F_{\sss T}(\Sigma_{k-1})} - 1 - F_{\sss T}(\Sigma_{k-1}) \big) \Big)\Big\vert\notag\\
&\leq   n^{-2/3}\sum_{k=1}^{\bar tn^{2/3}}2 F_{\sss T}(\Sigma_{k-1}),
\end{align*}%
where we have dominated the error term in the Taylor expansion \eqref{eq:TaylorExpansionOneOverOneMinusF} by $F_{\sss T}(\Sigma_{k-1})$ and used the fact that, as a function of $t$, the summation is an increasing function. We can uniformly dominate $F_{\sss T}$ as follows
\begin{align}\label{eq:FirstTermDriftGeneralArrivals}%
 n^{-2/3}\sum_{k=1}^{\bar tn^{2/3}}2 F_{\sss T}(\Sigma_{k-1}) & \leq   n^{-2/3}\sum_{k=1}^{\bar tn^{2/3}}2 \sup_{k\leq \bar t n^{2/3}} F_{\sss T}(\Sigma_{k-1}) \notag\\
 &= 2\bar t F_{\sss T}(\Sigma_{\bar t n^{2/3}}).
\end{align}%
The right-hand side of \eqref{eq:FirstTermDriftGeneralArrivals} tends to zero almost surely, and thus also in probability.

\medskip
\noindent
\textbf{Second term in \eqref{eq:LinearDriftGeneralArrivalsRescaled}.} Again, by \eqref{eq:TaylorExpansionOneOverOneMinusF}, the second term simplifies to
\begin{align*}%
\frac{f'_{\sss T}(0)\mathbb E [S]}{n^{4/3}}&\sum_{k=1}^{tn^{2/3}}\frac{\sum_{j=1}^{k-1}S_j}{1- F_{\sss T} (\Sigma_{k-1})}=\frac{f_{\sss T}'(0)\mathbb E[S]}{n^{4/3}}\sum_{k=1}^{tn^{2/3}}\sum_{j=1}^{k-1}S_j + \frac{f_{\sss T}'(0)\mathbb E[S] f_{\sss T}(0)}{n^{4/3}}\sum_{k=1}^{tn^{2/3}}\sum_{j=1}^{k-1}S_jO_{\mathbb P}\big(\Sigma_{k-1}\big).
\end{align*}%
By the first statement of Lemma \ref{lem:LLNForDoubleSums} the first term converges to $\frac{t^2}{2}\cdot f'_{\sss T}(0) \mathbb E [S_1]^2$ uniformly in t. Indeed,

\begin{align}%
\sup_{t\leq\bar t}\vert n^{-4/3}\sum_{k=1}^{tn^{2/3}}\sum_{j=1}^{k-1}S_j - \frac{t^2}{2}\mathbb E[S]\vert &= \sup_{t\leq\bar t}\Big\vert \frac{\sum_{n=1}^{tn^{2/3}}n\mathbb E[S]}{n^{4/3}}- \frac{t^2}{2}\mathbb E[S] + \frac{\sum_{k=1}^{tn^{2/3}}(\sum_{i=1}^k S_i -k\mathbb E[S])}{n^{4/3}} \Big\vert \nnl
&=\sup_{t\leq\bar t}\Big\vert\frac{t\mathbb E[S]}{2n^{2/3}} + \frac{\sum_{k=1}^{tn^{2/3}}(\sum_{i=1}^k S_i -k\mathbb E[S])}{n^{4/3}} \Big\vert \nnl
&\stackrel{\textrm{a.s.}}{\leq} \bar t \cdot \frac{\mathbb E[S]}{2n^{2/3}}  + \bar t \cdot\frac{ \sup_{k\leq tn^{2/3}}\vert(\sum_{i=1}^k S_i - k\mathbb E[S])\vert}{n^{2/3}}.
\end{align}%
The second term converges to zero in probability by Lemma \ref{lem:LLNForDoubleSums}. Indeed,
\begin{align}\label{eq:SecondTermDriftGeneralArrivals}%
\sup_{t\leq\bar t}\Big\vert \frac{f_{\sss T}'(0)\mathbb E[S] f_{\sss T}(0)}{n^{4/3}}\sum_{k=1}^{tn^{2/3}}\Big(\sum_{j=1}^{k-1}S_j\Big)O_{\mathbb P}\big(\Sigma_{k-1}\big) \Big\vert &\stackrel{\textrm{a.s.}}{\leq} \frac{Cf_{\sss T}'(0)\mathbb E[S] f_{\sss T}(0)}{n^{7/3}}\sum_{k=1}^{\bar tn^{2/3}}\Big\vert\Big(\sum_{j=1}^{k-1}S_j\Big)^2 \Big\vert,
\end{align}%
where we used the domination $\OP(\Sigma_{k-1}) \leq C \Sigma_{k-1}$. The right-most term in \eqref{eq:SecondTermDriftGeneralArrivals} then converges  to zero in probability by Lemma \ref{lem:LLNForDoubleSums}.

\medskip
\noindent
\textbf{Third term in \eqref{eq:LinearDriftGeneralArrivalsRescaled}.} The error term originates  from the Taylor expansion of $n(F_{\sss T}(\Sigma_k) - F_{\sss T}(\Sigma_{k-1}))$ done in \eqref{eq:ArrivalsConditionedGeneralArrivals2}. To see that it is uniform in $k\leq \bar k$, we write $F_{\sss T}(\Sigma_k) - F_{\sss T}(\Sigma_{k-1}) - f_{\sss T}(\Sigma_{k-1})D_k = \epsilon_k$ and, since by assumption $\fT'(\cdot)$ exists and is continuous in a neighborhood of $0$, we bound it as
\begin{align}\label{eq:UniformDominationErrorTerms}%
\vert n\epsilon_k \vert\leq  n\sup_{x\leq C n^{-1/3}}\vert f'_{\sss T}(x)\vert \frac{D_k^2}{2},
\end{align}%
similarly as in Lemma \ref{lem:oTaylorExpansionDistributionGeneralArrivals}, when $k=O(n^{2/3})$. In particular,
\begin{align}%
\sup_{k\leq t n^{2/3}} \mathbb E[n\epsilon_k\vert\mathcal E_{k-1}] \leq \sup_{x\leq C n^{-1/3}}\vert f'_{\sss T}(x)\vert \frac{\mathbb E[S^2]}{2n} = o(n^{-1/3}),
\end{align}%
and the right-hand side is independent of $k$. This concludes the bound on the third term in \eqref{eq:LinearDriftGeneralArrivalsRescaled} and thus the proof that 
\begin{equation}%
\sup_{t\leq \bar t}\vert n^{-1/3} C_n^{\sss (1)}(tn^{2/3}) - \beta t - t^2 \fT'(0) \E[S]^2/2 \vert \sr{\mathbb P}{\rightarrow} 0.
\end{equation}%
\qed

To conclude, it remains to be proven that $\sup_{t\leq \bar t}n^{-1/3}\vert C_n^{\sss (2)}(tn^{2/3})\vert$ vanishes in the limit. To do so, we develop the terms of $A_n^{\sss (2)}(k)$ similarly as before, obtaining
\begin{align}\label{eq:ArrivalsConditionedErrorFinalExpression}%
A_n^{\sss (2)}(k) &= \frac{1}{1-F_{\sss T}(\Sigma_{k-1})}\mathbb E \big[{\red \vert\nu_{k}\vert}(  f_{\sss T}(\Sigma_{k-1})D_k + o_{\mathbb P}(n^{-1}))-f_{\sss T}(0)\Sigma_{k-1} + o_{\mathbb P}(n^{-1/3})\big\vert \mathcal{E}_{k-1}\big]\nnl
&= \frac{1}{1-F_{\sss T}(\Sigma_{k-1})}\mathbb E \big[ \vert\nu_{k}\vert f_{\sss T}(\Sigma_{k-1})D_k -f_{\sss T}(0)\Sigma_{k-1} + \vert\nu_{k}\vert o_{\mathbb P}(n^{-1})+ o_{\mathbb P}(n^{-1/3})\big\vert \mathcal{E}_{k-1}\big]\nnl
&= \frac{\mathbb E \big[  \vert\nu_{k}\vert f_{\sss T}(0)D_k  -f_{\sss T}(0)\Sigma_{k-1} + f_{\sss T}'(0)\vert\nu_{k}\vert\Sigma_{k-1}D_k + \vert\nu_{k}\vert o_{\mathbb P}(n^{-1})+ o_{\mathbb P}(n^{-1/3})\big\vert \mathcal{E}_{k-1}\big]}{1-F_{\sss T}(\Sigma_{k-1})},
\end{align}%
where $o_{\mathbb P}(n^{-1})$ is a convenient notation for $ F_{\sss T}(\Sigma_k)- F_{\sss T}(\Sigma_{k-1}) - f_{\sss T}(\Sigma_{k-1})D_k$ and $o_{\mathbb P}(n^{-1/3}) $ for $ F_{\sss T}(\Sigma_{k-1}) - f_{\sss T}(0)\Sigma_{k-1} $. We now sum \eqref{eq:ArrivalsConditionedErrorFinalExpression} over $k$ to obtain
\begin{align}\label{eq:DriftErrorTermGeneralArrivals}%
C_n^{\sss (2)}(s) &= \sum_{k=1}^s \frac{1}{1-F_{\sss T}(\Sigma_{k-1})}\mathbb E \big[ {\red \vert\nu_{k}\vert} f_{\sss T}(0)D_k  -f_{\sss T}(0)\Sigma_{k-1} + f'(0)\vert\nu_{k-1}\vert\Sigma_{k-1}D_k \nnl
&\quad+ \vert\nu_{k}\vert o_{\mathbb P}(n^{-1})+ o_{\mathbb P}(n^{-1/3})\big\vert \mathcal{E}_{k-1}\big].
\end{align}%
Recall that, by definition, $\vert \nu_k\vert =  k + N_n(k-1)$. Intuitively, $k$ is of a much larger order than $N_n(k-1)$, therefore at  first approximation we ignore $N_n(k-1)$ and we will later prove convergence of the terms containing it. We rescale, and split $C_n^{(2)}(k)$ into
\begin{align}\label{eq:ArrivalsErrorsGlobal}%
&n^{-1/3}C_n^{\sss (2)}(tn^{2/3}) = n^{-1/3}\sum_{k=1}^{tn^{2/3}}\frac{1}{1-F_{\sss T}(\Sigma_{k-1})} \mathbb E \big[ {\red k} f_{\sss T}(0)D_k  -f_{\sss T}(0)\Sigma_{k-1}\big\vert \mathcal{E}_{k-1}\big]\\
&+ n^{-1/3}\sum_{k=1}^{tn^{2/3}} \frac{1}{1-F_{\sss T}(\Sigma_{k-1})}\mathbb E \big[f_{\sss T}'(0)k\Sigma_{k-1}D_k + k o_{\mathbb P}(n^{-1})+ o_{\mathbb P}(n^{-1/3})\big\vert \mathcal{E}_{k-1}\big]+\varepsilon_n\notag,
\end{align}%
where $\varepsilon_n$ represents the terms containing $N_n(k-1)$. Again we rescale and perform a term-by-term analysis.

\medskip
\noindent
\textbf{First term in \eqref{eq:ArrivalsErrorsGlobal}.} Expanding $(1-F_{\sss T}(\Sigma_{k-1}))^{-1}$ gives
\begin{align}\label{eq:ArrivalsErrorsFirst}%
&\frac{f_{\sss T}(0)}{n^{4/3}}\sum_{k=1}^{tn^{2/3}} \mathbb E \big[ {\red  k}S_k  -\sum_{j=1}^{k-1}S_j\big\vert \mathcal{E}_{k-1}\big] + \frac{f^2_{\sss T}(0)}{n^{1/3}}\sum_{k=1}^{tn^{2/3}}O_{\mathbb P}(\Sigma_{k-1})\cdot\mathbb E \big[ {\red k} D_k  -\Sigma_{k-1}\big\vert \mathcal{E}_{k-1}\big],
\end{align}%
The second term is almost surely dominated by the first for $n$ sufficiently large, so that it is enough to show (uniform) convergence of the first term. By Lemma \ref{lem:LLNForDoubleSums},
\begin{align}%
-\frac{1}{n^{4/3}}\sum_{k=1}^{tn^{2/3}} \sum_{j=1}^{k-1}S_j&\stackrel{\mathbb P}{\longrightarrow}-\frac{t^2}{2}\mathbb E[S],\qquad
\frac{1}{n^{4/3}}\sum_{k=1}^{tn^{2/3}}{\red k}\mathbb E[S_k]\stackrel{\mathbb P}{\longrightarrow}\frac{t^2}{2}\mathbb E[S].
\end{align}%
Therefore, \eqref{eq:ArrivalsErrorsFirst} converges to zero in probability. Moreover, the convergence is uniform in $t\leq \bar t$ by Lemma \ref{lem:ApproxSWithExpectation}.

\medskip
\noindent
\textbf{Second term in \eqref{eq:ArrivalsErrorsGlobal}.} Expanding $(1-F_{\sss T}(\Sigma_{k-1}))^{-1}$ and ignoring all but the highest order term, which can be almost surely dominated, we get for the second term
\begin{align}%
f'_{\sss T}(0)\mathbb E[S]n^{-7/3}\sum_{k=1}^{tn^{2/3}}{\red k}\sum^{k-1}_{j=1}S_j + n^{-1/3}\sum_{k=1}^{tn^{2/3}} k o_{\mathbb P}(n^{-1})+tn^{1/3} o_{\mathbb P}(n^{-1/3}).
\end{align}%
One can check, similarly as in Lemma \ref{lem:LLNForDoubleSums}, that $N^{-3}\sum_{k=1}^Nk\sum_{j=1}^kS_j$ converges in probability to a non-trivial limit. Since $7/3 > (2/3)\cdot3$, the first term converges to zero in probability. In addition, it converges uniformly in $t\leq \bar t$ because of the monotonicity of the sum. The small-o terms can be uniformly dominated as has already been done in \eqref{eq:UniformDominationErrorTerms}.

\medskip
\noindent
\textbf{Third term in \eqref{eq:ArrivalsErrorsGlobal}.} The remaining term  is
\begin{align*}%
\varepsilon_n=\sum_{k=1}^{tn^{2/3}}\frac{n^{-1/3}}{1-F_{\sss T}(\Sigma_{k-1})}\mathbb E\big[N_n(k-1)(f_{\sss T}(0)D_k + f_{\sss T}'(0)\Sigma_{k-1}D_k +o_{\mathbb P}(n^{-1}))\big\vert\mathcal E_{k-1} \big].
\end{align*}%
Again it is sufficient to prove (uniform) convergence for the first term in the Taylor expansion of $(1-F_{\sss T}(\Sigma_{k-1}))^{-1}$. This simplifies the previous expression to
\begin{align}%
\frac{f_{\sss T}(0)\mathbb E[S]}{n^{4/3}}\sum_{k=1}^{tn^{2/3}}N_n(k-1) &+ \frac{f_{\sss T}'(0)\mathbb E[S]}{n^{4/3}}\sum_{k=1}^{tn^{2/3}}N_n(k-1)\Sigma_{k-1}+ n^{-4/3}\sum_{k=1}^{tn^{2/3}} N_n(k-1) o_{\mathbb P}(n^{-1}).
\end{align}%
The second and third terms are again almost surely dominated by the first for $n$ large. Moreover, the first converges to zero uniformly in probability by the following lemma:
\begin{lemma}\label{lem:QueueUnifConvGeneralArrivals}%
Let $(N_n(k))_{k\geq0}$ be the process \eqref{QueueLength} for the General Arrivals Model and $a\in\mathbb R^+$. Then,
\begin{align}%
n^{-2/3}\sup_{j\leq an^{2/3}}\vert N_n(j)\vert \stackrel{\mathbb P}{\longrightarrow}0.
\end{align}%
\end{lemma}%
\begin{proof}
The proof proceeds as in Lemma \ref{oQueueLemmaUnif}. We split $N_n(j)$ as the sum of a martingale and a predictable process, $N_n(j) = M_n(j) + C_n(j)$, and bound each one separately. The term containing the martingale $M_n(j)$ can be bounded through Doob's inequality, leaving us to bound $\mathbb E[M_n^2(an^{2/3})]/(\varepsilon n^{2/3})^2$. As was noticed in Lemma \ref{oQueueLemmaUnif}, $\mathbb E[M_n(k)^2] = \mathbb E[V_n(k)]$, where $V_n(k)$ is the predictable quadratic variation of $M_n(k)$, and its expectation is given by
\begin{align}%
\mathbb E [V_n(k)] &= \mathbb E\Big[\sum_{i=1}^k (\mathbb E[A_n(i)^2\vert\mathcal F_{i-1}] - \mathbb E^2[A_n(i)\vert \mathcal F_{i-1}])\Big]\leq \sum_{i=1}^k \mathbb E[A_n(i)^2].
\end{align}%
This term can be easily bounded exploiting Lemma \ref{LemmaSecondMomentAPrimeGeneralArrivals}. We have
\begin{equation}%
\frac{1}{(\varepsilon n^{2/3})^2}\mathbb E [V_n(an^{2/3})] \leq \frac{1}{(\varepsilon n^{2/3})^2}\sum_{i=1}^{an^{2/3}} \mathbb E[(A_n'(i))^2],
\end{equation}%
which tends to zero because $\E[A'^2_n(i)]<\infty$ uniformly in $i = O(n^{2/3})$ by Lemma \ref{LemmaSecondMomentAPrimeGeneralArrivals}.
The $C_n(j)$ term (that was computed in \eqref{eq:ArrivalsRewritingTwoTerms} and \eqref{eq:DriftTermDuringProof}) is the difference of two increasing processes, and thus can be easily bounded from above and below as was done in Lemma \ref{LemmaSecondMomentAPrime}.
\end{proof}
This concludes the proof that 
\begin{align*}%
\sup_{t\leq\bar t}\vert n^{-1/3}C_n^{\sss (2)}(tn^{2/3})\vert\stackrel{\mathbb P}{\longrightarrow} 0.
\end{align*}%
and thus we have proven that
\begin{align*}%
\sup_{t\leq\bar t}\vert n^{-1/3}C_n(n^{2/3}t) - \beta t - t^2f'_{\sss T}(0)\mathbb E[S]^2/2\vert\stackrel{\mathbb P}{\longrightarrow} 0. 
\end{align*}%
This completes the proof of (i).
\qed

\subsubsection{Proof of (ii)}

First we compute $\mathbb E [A_n(k)^2\vert \mathcal F_{k-1}]$. By proceeding as in \eqref{ArrivalsSecondMomentBegin} we obtain
\begin{align}%
\mathbb E [A_n(k)^2\vert \mathcal F_{k-1}] - \mathbb E[A_n(k)\vert\mathcal F _{k-1}] &=  \mathbb E \big[{\red \sum_{\substack{l\neq m \\ l,m\nin\nu_{k}}}} \mathds 1_{\{\Sigma_{k-1}\leq {\red T_{m}}\leq\Sigma_k\}}\cdot\mathds 1_{\{\Sigma_{k-1}\leq {\red T_{l}}\leq\Sigma_k\}}\big\vert\mathcal F_{k-1}\big]\notag\\
&= \sum_{\substack{l\neq m \\ l,m\nin\nu_{k}}} \mathbb E \Big[ \frac{F_{\sss T}(\Sigma_k)-F_{\sss T} (\Sigma_{k-1})}{1-F_{\sss T}(\Sigma_{k-1})}\cdot \frac{F_{\sss T}(\Sigma_k)-F_{\sss T} (\Sigma_{k-1})}{1-F_{\sss T}(\Sigma_{k-1})}\Big\vert\mathcal F_{k-1}\Big]\notag\\
&= \frac{\sum_{l\neq m} \mathbb E [ (f_{\sss T}(\Sigma_{k-1})  D_k +o_{\mathbb P}(D_k^{-4/3}))^2\vert \mathcal F_{k-1}]}{(1-F_{\sss T}(\Sigma_{k-1}))^2}, 
\end{align}%
where the sum is over the set $\{l,m:l\neq m, l,m\nin {\red \nu_{k}}\}$ wherever not specified. We also denoted, for convenience,  $F_{\sss T}(\Sigma_k)-F_{\sss T} (\Sigma_{k-1})-f_{\sss T}(\Sigma_{k-1})  D_k$ as $o_{\mathbb P}(D_k^{4/3})$. We proceed as in \eqref{eq:CondExpectationDistrFunctProduct} and \eqref{ArrivalsSecondMoment}. By Lemma \ref{lem:oTaylorExpansionDistributionGeneralArrivals},
\begin{align}%
&\mathbb E [A_n(k)^2\vert \mathcal F_{k-1}] - \mathbb E[A_n(k)\vert\mathcal F _{k-1}] \notag\\
&=  \frac{1}{(1-F_{\sss T}(\Sigma_{k-1}))^2} \sum_{l\neq m}\mathbb E [f_{\sss T}^2(\Sigma_{k-1}) D^2_k +2f_{\sss T}(\Sigma_{k-1})D_k\cdot o_{\mathbb P}(D_k^{4/3}) + o_{\mathbb P}(D_k ^{2}) \vert\mathcal F_{k-1}]\notag\\
&= \frac{1}{(1-F_{\sss T}(\Sigma_{k-1}))^2} \sum_{l\neq m} \Big((f_{\sss T}(0)+ f'_{\sss T}(0)\Sigma_{k-1} + o_{\mathbb P}(\Sigma_{k-1}))^2\mathbb E [D_k^2]+o_{\mathbb P}(n^{-2})\Big).
\end{align}%
%
As usual, here $o_{\mathbb P}(\Sigma_{k-1})$ is a shorthand notation for $f_{\sss T}(\Sigma_{k-1}) - f_{\sss T}(0) - f_{\sss T}'(0)\Sigma_{k-1}$. Developing the coefficient of $\mathbb E[D_k^2]$ reveals that it has the form $f_{\sss T}(0)^2 + \alpha_k n^{-1/3} + \beta_k o(n^{-1/3})$, with $\alpha_k$ and $\beta_k$ converging in probability to a constant, for $k= O (n^{2/3})$. We can ignore all the terms except the one with the leading order, $f^2_{\sss T}(0)$. From this point onwards the computations are identical to \eqref{eq:QuadraticVariationExplicit}.

\subsubsection{Proof of (iii) and (iv)} 

The proof of (iii) in the exponentials arrivals case  can be carried over to the general arrivals case without any significant changes, since it relies only on (ii) and Lemma \ref{lem:QueueUnifConvGeneralArrivals}.

For (iv), we split the quantity according to
\begin{align}%
n^{-2/3}\mathbb E\Big[\sup_{t\leq \bar t}\vert M_n(tn^{2/3})-M(tn^{2/3}-)\vert^2\Big]&\leq n^{-2/3}\mathbb E\Big[\sup_{k\leq \bar t n^{2/3}}\vert A_n(k)\vert^2\Big]\notag\\
&\quad+n^{-2/3}\mathbb E \Big[\sup_{k\leq \bar t n^{2/3}} \vert \mathbb E [A_n(k)\vert \mathcal F_{k-1}]\vert ^2\Big].
\end{align}%
The second term is straightforward. Indeed, \eqref{eq:ArrivalsRewritingTwoTerms} and \eqref{eq:ArrivalsConditionedFinalExpression} give the crude bound
\begin{align*}
\mathbb E [A_n(k)\vert\mathcal F_{k-1}]\stackrel{\textrm{a.s.}}{\leq} A_n^{\sss (1)}(k)\stackrel{\textrm{a.s.}}{\leq} C + o_{\mathbb P}(1),
\end{align*}%
for $k= O(n^{2/3})$ and some positive constant $C>1$, uniform over $k\leq \bar t n^{2/3}$.  The first term can also be estimated imitating \eqref{ArrivalsDomination6}. Indeed, fix $\varepsilon>0$ and split it as
\begin{align}%
\mathbb E\Big[\sup_{k\leq\bar tn^{2/3}}\vert A_n(k) \vert ^2\Big]&=\mathbb E\Big[\sup_{k\leq\bar tn^{2/3}} A_n(k)^2\mathds{1}_{\{\sup_{k\leq\bar tn^{2/3}} A_n(k)^2\leq \varepsilon n^{2/3}\}}\Big]\notag\\
&\quad+\mathbb E\Big[\sup_{k\leq\bar tn^{2/3}} A_n(k)  ^2\mathds{1}_{\{\sup_{k\leq\bar tn^{2/3}} A(k)^2 > \varepsilon n^{2/3}\}}\Big].
\end{align}%
The first term is trivially bounded. The bounding of the second term proceeds as in \eqref{ArrivalsDomination5} and is concluded through Lemma \ref{LemmaSecondMomentAPrimeGeneralArrivals}. 
\qed

\section{Proof of Theorem \ref{th:MainTheoremKthOrderContact}}\label{sec:ell_order_contact_proof}
 Note that $\alpha$ in \eqref{eq:AlphaDefinition} is such that $\alpha < 1$. Some simple relations hold between $\alpha$ and $\ell$ and we will use these throughout this section:
\begin{align}%
\ell -\frac{\alpha}{2} &= \alpha  \ell,\qquad \ell + \frac{\alpha}{2}= \alpha  (\ell+1).
\end{align}%
We need to adapt the model assumptions to the $\ell$-th order contact case. Assume that the distribution function of the arrival times satisfies
\begin{align}%
F_{\sss T}(x) - F_{\sss T}(\bar x) = f_{\sss T}(\bar x)\cdot (x - \bar x) + o(\vert x - \bar x\vert ^{1 + \alpha/2}). 
\end{align}%
This is, for example, the case when $F_{\sss T}\in \mathcal C^2([0,\infty))$. The service times are given by
\begin{align}%
D_i := \frac{S_i}{n}(1 + \beta n^{-\alpha /2}),\qquad i\geq1.
\end{align}%
We assume that the maximum of the density $f_{\sss T}(\cdot)$ is obtained in zero, and the heavy-traffic condition is then defined as
\begin{align}%
nf_{\sss T}(0)\cdot \mathbb E[D] = 1 + \beta n^{-\alpha/2}.
\end{align}%
\subsection{Proof of Theorem \ref{th:MainTheoremKthOrderContact}}
We proceed by proving conditions (i)-(iv). We will treat condition (i) in great detail as this changes profoundly, since the limiting drift is significantly different. We will then discuss how (ii)-(iv) follow from the calculations in Section \ref{sec:GeneralArrivals}. 

\noindent
\textbf{Proof of (i).}
The starting point is again  \eqref{eq:ArrivalsRewritingTwoTerms}, 
\begin{align}%
\mathbb E [A_n(k)\vert \mathcal F_{k-1}] -1 &= \frac{n}{1-F_{\sss T}(\Sigma_{k-1})}\mathbb E \big[ ( F_{\sss T}(\Sigma_{k})-F_{\sss T}(\Sigma_{k-1}) - 1/n)\big\vert \mathcal{E}_{k-1}\big]\notag\\
&\quad-\frac{1}{1-F_{\sss T}(\Sigma_{k-1})}\mathbb E \big[({\red \vert\nu_{k}\vert}(  F_{\sss T}(\Sigma_{k})-F_{\sss T}(\Sigma_{k-1}))-F_{\sss T}(\Sigma_{k-1}))\big\vert \mathcal{E}_{k-1}\big]\notag\\
&=: A_n^{\sss (1)}(k) - A_n ^{\sss (2)}(k),
\end{align}%
and the corresponding drift decomposition
\begin{align}%
C_n(s)=\sum_{k=1}^s\left(A_n^{\sss (1)}(k)-A_n^{\sss (2)}(k)\right)=: C_n^{\sss (1)}(s) + C_n^{\sss (2)}(s).
\end{align}%
Recall that $A_n^{\sss (1)}(k)$ contains the terms appearing in the limit, while $A_n^{\sss (2)}(k)$ the ones that vanish. Expanding gives us the equivalent of \eqref{eq:ArrivalsConditionedFinalExpression}, in the form
\begin{align}%
A_n^{\sss (1)}(k)&=\frac{1}{1-F_{\sss T}(\Sigma_{k-1})}\cdot \Big(\beta  + \frac{f^{(\ell)}_{\sss T}(0)}{\ell!}\mathbb E [S_1]\frac{\big(\sum_{j=1}^{k-1}S_j\big)^{\ell}}{n^{\ell\cdot\alpha}} + o_{\mathbb P}(1)\Big)\cdot n^{-\alpha/2}
\end{align}%
It is easy to check that both the linear part of the drift and the error term converge uniformly by proceeding as in \eqref{eq:FirstTermDriftGeneralArrivals} and the following computations. Therefore, we focus on the second term of $A_n^{\sss (1)}(k)$, that is, on
\begin{align}%
\Big(\mathbb E [S]\frac{f^{(\ell)}_{\sss T}(0)}{\ell!}\Big)n^{-\alpha/2}\sum_{i=1}^{tn^{\alpha}}\frac{1}{1-F_{\sss T}(\Sigma_{i-1})}\Big(\frac{\sum_{j=1}^{i-1}S_j}{n}\Big)^{\ell},
\end{align}%
for which we prove (uniform) convergence in probability. We begin by computing
\begin{align}\label{eq:DriftTermUpperBoundKthOrderContactPre}%
\Big\vert\frac{1}{n^{\alpha(\ell+1)}}\sum_{i=1}^{tn^{\alpha}}\big(\sum_{j=1}^{i-1}S_j\big)^{\ell} - \frac{t^{\ell+1}}{\ell+1}\mathbb E[S]^{\ell}\Big\vert &= 
\frac{1}{n^{\alpha(\ell+1)}}\Big\vert\sum_{i=1}^{tn^{\alpha}}\big(\sum_{j=1}^{i-1}S_j\big)^{\ell} - \sum_{i=1}^{tn^{\alpha}}\mathbb ((i-1)\mathbb E[S])^{\ell} + o(n^{\alpha(\ell+1)})\Big\vert\notag\\
&\leq \frac{1}{n^{\alpha(\ell+1)}}\sum_{i=1}^{tn^{\alpha}}\Big\vert\big(\sum_{j=1}^{i-1}S_j\big)^{\ell}-\mathbb (i-1)^{\ell}\mathbb E[S]^{\ell} \Big\vert + o(1)
\end{align}%
We now make use of the following generalization of Lemma \ref{lem:ApproxSWithExpectation}:
\begin{lemma}\label{lem:ApproxSEllWithExpectation}
Assume $(S_i)_{i\geq0}$ is a sequence of \emph{i.i.d.} random variables with finite second moment. Then for any $\alpha>0, \beta \in\mathbb R$ such that $-\alpha < 2\beta$,
\begin{align}
\frac{\sup_{k\leq tn^{\alpha}}\Big\vert\big(\sum_{i=1}^kS_i\big)^{\ell}-k^{\ell}\mathbb E[S]^{\ell}\Big\vert}{n^{\alpha\ell+\beta}}\stackrel{\mathbb P}{\longrightarrow}0.
\end{align}
\end{lemma}
\begin{proof}
The proof is an application of Lemma \ref{lem:ApproxSWithExpectation}, hence we only sketch it. We have
\begin{align}%
\sup_{k\leq tn^{\alpha}}\Big\vert\Big(\sum_{i=1}^kS_i\Big)^{\ell}-k^{\ell}\mathbb E[S]^{\ell}\Big\vert &\leq \sup_{k\leq tn^{\alpha}}\Big\vert \Big( k\mathbb E[S] + \sup_{k\leq tn^{\alpha}}\vert\sum_{i=1}^k S_i - k \mathbb E[S]\vert \Big)^{\ell} -k^{\ell}\mathbb E[S]^{\ell}\Big\vert\notag\\
&\qquad \vee \sup_{k\leq tn^{\alpha}}\Big\vert \Big( k\mathbb E[S] - \sup_{k\leq tn^{\alpha}}\vert\sum_{i=1}^k S_i - k \mathbb E[S]\vert \Big)^{\ell} -k^{\ell}\mathbb E[S]^{\ell}\Big\vert.
\end{align}%
Both terms on the right can be shown to converge to zero when appropriately rescaled. To do so, it is enough to study the leading order term, which is 
\begin{align}%
\sup_{k\leq tn^{\alpha}}  k^{\ell  -1}\mathbb E[S]^{\ell -1} \sup_{k\leq tn^{\alpha}}\vert\sum_{i=1}^k S_i - k \mathbb E[S]\vert  = (tn)^{\alpha(\ell -1)}\mathbb E[S]^{\ell -1} \sup_{k\leq tn^{\alpha}}\vert\sum_{i=1}^k S_i - k \mathbb E[S]\vert,
\end{align}%
and this converges to zero when divided by $n^{\alpha\ell + \beta}$, with $\beta >-\alpha/2$.
\end{proof}
Note that when $\ell = 1$ we do, in fact, recover Lemma \ref{lem:ApproxSWithExpectation}, since in that case $\alpha + \beta > \alpha -\alpha/2 = \alpha/2$.

Then, by Lemma \ref{lem:ApproxSEllWithExpectation} the right side of \eqref{eq:DriftTermUpperBoundKthOrderContactPre} converges to zero. The convergence is uniform in $t\leq \bar t$ by monotonicity in $t$.
We can similarly analyse $A_n^{\sss (2)}(k)$ and $C_n^{\sss (2)}(k)$. Equation \eqref{eq:DriftErrorTermGeneralArrivals} then becomes
\begin{align}%
C_n^{\sss (2)}(s) &= \sum_{k=1}^s \frac{1}{1-F_{\sss T}(\Sigma_{k-1})}\times\\
&\times\mathbb E \big[ {\red \vert\nu_{k}\vert} f_{\sss T}(0)D_k  - f_{\sss T}(0)\Sigma_{k-1} + \frac{f^{(\ell)}(0)}{\ell!}\vert\nu_{k}\vert(\Sigma_{k-1})^{\ell}D_k + \vert\nu_{k}\vert o_{\mathbb P}(n^{-1})+ o_{\mathbb P}(n^{-\alpha/2})\big\vert \mathcal{E}_{k-1}\big],\notag
\end{align}%
with the usual convention that $o_{\mathbb P}(n^{-1}) =:\vert F_{\sss T}(\Sigma_k)- F_{\sss T}(\Sigma_{k-1}) - f_{\sss T}(\Sigma_{k-1})D_k\vert$ and $o_{\mathbb P}(n^{-\alpha/2}) =:\vert F_{\sss T}(\Sigma_{k-1}) - f_{\sss T}(0)\Sigma_{k-1} \vert$. This gives a decomposition of the drift $C_n^{\sss (2)}(k)$ similar to \eqref{eq:ArrivalsErrorsGlobal} as
\begin{align}\label{eq: KthOrderContactDriftErrorDecomposition}%
&n^{-\alpha/2}C_n^{\sss (2)}(tn^{\alpha}) = n^{-\alpha/2}\sum_{k=1}^{tn^{\alpha}}\frac{1}{1-F_{\sss T}(\Sigma_{k-1})} \mathbb E \big[ {\red k} f_{\sss T}(0)D_k  -f_{\sss T}(0)\Sigma_{k-1}\big\vert \mathcal{E}_{k-1}\big]\\
&+ \sum_{k=1}^{tn^{\alpha}} \frac{n^{-\alpha/2}}{1-F_{\sss T}(\Sigma_{k-1})}\mathbb E \big[f_{\sss T}^{(\ell)}(0){\red k}(\Sigma_{k-1})^{\ell} D_k  + {\red k}o_{\mathbb P}(n^{-1})+ o_{\mathbb P}(n^{-\alpha/2})\big\vert \mathcal{E}_{k-1}\big]+\varepsilon_n,\notag
\end{align}%
where $\varepsilon_n$ groups all the terms containing $N_n(k)$. The first term is less straightforward now, because the first order terms in its expansion do not cancel out immediately. Consider (again ignoring the higher order terms in the expansion of $(1-F_{\sss T}(\Sigma_{k-1}))^{-1}$)
\begin{align}\label{eq:KthOrderContactDriftErrorFirstTerm}%
n^{-\alpha/2}\Big\vert\sum_{k=1}^{tn^{\alpha}}  \big( {\red k} f_{\sss T}(0)\mathbb E[D_1]  -f_{\sss T}(0)\Sigma_{k-1}\big)\Big\vert &= f_{\sss T}(0)n^{-1-\alpha/2}\Big\vert\sum_{k=1}^{tn^{\alpha}}   \big( \sum_{j=1}^{k-1}S_j - {\red k} \mathbb E[S] \big)\Big\vert\notag\\
&\leq  f_{\sss T}(0)n^{-1-\alpha/2}\sum_{k=1}^{tn^{\alpha}}\Big\vert \sum_{j=1}^{k-1}S_j - k \mathbb E[S]\Big\vert
\end{align}%
It is now sufficient to apply Lemma \ref{lem:ApproxSWithExpectation} with $\beta = 1-\alpha/2$ (note that $2\beta < \alpha$). Therefore, the first term in \eqref{eq: KthOrderContactDriftErrorDecomposition} tends to zero in probability, uniformly in $t\leq \bar t$. The remaining terms are treated, without additional complications, in a similar manner as in Section \ref{sec:main_theorem_exponential}.

\noindent
\textbf{Proof of (ii).}
This was proved through an analysis of the leading-order term, in which the $\ell$-th derivative of the density played no role.

\noindent
\textbf{Proof of (iii).}
The proof of (iii) relies heavily on Lemma \ref{lem:QueueUnifConvGeneralArrivals}  which, in turn, relies on the analysis of the order statistics done in Lemma \ref{LemmaSecondMomentAPrimeGeneralArrivals}. Since the latter does not depend on the derivatives of the density (but rather on its continuity), the proof carries over.

\noindent
\textbf{Proof of (iv).}
Again, the proof relies on the analysis carried out in Lemma \ref{LemmaSecondMomentAPrimeGeneralArrivals}.

Having proved conditions (i)-(iv), this concludes the proof of Theorem \ref{th:MainTheoremKthOrderContact}.
\qed

\section{Proof of Theorem \ref{MainTheorem_delta_G_1}}\label{sec:proof_main_theorem_delta}
{\gr The process   $Q_n(\cdot)$ approximates the real queueing process $Q_n^{\Delta}(\cdot)$, because $Q_n(\cdot)$ never idles and is embedded at service completions. Therefore, establishing Theorem \ref{MainTheorem_delta_G_1} from Theorem \ref{MainTheorem_exponential} requires to prove that these differences between $Q_n(\cdot)$ and $Q_n^{\Delta}(\cdot)$ are asymptotically negligible as $n\rightarrow\infty$. We assume without loss of generality that the exponential random variables $(T_i)_{i=1}^n$ have unit mean (equivalently $\E[S]=1$).}
\subsection{Idle times}
Recall that, whenever the queue $Q_n(\cdot)$ is empty at the end of a service (say, the $k$-th service), the customer with the smallest arrival time is drawn from the pool and placed into service. Denote this customer by $c(k)$.
Since the minimum of $n$ rate one exponential random variables is again an exponential random variable with rate $n$, conditioned on $\vert \nu_k\vert$, $I_k:=T_{c(k)}$ is distributed as
\begin{equation}\label{eq:IdlePeriodDistr}%
I_k \sr{\mathrm d}{=} \frac{\Exp(1)}{n - \vert \nu_k \vert},
\end{equation}%
where $\Exp(1)$ is an exponential random variable with rate one.
The random variable $I_k$ represents \emph{the time the server would have idled if customer $c(k)$ was not immediately placed into service.} Alternatively, $I_k$ is the idle period after the $k$-th service in the coupled $\Delta_{(i)}/G/1$ queue. Thus we name $I_k$ a \emph{virtual} idle time.
In particular $\vert \nu_k \vert = O(n^{2/3})$ for $k= O(n^{2/3})$, and therefore its contribution in \eqref{eq:IdlePeriodDistr} is negligible and we can think of the virtual idle periods up to time $k= O(n^{2/3})$ as being independent exponential random variables with rate $n$. 

Let $\beta_n(k)$ be the number of customer who have been taken from the population for immediate service before the completion of the $(k+1)$-th service. {\bl Equivalently, $\beta_n(k)$ is the number of idle periods in the coupled $\Delta_{(i)}/G/1$ queue before the completion of the $(k+1)$-th service.} The \emph{cumulative} virtual idle time up to the $k$-th service completion takes the form
\begin{equation}\label{eq:TimeShift}%
\mathcal I(k) = \sum_{i=1}^{\beta_n(k)}I_i.
\end{equation}%
For exponential arrivals, an explicit expression for the virtual idle periods is available in \eqref{eq:IdlePeriodDistr}, thus we can estimate the average cumulative virtual idle time at step $k = O(n^{2/3})$ to be
\begin{equation}\label{eq:IdleTimesApprox}%
\sum_{i=1}^{\beta_n(k)}I_i \approx \beta_n(k)\cdot\frac{1}{n}
\end{equation}%
{\gr We now aim at making \eqref{eq:IdleTimesApprox} rigorous by proving that $\mathcal I(k)$ is asymptotically negligible, uniformly in $k=O(n^{2/3})$. First we show that $\beta_n(k) = \OP(n^{1/3})$ and to this end we prove the following representation:
}
\begin{lemma} \label{lem:representation_number_busy_periods} Under the same assumptions as in \emph{Theorem \ref{MainTheorem}},
\begin{equation}\label{eq:representation_number_busy_periods}%
\beta_n(k) = -\inf_{j\leq k}(N_n(j)\wedge 0),\qquad k\geq0,~\textup{a.s.}
\end{equation}%
\end{lemma}
\begin{proof}
{\red 
\eqref{eq:representation_number_busy_periods} holds for $k = 0$ because in this case both $\beta_n(0) = 0$ and $N_n(0) = 0$. Assume \eqref{eq:representation_number_busy_periods} holds for $k\geq1$. Without loss of generality we can also assume that 
\begin{equation}%
\beta_n(k) = -\inf_{j\leq k}(N_n(j)\wedge 0) = -N_n(k).
\end{equation}%
Let $\bar k$ be the minimum index such that $\bar k > k$, $N_n(\bar k -1) = N_n(k)$ and $A_n(\bar k ) = 0$. Equivalently, at the end of the $\bar k$-th service time there are no customers in the queue (and it is the first time after the $k$-th service that this happens). By the definition of $\beta_n(\cdot)$, we have that $\beta_n(\bar k - 1) = \beta _n(k)$ and $\beta_n(\bar k) = \beta_n(k) +1$. On the other hand we have that $N_n(\bar k)= N_n(\bar k -1) + A_n(\bar k) -1 = N_n(k) -1$. This gives
\begin{equation}%
\beta_n(\bar k) = \beta_n(k) + 1 = - N_n(k) + 1 = - N_n(\bar k) = - \inf_{j\leq \bar k}( N_n(j) \wedge 0).
\end{equation}%
Moreover, for every $q$ such that $k < q < \bar k$,
\begin{equation}%
-\inf_{j\leq q}(N_n(j)\wedge0) = -N_n(k),
\end{equation}%
and $\beta_n(q) = \beta_n(k)$, by definition of $\bar k$.
}
\end{proof}
The following lemma shows that $n^{-1/3}\beta_n(tn^{2/3})$ converges in distribution to a non-trivial random variable, hence \eqref{eq:IdleTimesApprox} is negligible in the limit when $k =O(n^{2/3})$. Recall that $\beta_n(k)$ denotes the number of {\red customers that have been removed from the population and directly put into service before the end of the $(k+1)$-st service}.

\begin{lemma}[Convergence of the number of idle periods]\label{cor:NumbBusyPeriodsConv_first}
Fix $t\in(0,\infty)$. Under the same assumptions as in \emph{Theorem \ref{MainTheorem}},
\begin{equation}\label{eq:BusyPeriodsConvergence_first}%
n^{-1/3}\beta_n(tn^{2/3}) \stackrel{\mathrm d}{\rightarrow} -\inf_{s\leq t} (W(s) \wedge 0),
\end{equation}%
where $W(\cdot)$ is given in \eqref{eq:main_theorem_diffusion_definition}.
\end{lemma}

\begin{proof}
The operator $\psi: f\mapsto \psi(f)(t) = -\inf_{s\leq t} (f(s) \wedge 0)$ acting from $\mathcal D$ to itself is  Lipschitz continuous with respect to the Skorohod topology by \cite[Theorem $6.1$]{whitt1980some}. Note that
\begin{equation}%
n^{-1/3}\beta_n(tn^{2/3}) = \psi (n^{-1/3}N_n( \cdot n^{2/3}))(t).
\end{equation}%
Then, since $n^{-1/3}N_n( \cdot n^{2/3})\stackrel{\mathrm d}{\rightarrow} W$, by the Continuous Mapping Theorem the image of $n^{-1/3}N_n(\cdot n^{2/3})$ through $\psi$ converge,  i.e.,
\begin{equation}
\psi (n^{-1/3}N_n(\ \cdot\ n^{2/3}))\stackrel{\mathrm d}{\rightarrow} \psi (W),
\end{equation}%
and this is \eqref{eq:BusyPeriodsConvergence_first}.
\end{proof}

As a direct consequence, the following lemma shows that the cumulative virtual idle time is asymptotically negligible:
\begin{lemma}[Convergence of the cumulative idle time]\label{lem:CumulativeIdleTimesConvergence}%
Assume that the arrival clocks $(T_i)_{i=1}^n$ are exponentially distributed. Fix $t\in(0,\infty)$. Then, under the same assumptions as in Theorem \ref{MainTheorem_exponential} and conditioned on $\{Q_n(s), s\in(0,tn^{2/3})\}$,
\begin{equation}%
\frac{n}{\beta_n(tn^{2/3})}\sum_{i=1}^{\beta_n(tn^{2/3})}I_i\stackrel{\mathbb P}{\rightarrow} 1.
\end{equation}%
\end{lemma}%
\begin{proof}
As was noted in \eqref{eq:IdlePeriodDistr}, $I_i$ is distributed as an exponential random variable with rate $n-\vert \nu_{k_i}\vert$, where $k_i$ is the time step corresponding to the $i$-th customer being placed directly into service. We rewrite the sum as
\begin{align}\label{eq:CumulativeIdleTimesConvergence}%
\frac{n}{\beta_n(tn^{2/3})}\sum_{i=1}^{\beta_n(tn^{2/3})}\frac{E_i}{n-\vert\nu_{k_i}\vert} &= \frac{1}{\beta_n(tn^{2/3})}\sum_{i=1}^{\beta_n(tn^{2/3})}\frac{E_i}{1-\frac{\vert\nu_{k_i}\vert}{n}}\notag\\
&= \frac{1}{\beta_n(tn^{2/3})}\sum_{i=1}^{\beta_n(tn^{2/3})}E_i + \frac{1}{\beta_n(tn^{2/3})}\sum_{i=1}^{\beta_n(tn^{2/3})}E_i\frac{\vert\nu_{k_i}\vert}{n}+\varepsilon_n,
\end{align}%
where $\varepsilon_n = o_{\mathbb P}(1)$ and $(E_i)_{i\geq1}$ are i.i.d. exponential random variables with rate $1$.
By Lemma \ref{cor:NumbBusyPeriodsConv_first}, $\beta_n(tn^{2/3})\geq M n^{\alpha}$ w.h.p. for a fixed $M>0$ and $\alpha\in(0,1/3)$. By the LLN, the first term in \eqref{eq:CumulativeIdleTimesConvergence} converges in probability to $1$, and by Lemma  \ref{oQueueLemmaUnif} the second term (and consequently the error term) converges to zero.
\end{proof}
Lemma \ref{lem:CumulativeIdleTimesConvergence} intuitively says that the total virtual idle time up to time $tn^{2/3}$ is of the same order of magnitude of the number of virtual idle periods up to time $tn^{2/3}$ times the average interarrival time. In particular, as we prove below, $\mathcal I (tn^{2/3}) = o_{\mathbb P}(1)$, that is, the cumulative virtual idle time up to times of the order $n^{2/3}$ is negligible:
\begin{corollary}[Cumulative idle time is negligible]\label{cor:cumul_idle_time_is_negligible}%
Fix $T>0$. Then,
\begin{equation}%
\sup_{t\leq T}\mathcal I(tn^{2/3})  \sr{\mathbb P}{\rightarrow} 0.
\end{equation}%
\end{corollary}%
\begin{proof}%
By monotonicity, $\sup_{t\leq T}\mathcal I(tn^{2/3}) = \mathcal I(T)$. Fix $\varepsilon >0$. Then, by Lemma \ref{lem:CumulativeIdleTimesConvergence}
\begin{equation}%
\mathcal I(Tn^{2/3}) \leq (1+\varepsilon) \beta_n(tn^{2/3})/n
\end{equation}%
with high probability. By Lemma \ref{cor:NumbBusyPeriodsConv_first}, $\beta_n(tn^{2/3})/n\sr{\mathbb P}{\rightarrow}0$, so that 
\begin{equation}%
\mathcal I(Tn^{2/3}) \sr{\mathbb P}{\rightarrow}0,
\end{equation}%
as desired.
\end{proof}%
%
%

\subsection{Proof of Theorem \ref{MainTheorem_delta_G_1}}

We start by defining the time change $\varphi_n(\cdot)$. Fix a realization of the arrival and service times $(S_i)_{i=1}^n = (S_i(\omega))_{i=1}^n$ and $(T_i)_{i=1}^n = (T_i(\omega))_{i=1}^n$, where $\omega\in\Omega$, the sample space. We assume throughout that $\E[S]=1$. We define $\varphi_n(t)$ piece-wise, depending whether a customer is in service at time $t$ or the server is idling at time $t$. 

Because the time scaling in $ Q^{\Delta}_n(\cdot n^{-1/3})$ is $n^{-1/3}$, for the two processes $ Q^{\Delta}_n(\cdot n^{-1/3})$ and $Q_n(\varphi_n(\cdot)n^{2/3})$ to be on a comparable time scale, the time change $\varphi_n(\cdot)$ should be such that 
\begin{equation}\label{eq:time_change_expression}%
\varphi_n(\cdot) = n^{-2/3}\phi_n(\cdot n^{-1/3}),
\end{equation}%
for a suitable $\phi_n(\cdot):\mathbb R^+\mapsto\mathbb R^+$. We now provide a precise expression of $\phi_n(\cdot)$. We remark that many other choices for $\phi_n(\cdot)$ would work for proving Lemma \ref{lem:claim_1} and Lemma \ref{lem:claim_2}, and the one we present here is only the simplest one.

To increase readability, define
\begin{equation}%
\vartheta_{k_2}(k_1) := \frac{\sum_{i=1}^{k_1}S_i}{n} + \mathcal I (k_2).
\end{equation}%
First assume that at time $t$ a service (say, the $k$-th service) is taking place, then
\begin{equation}%
\phi_n(t)=(k-1)+\frac{1}{S_k/n}(t - \nu_k(k-1))  \qquad \textrm{for}~t\in \Big[\vartheta_k(k-1),\vartheta_k(k)\Big]. 
\end{equation}%
Note that, since the queue is serving at time $t$, $\mathcal I(k-1) = \mathcal I(k)$. In other words $\phi_n(t)$ is the line joining the points $(\vartheta_k(k-1), k-1)$ and $(\vartheta_k(k), k)$, where the term $\mathcal I (k)$ in $\vartheta_k(k)$ takes into account  previous idle periods which might have occurred before $t$.

Assume now that an idle period (say, the $(\beta(k)+1)$-th idle period) is under way at time $t$. Then $\phi_n(t)$ takes the form 
\begin{equation}%
\phi_n(t)= k  \qquad\textrm{for}~t\in\big[\vartheta_k(k), \vartheta_k(k) + I_{\beta(k)+1}\big].
\end{equation}%
In other words, $\phi_n(t)$ is constant during an idle period. This is because $\phi_n(t)$ (being the input of $Q_n(\cdot)$) represents the \emph{number of completed services}. Again, here $\sum^k_{i=1} S_i/n$ takes previous services that have occurred before $t$ into account.
Summarizing, $\phi_n(\cdot)$ takes the form 
\begin{equation}%
\phi_n(t)= \left\{\begin{array}{ll}
\Big(k-1+\frac{n}{S_{k}}(t - \vartheta_k(k-1))\Big) & \textrm{for}~t\in \big[\vartheta_k(k-1),\vartheta_k(k)\big]~\text{for some~}k;\\
k & \textrm{for}~t\in\big[\vartheta_k(k),\vartheta_k(k) + I_{\beta(k)+1}\big]~\text{for some~}k.
\end{array}\right.
\end{equation}%
In particular $\phi_n(\cdot)$ (and therefore also $\varphi_n(\cdot)$) is a piecewise linear continuous function. See Figure \ref{fig:time_change} for one possible sample path of $\phi_n(\cdot)$.
\begin{figure}
\centering
\includegraphics{
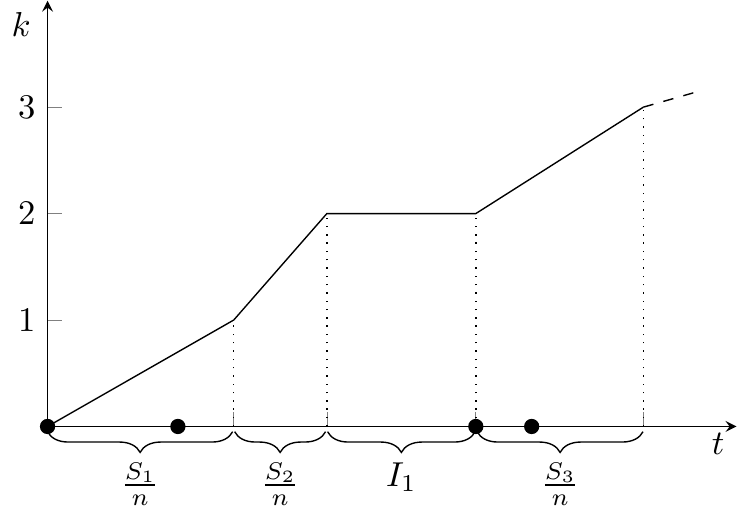}
\caption{An example of the time change $\phi_n(\cdot)$. The $\bullet$'s indicate arrival times of customers.}\label{fig:time_change}
\end{figure}
We now focus on $\varphi_n(\cdot)$ in \eqref{eq:time_change_expression} and show that it converges to the identity function. 
The time change $\varphi_n(t)$ takes the form
\begin{equation}%
\varphi_n(t)= \left\{\begin{array}{ll}
n^{-2/3}\Big(k-1+\frac{n}{S_{k}}(\frac{t}{n^{1/3}} - \vartheta_k(k-1))\Big) & \textrm{for}~(\frac{t}{n^{1/3}})\in \big[\vartheta_k(k-1),\vartheta_k(k)\big];\\
n^{-2/3}k & \textrm{for}~(\frac{t}{n^{1/3}})\in\big[\vartheta_k(k),\vartheta_k(k) + I_{\beta(k)+1}\big].
\end{array}\right.
\end{equation}%
Note that the only values of $k$ for which $\phi_n(t)$ has a meaningful limit as $n\rightarrow\infty$ are $k=O(n^{2/3})$. This observation is consistent with the fact that the original scaling for $Q_n(\cdot)$ is $Q_n(\cdot n^{2/3})$, that is, in order to obtain a meaningful limit we observe the queue length at times when $O(n^{2/3})$ services have been completed. As a consequence, we assume that $k=sn^{2/3}$ for some $s\in\mathbb R^+$. Summarizing, the final form of $\varphi_n(t)$ is
\begin{equation}\label{eq:definition_time_change_final}%
\varphi_n(t)= \left\{\begin{array}{ll}
s-\frac{1}{n^{2/3}}+\frac{1}{S_{sn^{2/3}}}(t - \vartheta_{sn^{2/3}}(sn^{2/3}-1)) & \textrm{for}~t\in n^{1/3}\big[\vartheta_{sn^{2/3}}(sn^{2/3}-1),\vartheta_{sn^{2/3}}(sn^{2/3})\big];\\
s & \textrm{for}~t\in n^{1/3}\big[\vartheta_{sn^{2/3}}(sn^{\frac{2}{3}}),\vartheta_{sn^{ 2/3}}(sn^{2/3})+I_{\beta (sn^{2/3})}\big].
\end{array}\right.
\end{equation}%
We now turn to proving the first claim, as formulated in Lemma \ref{lem:claim_1}.

\begin{proof}[Proof of Lemma \ref{lem:claim_1}]%
First note that
\begin{equation}%
\sup_{t\leq T}\Big\vert -\frac{1}{n^{2/3}}+\frac{1}{S_{sn^{2/3}}}(t - \vartheta_{sn^{2/3}}(sn^{2/3}-1))\Big\vert\leq \frac{2}{n^{2/3}}\rightarrow0,\qquad n\rightarrow\infty,
\end{equation}%
implying that we can treat $\varphi_n(t)$ as piece-wise constant, $\varphi_n(t) \equiv s$ on intervals of the form $[l(s), u(s)]$. We now prove \eqref{eq:claim_time_change_converges_identity}. Since the function $t\mapsto t-\varphi_n(t)= t - s$ (for some fixed $s$) defined on an interval $[l,u]= [l(s),u(s)]$ is linear in $t$, it obtains its maximum  either in $l$ or $u$. This implies
\begin{align}\label{eq:time_change_converges_identity_middle}%
\lim_{n\rightarrow\infty}\|t-\varphi_n(t)\|_T \leq \lim_{n\rightarrow\infty}\Big(\sup_{s\leq S}\Big\vert \vartheta_{sn^{2/3}}(sn^{2/3}-1) - s\Big\vert \vee  \sup_{s\leq S}&\Big\vert  \vartheta_{sn^{ 2/3}}(sn^{2/3})+I_{\beta (sn^{2/3})} - s\Big\vert  \nnl
& \vee\sup_{s\leq S}\Big\vert \vartheta_{sn^{2/3}}(sn^{2/3}) - s\Big\vert\Big),
\end{align}%
where $x\vee y:=\max\{x,y\}$ and  $S=S(T)= \frac{T}{\E[S]}(1+\varepsilon)>0$ for some $\varepsilon>0$. The inequality comes from the fact that the three supremums are taken over a larger set, by the definition of $S(T)$. We prove convergence to zero of one of the three terms on the right in \eqref{eq:time_change_converges_identity_middle}, the others being analogous. Applying the triangle inequality to the last term yields
\begin{equation}%
\sup_{s\leq S}\Big\vert \frac{\sum_{i=1}^{sn^{2/3}}S_i}{n^{2/3}} + n^{1/3}\mathcal I (sn^{2/3}) - s\Big\vert \leq \sup_{s\leq S}\Big\vert   \frac{\sum_{i=1}^{sn^{2/3}}S_i}{n^{2/3}} -s\Big\vert + \sup_{s\leq S}\Big\vert n^{1/3}\mathcal I (sn^{2/3}) \Big\vert.
\end{equation}%
The first term converges to zero in probability by the functional Law of Large Numbers. The second term converges to zero in probability by Corollary \ref{cor:cumul_idle_time_is_negligible}. Indeed, the proof of Lemma \ref{lem:CumulativeIdleTimesConvergence} and Corollary \ref{cor:cumul_idle_time_is_negligible} show that $\mathcal I (sn^{2/3}) = \OP(n^{-2/3})$. This concludes the proof of \eqref{eq:claim_time_change_converges_identity}.

We now turn to proving \eqref{eq:claim_distance_between_queue_and_time_scaled_queue_converges_to_zero}.
Note that $\varphi_n(\cdot)\sr{\mathbb P}{\rightarrow}\mathrm{id}$, the identity map on $[0,T]$. In particular, $\mathrm{id}$ is deterministic, so that by \cite[Theorem 11.4.5]{StochasticProcess},
\begin{equation}\label{eq:joint_convergence_queue_and_time_change}%
(n^{-1/3}Q_n(\cdot n^{2/3}), \varphi_n) \sr{\mathrm{d}}{\rightarrow} (W, \mathrm{id}).
\end{equation}%
This allows to prove claim (i), as follows. 
By Skorokhod's representation theorem and \eqref{eq:joint_convergence_queue_and_time_change} there exist $\bar Q_n(\cdot)$, $ \bar\varphi_n(\cdot)$ and $\bar W(\cdot)$ defined on the same probability space $\bar \Omega$ such that 
\begin{equation}
(\bar Q_n(\cdot n^{2/3}), \bar\varphi_n)\sr{\mathrm d}{=}( Q_n(\cdot n^{2/3}), \varphi_n)
\end{equation}
and 
\begin{equation}%
(n^{-1/3}\bar Q_n(\cdot n^{2/3}), \bar\varphi_n)\sr{\mathrm{a.s.}}{\rightarrow} (\bar W(\cdot), \mathrm{id}).
\end{equation}%
We now dominate \eqref{eq:claim_distance_between_queue_and_time_scaled_queue_converges_to_zero}
using the random variables provided by the representation theorem, as follows:
\begin{align}\label{eq:middle_distance_between_queue_and_time_scaled_queue_converges_to_zero}%
n^{-1/3}\| \bar Q_n(\cdot n^{2/3}) - \bar Q_n&(\varphi_n(\cdot) n^{2/3})\|_T \nnl
&\leq n^{-1/3}\big( \| \bar Q_n(\cdot n^{2/3}) - \bar W(\cdot) \|_T + \| \bar W(\cdot) - \bar Q_n(\varphi_n(\cdot) n^{2/3})\|_T\big).
\end{align}%
Since the limiting process $\bar W(\cdot)$ is almost surely continuous, by standard arguments both the convergence $n^{-1/3}\bar Q_n(\cdot n^{2/3})\sr{\mathrm{a.s.}}{\rightarrow}\bar W(\cdot)$ and $n^{-1/3}\bar Q_n(\varphi_n(\cdot) n^{2/3})\sr{\mathrm{a.s.}}{\rightarrow}\bar W(\cdot)$ hold with respect to the uniform topology. In particular, both terms on the right in \eqref{eq:middle_distance_between_queue_and_time_scaled_queue_converges_to_zero}
converge to zero in probability. Moreover, since $(\bar Q(\cdot n^{2/3}), \bar \varphi_n(\cdot))\sr{\mathrm{d}}{=}( Q_n(\cdot n^{2/3}),  \varphi_n(\cdot))$,
\begin{equation}%
n^{-1/3}\| \bar Q_n(\cdot n^{2/3}) - \bar Q_n(\varphi_n(\cdot) n^{2/3})\|_T \sr{\mathrm{d}}{=}n^{-1/3}\|  Q_n(\cdot n^{2/3}) -  Q_n(\varphi_n(\cdot) n^{2/3})\|_T,
\end{equation}%
so that
\begin{equation}%
n^{-1/3}\|  Q_n(\cdot n^{2/3}) -  Q_n(\varphi_n(\cdot) n^{2/3})\|_T \sr{\mathbb P}{\rightarrow} 0,
\end{equation}%
as desired. This concludes the proof of claim (i) and thus of Lemma \ref{lem:claim_1}.
\end{proof}
The fact that during one service the number of arrivals is asymptotically small is crucial in proving that $Q_n(\cdot)$ and $Q_n^{\Delta}(\cdot)$ are close in the supremum norm. This is the subject of Lemma \ref{lem:claim_2}.

\begin{proof}[Proof of Lemma \ref{lem:claim_2}]%
Note that if $t/n^{1/3} = \sum_{i=1}^{k}S_i/n +\mathcal I (k)$ for $k\in\mathbb N$ (i.e. if a service is completed in $t$, or an idle period is undergoing in $t$), then 
\begin{equation}%
Q_n(\phi_n(t) ) = Q_n^{\Delta}(t).
\end{equation}%
This follows from the definition of the time change $\phi_n(t)$, as discussed above. This implies that during any service time $Q_n(\phi_n(t))$ and $Q_n^{\Delta}(t)$ differ by the number of arrivals that have occurred \emph{during that} service time. Moreover, during any idle time $Q_n(\phi_n(t))$ and $Q_n^{\Delta}(t)$ are both equal to zero. These considerations imply
\begin{equation}%
\sup_{t\leq T}\big\vert Q_n(\varphi_n(\cdot) n^{2/3}) - {Q}^{\Delta}_n(\cdot n^{-1/3})\big\vert = \max_{k\leq Tn^{2/3}}A_n(k).
\end{equation}%
Hence \eqref{eq:number_arrivals_o_small_n_one_third} implies \eqref{eq:claim_two_in_lemma}.
Let now $\varepsilon>0$ be arbitrary. Then
\begin{equation}%
\mathbb P \big(n^{-1/3} \max_{k\leq T n^{2/3}} A_n(k)\geq\varepsilon\big)= \mathbb P\big(n^{-1/3} \max_{k\leq T n^{2/3}} A_n(k)\mathds 1_{\{ A_n(k)>\varepsilon n^{1/3}\}}\geq\varepsilon\big).
\end{equation}%
In other words, only the very large values of $A_n$ contribute to the probability being computed. By Chebyshev's inequality,
\begin{align}%
\mathbb P (n^{-1/3} \max_{k\leq T n^{2/3}} A_n(k)\geq\varepsilon) &\leq \frac{\E[\max_{k\leq T n^{2/3}} A^2_n(k)\mathds 1_{\{ A_n(k)>\varepsilon n^{1/3}\}}]}{(n^{1/3}\varepsilon)^2}\nnl
&\leq \sum_{k=1}^{Tn^{2/3}}\frac{\E[A^2_n(k)\mathds 1_{\{ A_n(k)>\varepsilon n^{1/3}\}}]}{n^{2/3}\varepsilon^2}.
\end{align}%
The almost sure domination $A_n(k) \sr{\mathrm{a.s.}}{\leq} A_n'(k):= \sum_{i=1}^n\mathds 1_{\{T_i \leq S_k/n\}}$ (note that this domination holds for all $k\leq T n ^{2/3}$ simultaneously) gives
\begin{align}\label{eq:bound_maximum_number_arrivals_final}%
\mathbb P (n^{-1/3} \max_{t\leq T} A_n(tn^{2/3})\geq\varepsilon) &\leq \sum_{k=1}^{Tn^{2/3}}\frac{\E[A_n'^2(k)\mathds 1_{\{ A'_n(k)>\varepsilon n^{1/3}\}}]}{n^{2/3}\varepsilon^2}\nnl
&= \frac{T}{\varepsilon^2} \E[A_n'^2(1)\mathds 1_{\{ A'_n(1)>\varepsilon n^{1/3}\}}].
\end{align}%
The right-most term in \eqref{eq:bound_maximum_number_arrivals_final} tends to zero because $A_n'^2$ is stochastically dominated by a uniformly integrable random variable, as was proven in Lemma \ref{LemmaSecondMomentAPrime}.
\end{proof}

\section{Extended discussion}\label{sec:ExtendedDiscussion}

\subsection{On other regimes of criticality}\label{sec:OtherRegimesOfCriticality}
Throughout this paper, we have focused on the critical regime, when $\nu:=\sup_{t\geq0}f_{\sss T}(t)$ is such that $\nu = \mu$, where $\mu=1/\E[S]$. This is by far the least studied case in the literature, but other cases may be of interest too. In this section, we discuss some known results also in relation to our own results. We denote by $t_{\mathrm{max}}$ the argmax of $f_{\sss T}(t)$. 

The case $\nu > \mu$ has been studied extensively in \Ward. When $\nu > \mu$, the queue length grows linearly in a neighborhood of $t_{\mathrm{max}}$ as has been discussed in Section \ref{sec:comparison_known_results}. In particular, it is possible to prove that there exists a constant $m>0$ such that the maximum of the fluid-scaled queue length process converges to $m$, meaning that the queue during `peak hour' grows linearly.
%

When $\nu = \mu$, $t_{\mathrm{max}}$ can be either zero, or greater than zero. If it is zero, and if $f_{\sss T}'(0) < 0$, then our results apply, and thus the correct scaling is $n^{-1/3}$ for the space, $n^{-1/3}$ for the time around $t_{\mathrm{max}}$. If the first derivative (or possibly all the derivatives up to a finite order) of $f_{\sss T}$ is zero in zero, by appropriately modifying the scalings our approach still yields the correct limit, as was shown in Section \ref{sec:HigherOrderContact}. Here `correct' means that both a deterministic drift indicating a depletion-of-points effect and a Brownian contribution appear. If $t_{\mathrm{max}}$ is greater than zero, and $f_{\sss T}(t)$ is sufficiently smooth, then we must have that $f_{\sss T}'(t_{\mathrm{max}}) = 0$, suggesting that one has to investigate higher order derivatives as was done in Section \ref{sec:HigherOrderContact} (note that the first non-zero derivative will necessarily be even, provided $f'_{\sss T}(t)$ is smooth). If $f_{\sss T}(t)$ is not smooth, i.e., $f_{\sss T}'(t)$ is discontinuous in $t_{\mathrm{max}}$, an even wider range of behaviors will possibly be displayed, due to the fact that $f_{\sss T}(t)-\mu $ might have different contact orders in $t_{\mathrm{max}}^-$ and $t_{\mathrm{max}}^+$.

The uniform arrivals case (uniform in $[0,1]$ say) is  special, since all the derivatives of $f_{\sss T}(t)$ are identically zero. In some sense, it interpolates between our result and \Ward. The queue is  critical during its entire lifetime, and thus it never grows linearly, similarly to our model. On the other hand, the uniform arrivals model does not lead to a downward drift. Intuitively, this is because there is no `tipping point' after which the intensity of the arrival process diminishes and therefore the depletion-of-points effect does not occur. By observing the queue at a time scale {\bl  of order $O(1)$}, it displays diffusion-type fluctuations (of the order $n^{1/2}$), which end abruptly in $1$, when customers stop joining and thus the queue behaves deterministically, with the server flushing out the customers still in the queue. The correct (optimal) scaling in this case has been first identified in \cite{iglehart1970multipleII} (Example 3, p.~365), and can be treated using the general theory developed in \Ward. Indeed, \cite[Theorem 1]{honnappa2012delta}  implies that $Q_n(t)/n\stackrel{\text{d}}{\rightarrow} 0$, while \cite[Theorem 2]{honnappa2012delta} implies that $Q_n(t)/\sqrt{n}$ converges to $ X^*(t)$, the reflected version of 
\begin{align*}%
X(t) = \left\{\begin{array}{ll}
B^0(t)-\sigma B(t),& t\in[0,1], \\
0,& t>1,
\end{array}\right.
\end{align*}%
where $B^0(t)$ is a standard Brownian Bridge.

The case $\nu < \mu$ is, from the scaling perspective, the least interesting one. The service rate is so large, compared to the instantaneous arrival rate $f_{\sss T}(t)$, that the queue is empty most of the time and in no instant it builds up as a power of $n$. One is then tempted to conjecture that, for fixed $t$, the queue length $Q_n(t)$ converges in distribution to a proper random variable $Q(t)$. Moreover, for $s<t$, $Q_n(\cdot)$ will be zero w.h.p. at some time $\zeta \in (s,t)$, suggesting that $Q(s)$ is close to independent from $Q(t)$. Since locally the arrival process is roughly a Poisson process with rate $f_{\sss T}(t)$, it is reasonable to also conjecture that $Q(t)$ is the stationary distribution of an $M/G/1$ queue with arrival rate $f_{\sss T}(t)$. In fact, this has already been proven for the $M_t/M_t/1$ queue in \cite{mandelbaum1995strong}  and, using analogous arguments, it is also possible to prove it for the $\Delta_{(i)}/G/1$ model with exponential arrivals.

We conclude by another connection between the present work and \cite{honnappa2012delta}. In \cite{honnappa2012delta}, the queue length process $Q_n$ is such that $Q_n = \phi ( X_n)$, where
\begin{equation}%
 X_n (t) = \left(A_n(t)-nF_{\sss T}(t)\right)-\left(S^n(B_n(t))-\mu nB^n(t)\right) + n(F_{\sss T}(t) - \mu t).
\end{equation}%
Here $A_n(t) := \sum_{i = 1}^n \mathds 1_{\{T_i\leq t\}}$, $F_{\sss T}$ is the distribution function of $T$, $S^n$ is the  renewal process associated with the (rescaled) service times $S_1/n,S_2/n,\ldots$ and $B_n:= \int_0^t \mathds 1_{\{ Q_n(s) > 0\}}\mathrm{d}s$ is the busy time process. Then, for example, \cite[Theorem 1]{honnappa2012delta} is proved by applying the \emph{continuous-mapping} approach to the fluid-scaled process $X_n/n$. Let us now scale the process $X_n$ as 
\begin{align}\label{eq:WardRescaledOurWay}%
n^{-1/3}X^n(tn^{-1/3}) &= n^{2/3}\Big( \frac{A_n(tn^{-1/3})}{n}- F_{\sss T}(tn^{-1/3}) \Big) - n^{2/3} \Big(  \frac{S^n(B_n(tn^{-1/3}))}{n}-\mu B_n (tn^{-1/3})\Big)\notag\\ 
&\quad+ n^{2/3}(F_{\sss T}(tn^{-1/3}) - \mu tn^{-1/3}).
\end{align}%
It is immediate that, if $F_{\sss T}(x) \approx f_{\sss T}(0) x + \frac{f'_{\sss T}(0)}{2} x^2$, the third term in \eqref{eq:WardRescaledOurWay} yields the quadratic drift. Moreover, it seems plausible that the first and second term converge to (independent) Brownian motions, respectively through a martingale FCLT (or generalized Donsker's Theorem) and through a FCLT for renewal processes. Since the sum of two independent Brownian motions is a Brownian motion itself, the limit of $n^{-1/3}X^n(tn^{-1/3})$ coincides with our result.

\subsection{On the relationship between the queue and random graphs}
\label{sec-queues-graph}
The queueing system in this paper can be used to generate a directed random graph by saying that customer $i$ is connected to customer $j$ when $i$ arrives during the service period of $j$. Assume the customers have an exponential clock. Conditionally on customer $i$ not having joined the queue yet and conditionally on the service time $D_j$ of the presently served customer $j$,  customer $i$ will be connected to customer $j$ with probability 
	\eqn{
	\label{pij-def}
	p_{ij}=1-\e^{-D_j}=1-\e^{-S_j/n}.
	}
In particular, when the service times are deterministic, i.e., $S_j\equiv \lambda$, these probabilities are precisely equal to the edge probabilities of the Erd\H{o}s-R\'enyi random graph (see \cite{Boll01,RandomGraphs,JanLucRuc00} for monographs on random graphs) with parameter $1-\e^{-\lambda/n}\approx \lambda/n$. In this case, Aldous \cite{Aldo97} has shown that the scaling limit in Theorem \ref{MainTheorem} appears precisely when $\lambda=1$, and with the constant $f'_{\sss T}(0)/f_{\sss T}(0)^2$ replaced by $-1$. This follows, intuitively, by noting that the depletion-of-points effects in the two problems is the same. {\bl  Note however that} the process described by \eqref{TrueQueueLengthIntro} is not the graph construction itself, but is instead equivalent to the exploration of the connected components of the graph. Indeed, edges in the graph that form cycles are not explored {\bl (i.e. a customer in the queue is not allowed to rejoin the queue)}, and thus also not observed, so that we obtain a tree rather than a graph.

When $S_j$ is not a.s.\ constant, then the edge probabilities in \eqref{pij-def} are not symmetric. We can think of this as the cluster exploration of a {\em directed} random graph, where, due to the depletion-of-points effect, at most one of the two possible edges between two vertices will ever be found {\bl (i.e. for each pair of customers $i$ and $j$, either $i$ joins the queue while $j$ is being served, or vice versa, or neither of the two)}. The service time random variables $(S_j)_{j=1}^n$ create {\em inhomogeneity} in the graph, of a form that has not been investigated in the random graph community. 

If, instead of our assumptions, we assume that the clock of customer $i$ has a rate proportional to $D_i$ (i.e., when the job size is larger, then it will join the queue more quickly, inversely proportional to the job size), then the edge probabilities become
	\eqn{
	\label{pij-def-als}
	p_{ij}=1-\e^{-\lambda D_iD_j}=1-\e^{-\lambda S_i S_j/n^2}.
	}
Taking $\lambda=an$, this becomes $p_{ij}=1-\e^{-a S_i S_j/n}$. This is closely related to the so-called {\em generalized random graph} with i.i.d.\ vertex weights, as studied in \cite{BhaHofLee09b,bhamidi2010scaling,Hofs09a,Jose10}. See also \cite{BolJanRio07} for the most general setting, and \cite{BriDeiMar-Lof05,ChuLu06c, NorRei06} for related models. Interestingly, a {\em size-biasing} takes place in the form that the vertices that are found by the exploration process have weights with the size-biased distribution $S_j^\star$ given by
	\eqn{
	\label{S-sb}
	\prob(S_j^\star\leq x)=\frac{\expec[S_j\indic{S_j\leq x}]}{\expec[S]}.
	}
{\bl  In the queueing setting, this means that, if $c(i)$ is the $i$-th customer joining the queue, the service time of $c(i)$, $S_{c(i)}$ is not distributed as $S$.} In \cite{BolJanRio07} it is shown that the cluster exploration obeys the scaling in  Theorem \ref{MainTheorem} precisely when $\expec[S^3]<\infty$, while different scalings can occur when $\expec[S^3]=\infty$ and $S$ has a power-law distribution. {\bl  Since the size-biasing is absent in our work, our assumption $\E[S^2]<\infty$ corresponds to the assumption $\E[S^3]<\infty$ in \cite{BolJanRio07}. This suggests that similar scaling limits as the ones in \cite{BhaHofLee09b,Jose10} (in which the exploration process for inhomogeneous random graphs with $\E[S^3]=\infty$ was shown to converge) might appear for queues with vanishing populations of customers when $S$ has a power-law distribution with $\expec[S^2]=\infty$.} This is an interesting problem for future research. 

\subsection{On the relationship between the queue and the Grenander estimator} \label{sec:RelationshipWithGrenander}
The arrival process in \eqref{APrimeDef} can be seen as ($n$ times) the empirical distribution function for $n$ trials of the random variable $T$. This seemingly trivial observation reveals a connection between our result Theorem \ref{MainTheorem} and some results in the statistics literature. Suppose we are given a random variable $T$ whose law is supported on $[0,1]$ and admits a positive density $f_{\sss T}(t)$ such that $\inf_{t\in[0,1]} \vert f_{\sss T}(t)\vert >0$ and $t\mapsto f_{\sss T}(t)$ is non-increasing. Given that the distribution of $T$ is sampled $n$ times as $T_1,\ldots, T_n$, our task is to estimate $f_{\sss T}$. In \cite{grenander1956theory} the so-called Grenander estimator is introduced as an estimator of $f_{\sss T}$. It is constructed as follows. Let $F^{\sss (n)}_{\sss T}(t):= (1/n)\sum_{i=1}^n \mathds 1_{\{T_i\leq t\}}$ be the empirical distribution function of the given sample and let $\hat F^{\sss (n)}_{\sss T}(t)$ be the concave conjugate of $F^{\sss (n)}_{\sss T}(t)$. Recall that the concave conjugate is the (pointwise) smallest concave function such that
\begin{equation}%
\hat F^{\sss (n)}_{\sss T}(t) \geq F_{\sss T}^{\sss (n)}(t),\qquad\forall t\in[0,1].
\end{equation}%
Then, the Grenander estimator $\hat f^{\sss (n)}_{\sss T}$ is defined as the left derivative of $\hat F^{\sss (n)}_{\sss T}(t)$. It turns out that studying the \emph{inverse process} $U_n(a) = \sup\{t\in[0,1]:~F_{\sss T}^{\sss (n)}(t) - at \textrm{~is~maximal}\}$ is equivalent to studying $\hat f^{\sss (n)}_{\sss T}(t)$, but considerably easier. In particular, in \cite{groeneboom1999asymptotic} the task of estimating the (rescaled) $L^1$ error of the estimator, $n^{1/3}\|\hat f^{\sss (n)}_{\sss T} - f_{\sss T}\|_{L^1([0,1])}$, is proven to be equivalent to studying $\int \vert V_n^E(a)\vert \mathrm {d}a$, with $V_n^E(a) := n^{1/3}(U_n(a) - g(a))$ and $g(a)$ the inverse of $f_{\sss T}$. It is readily seen that $V_n^E$ can be rewritten as 
\begin{equation}%
V_n^{\sss E}(a) = \sup\{t\in\mathcal T: D_n^{\sss E}(a,t) - n^{1/3} a t\textrm{~is~maximal}\},
\end{equation}%
for some appropriate set $\mathcal T = \mathcal T_n(a)$. Here $D_n^{\sss E}$ is the process
\begin{equation}\label{eq:GrenandierDDef}%
D_n^{\sss E}(a,t) := n^{1/6}\big(E_n(g(a) + n^{-1/3}t) - E_n(g(a))\big) + n^{2/3}\big(F_{\sss T}(g(a) + n^{-1/3}t) - F_{\sss T}(g(a))\big),
\end{equation}%
where $E_n (x) = \sqrt n(F^{(n)}_{\sss T}(x) - F_{\sss T}(x))$. The problem in \cite{groeneboom1999asymptotic} of computing the error of the Grenander estimator then reduces to studying a certain complicated functional of the process 
\begin{equation}%
t\mapsto D_n^{\sss E}(a,t) - n^{1/3} at.
\end{equation}%
Notice that this can be interpreted as the queue length process associated with a system in which the service rate is deterministic, $1/S\equiv a$ and the arrival process is given by $t\mapsto F^{(n)}_{\sss T}(g(a) + t) - F^{(n)}_{\sss T}(g(a))$. Moreover, note that the derivative of $F_{\sss T}$ in $g(a)$ is
\begin{equation}\label{eq:GrenanderCriticality}%
F_{\sss T}'(g(a)) = f_{\sss T}(g(a)) = a = 1/S,
\end{equation}%
since $g$ is the inverse of $f_{\sss T}$. Remarkably, \eqref{eq:GrenanderCriticality} as seen from our perspective is nothing but the criticality assumption in the point $g(a)$. In \cite[Theorem 3.2]{groeneboom1999asymptotic} the authors prove a convergence result for $V_n^{\sss E}$ of the type of Theorem \ref{MainTheorem}, the most important difference being that the functional of interest in our case is the reflection mapping, while they are interested in the (supremum among all of the) argmax. Their technique is similar to ours, in the sense that in order to prove convergence of the functional they first prove convergence of the underlying process. This is a two-sided Brownian motion in $\mathbb R$, $(B(t))_{t\in\mathbb R}$ originating from zero, i.e. $B(0) = 0$. Whether there is an equivalent, in this setting, of our $\ell$-th order contact is an interesting question, and one which remains open for future research.

\section*{Acknowledgments}

This work is supported by the NWO Gravitation {\sc Networks} grant 024.002.003. The work of RvdH is further supported by the NWO VICI grant 639.033.806. The work of JvL is further supported by an NWO TOP-GO grant and by an ERC Starting Grant.
\bibliographystyle{plain}
\bibliography{mybib}

\end{document}